\RequirePackage{fix-cm}
\documentclass[smallextended]{svjour3}   
\smartqed  % flush right qed marks, e.g. at end of proof
\usepackage[T1]{fontenc}
\usepackage{graphicx}
\usepackage{amsmath,amssymb,amstext,amsfonts}
\usepackage{graphicx}
\usepackage{epstopdf}
\usepackage{algorithm}
\usepackage{algorithmic}
\usepackage{type1cm}        % activate if the above 3 fonts are
\usepackage{makeidx}         % allows index generation
\usepackage{graphicx}        % standard LaTeX graphics tool
\usepackage{multicol}        % used for the two-column index
\usepackage[bottom]{footmisc}% places footnotes at page bottom

\makeindex             % used for the subject index
\usepackage{graphicx}
\usepackage{amsfonts,amsmath, amssymb, amstext,color}
\usepackage{graphicx} % more modern

\usepackage{braket}
\usepackage{mathrsfs}
\usepackage{comment}
\usepackage{mathtools}

\newcommand{\KL}{{\rm KL}}

\newcommand{\Xcal}{\mathcal{X}}

\newcommand{\dettwo}{\mathrm{det_2}}

\newcommand{\SymTr}{\mathrm{SymTr}}

\newcommand{\JS}{\mathrm{JS}}
\newcommand{\bP}{\mathbb{P}}

\newcommand{\Hcal}{\mathcal{H}}
\newcommand{\Nbb}{\mathbb{N}}

\newcommand{\Ubb}{\mathbb{U}}

\newcommand{\PC}{\mathscr{PC}}

\newcommand{\Pcal}{\mathcal{P}}

\newcommand{\Gauss}{\mathrm{Gauss}}

\newcommand{\Sym}{\mathrm{Sym}}
\newcommand{\Tr}{\mathrm{Tr}}
\newcommand{\HS}{\mathrm{HS}}
\newcommand{\Ncal}{\mathcal{N}}
\newcommand{\trace}{\mathrm{tr}}
\newcommand{\R}{\mathbb{R}}

\newcommand{\la}{\ensuremath{\langle}}
\newcommand{\ra}{\ensuremath{\rangle}}
\newcommand{\mapto}{\ensuremath{\rightarrow}}
\newcommand{\Lcal}{\mathcal{L}}
\newcommand{\range}{\mathrm{range}}
\newcommand{\approach}{\ensuremath{\rightarrow}}

\newcommand{\equivalent}{\ensuremath{\iff}}
\newcommand{\tr}{\mathrm{tr}}
\newcommand{\imply}{\ensuremath{\Rightarrow}}
\newcommand{\SymHS}{\mathrm{SymHS}}

\newcommand{\Bsc}{\mathscr{B}}

\newcommand{\Fcal}{\mathcal{F}}
\newcommand{\JSD}{\mathrm{JSD}}
\newcommand{\logdet}{\mathrm{logdet}}
\newcommand{\traceX}{\mathrm{tr_X}}
\newcommand{\detX}{\mathrm{det_X}}
\newcommand{\trX}{\mathrm{tr_X}}
\newcommand{\etr}{\mathrm{tr_X}}
\newcommand{\onebf}{\mathbf{1}}
\newcommand{\Csc}{\mathscr{C}}
\title{Geometric Jensen-Shannon Divergence Between Gaussian Measures On Hilbert Space}
\titlerunning{Geometric Jensen-Shannon Divergence Between Equivalent Gaussian Measures}
\author{H\`a Quang Minh
	\and
	Frank Nielsen
	}
\authorrunning{H.Q. Minh and F. Nielsen}
\institute{RIKEN Center for Advanced Intelligence Project, Tokyo, Japan
\email{minh.haquang@riken.jp}
\and
Sony Computer Science Laboratories, Tokyo, Japan\\
\email{frank.nielsen.x@gmail.com}}

\begin{document}
\maketitle              % typeset the header of the contribution

\begin{abstract}
	This work studies the Geometric Jensen-Shannon divergence, based on the notion of geometric mean of probability measures, in the setting of Gaussian measures on an infinite-dimensional Hilbert space. On the set of all Gaussian measures equivalent to a fixed one, we present a closed form expression for this divergence that directly generalizes the finite-dimensional version. Using the notion of Log-Determinant divergences between positive definite unitized trace class operators, we then define a Regularized Geometric Jensen-Shannon divergence that is valid for any pair of Gaussian measures and that recovers the exact Geometric Jensen-Shannon divergence  between two equivalent Gaussian measures when the regularization parameter
	% approaches 
	tends to
	zero.
	
	%\keywords{First keyword  \and Second keyword \and Another keyword.}
\end{abstract}

	\section{Introduction}
\label{section:introduction}

We first briefly review the
% concept of 
Jensen-Shannon divergence between two probability measures and its generalization via weighted abstract means. 	
%We recall that for two measures $\mu$ and $\nu$ on a measure space $(\Omega, \Fcal)$, with $\mu$ $\sigma$-finite, 
%$\nu$ is said to be {\it absolutely continuous} with respect to $\mu$, denoted by $\nu << \mu$, if $\forall A \in \Fcal$, $\mu(A) = 0 \imply \nu(A) = 0$. The Radon-Nikodym density $\frac{d\nu}{d\mu}$
%$ \in \Lcal^1(\mu)$
%is then
%well-defined.
In the following, let $\Xcal$ be a Polish space
%complete separable metric space 
and $\Bsc(\Xcal)$ its Borel $\sigma$-algebra. Let $\Pcal(\Xcal)$ denote the set of probability measures on $(\Xcal, \Bsc(\Xcal))$.	
The 
Kullback-Leibler 
(KL) divergence \cite{Kullback1951information} between two probability measures $P,Q \in \Pcal(\Xcal)$ is defined by
\begin{align}
	\KL(P||Q) = 
	\left\{
	\begin{matrix}
		\int_{\Xcal}\log\left\{\frac{dP}{dQ}(x)\right\}dP(x) & \text{if $P \ll Q$},
		\\
		\infty & \text{otherwise}.
	\end{matrix}
	\right.
\end{align} 
Here $P \ll Q$ means $P$ is absolutely continuous with respect to $Q$.
%Let $\mu$ be a positive $\sigma$-finite measure on $(\Xcal, \Bsc(\Xcal))$ such that $P << \mu, Q<< \mu$. Let $p = \frac{dP}{d\mu}$, $q = \frac{dQ}{d\mu}$ be the corresponding Radon-Nikodym densities. Then
%\begin{align}
%	\KL(P||Q) = \int_{\Xcal}p(x)\log\frac{p(x)}{q(x)}d\mu(x).
%\end{align}
The KL divergence is asymmetric and one of its most well-known symmetrizations is the
Jensen-Shannon divergence, which is defined by \cite{lin1991JensenShannondivergence}
\begin{align}
	\label{equation:Jensen-Shannon-original}
	\JS(P||Q) = \frac{1}{2}\KL\left(P\middle\|\frac{P+Q}{2}\right) + \frac{1}{2}\KL\left(Q\middle\|\frac{P+Q}{2}\right).
\end{align}
%In contrast to 
While $\KL(P||Q)$ can be infinite,
%, which can be infinite, 
$\JS(P||Q)$ is always finite, since we always have $P\ll \frac{P+Q}{2},Q\ll \frac{P+Q}{2}$.
 %always. 
 Furthermore, its square root, namely $\JSD(P||Q) 
=\sqrt{\JS(P||Q)}$, called Jensen-Shannon distance, satisfies the triangle inequality \cite{endres2003JensenShannonDistance} and thus is a metric
on $\Pcal(\Xcal)$. Practical applications of $\JS/\JSD$ are many,
%: examples include 
including
bioinformatics \cite{sims2009alignment},
social systems \cite{dedeo2013bootstrap}, and
generative adversarial networks (GAN) in deep learning \cite{goodfellow2014generative}.
%\section{Finite-dimensional setting}
%\label{section:finite}
%The probability measure $\frac{P+Q}{2}$ is the {\it arithmetic mean} of $P$ and $Q$, which is an example of a quasi-arithmetic mean. By employing abstract means, as formulated in \cite{nielsen2019geometricJensenShannon} and discussed below,
%we have a general framework of $M$-Jensen-Shannon divergences, of which Eq.\eqref{equation:Jensen-Shannon-original}
%is a special case.
%For $P_i = \Ncal(m_i, C_i)$, $i=1,2$
%\begin{align}
%\KL(P_1||P_2) = \frac{1}{2}(m_2 - m_1)^TC_2^{-1}(m_2 - m_1)
%+ \frac{1}{2}[\trace(C_2^{-1}C_1 - I) - \log\det(C_2^{-1}C_1)]
%\end{align}	
%We first summarize the relevant results from \cite{nielsen2019geometricJensenShannon}.

%\begin{align}
%\Sigma_{\alpha} = (\Sigma_1\Sigma_2)^{\Sigma}_{\alpha}
%= ((1-\alpha)\Sigma_1^{-1} + \alpha \Sigma_2^{-1})^{-1}
%\end{align}
%\begin{align}
%	\mu_{\alpha} = (\mu_1\mu_2)^{\mu}_{\alpha} = \Sigma_{\alpha}[(1-\alpha)\Sigma_1^{-1}\mu_1 + \alpha \Sigma_2^{-1}\mu_2]
%\end{align}

%\begin{align}
%\JS(\Ncal(\mu_1, \Sigma_1)||\Ncal(\mu_2, \Sigma_2))
%= (1-\alpha)\KL(\Ncal(\mu_1,\Sigma_1)||\Ncal(\mu_{\alpha}, \Sigma_{\alpha}) +\alpha \KL(\Ncal(\mu_2,\Sigma_2)||\Ncal(\mu_{\alpha},\Sigma_{\alpha})
%\end{align}

%\begin{align*}
%\KL(\Ncal(\mu_1,\Sigma_1)||\Ncal(\mu_{\alpha}, \Sigma_{\alpha})
%\end{align*}

%\subsection{Statistical mixtures via abstract weighted means}
%\label{section:statistical-mixture}

{\bf Abstract means}. The probability measure $\frac{P+Q}{2}$ is the {\it arithmetic mean} of $P$ and $Q$, which is 
%an example of a 
a weighted mean. Let $I \subset \R$. An {\it abstract mean} 
\cite{nielsen2019geometricJensenShannon} is a  function $M: I \times I \mapto I$ such that $\inf\{x,y\} \leq M(x,y) \leq \sup\{x,y\}$ $\forall x,y \in I$. An abstract mean $M$ is said to be {\it 
	regular} if it is symmetric, continuous,
%in both arguments, 
homogeneous, 
%i.e. $M(\lambda x, \lambda y) = \lambda M(x,y)$ $\forall \lambda > 0$, $\forall x,y \in I$, 
and increasing in each variable. Let $M$ be a regular mean and $\alpha \in [0,1]$, then, using the dyadic expansion in \cite{Niculescu2003ConvexMeans}, there is a corresponding {\it weighted mean} $M_{\alpha}$ such that $M_0(x,y) =x, M_1(x,y) = y$.
%a quasi-arithmetic mean. 
%We briefly review the concept of {\it abstract mean}
%and {\it weighted mean} $M_{\alpha}:I \times I \mapto I$.
%Examples of 
Common weighted means, with $I = (0, \infty)$, include: 
%\begin{enumerate}
%\item 
{\it Arithmetic mean} $A_{\alpha}(x,y) = (1-\alpha)x + \alpha y$;
%\item 
{\it Geometric mean} $G_{\alpha}(x,y) = x^{1-\alpha}y^{\alpha}$; 
%\item 
{\it Harmonic mean} $H_{\alpha}(x,y) = \frac{xy}{(1-\alpha) y + \alpha x}$.
%\end{enumerate}
These are instances of the
% general
% framework of 
{\it quasi-arithmetic means} (or
Kolmogorov–Nagumo means) $M^h_{\alpha}(x,y) = h^{-1}((1-\alpha)h(x) + \alpha h(y))$, 
where $h:I \mapto \R$ is strictly monotonous.
%a strictly monotonous function. 
The geometric mean $G_{\alpha}$, e.g.,
%for example,
corresponds to $h(u) = \log(u)$. By employing abstract means, 
%as formulated in \cite{nielsen2019geometricJensenShannon} and discussed below,
we have a general framework of $M$-Jensen-Shannon divergences \cite{nielsen2019geometricJensenShannon}, of which Eq.\eqref{equation:Jensen-Shannon-original}
is a special case. 

{\bf Statistical mixture via weighted means \cite{nielsen2019geometricJensenShannon}}.
% In the following, 
%Let $\mu$ be a fixed positive measure on $(\Xcal, \Bsc(\Xcal))$. 
%Let $\Pcal(\Xcal, \mu) \subset \Pcal(\Xcal)$ denote the set of all probability measures on $(\Xcal, \Bsc(\Xcal))$ that are absolutely continuous with respect to $\mu$.
%The following is introduced in \cite{nielsen2019geometricJensenShannon}.
%Following Nielsen \cite{nielsen2019geometricJensenShannon} (see \cite{nielsen2019geometricJensenShannon}, Definition 1), we define
%\begin{definition}
%	[\textbf{Statistical mixtures via weighted means} \cite{nielsen2019geometricJensenShannon}]
%	\label{definition:M-alpha-interpolation}
Let $\alpha \in [0,1]$. Let $M_{\alpha}$ be a given weighted mean.
%	Let $P, Q \in \Pcal(\Xcal,\mu)$. 
Let $P, Q \in \Pcal(\Xcal)$. Let $\mu$ be a positive measure on $(\Xcal, \Bsc(\Xcal))$ such that $P\ll \mu, Q\ll \mu$.
We note that such a $\mu$ always exists, for example we can take $\mu = \frac{P+Q}{2}$ as in Eq.\eqref{equation:Jensen-Shannon-original}.
Let $p= \frac{dP}{d\mu}$, $q = \frac{dQ}{d\mu}$ be the corresponding Radon-Nikodym densities with respect to $\mu$. 
The $M_{\alpha}$-interpolation $(pq)^M_{\alpha}$, or $\alpha$-weighted $M$-mixture, of $p$ and $q$ is defined to be %\cite{nielsen2019geometricJensenShannon}
\begin{align}
	(pq)^M_{\alpha}(x) := \frac{M_{\alpha}(p(x), q(x))}{Z^M_{\alpha}(p:q)},
\end{align}
where 
%$Z^M_{\alpha}(p:q)$ is the normalizing factor defined by
%\begin{align}
$Z^M_{\alpha}(p:q) = \int_{\Xcal}M_{\alpha}(p(x), q(x))\; d\mu(x)$
is the normalizing factor.
%\end{align}
We call $(PQ)^M_{\alpha}$, the probability measure in $\Pcal(\Xcal)$ whose Radon-Nikodym density with respect to $\mu$ is $(pq)^M_{\alpha}(x)$, the $\alpha$-weighted $M$-mixture of $P$ and $Q$.

%\end{definition}
{\bf Weighted geometric mixture of Gaussian measures on $\R^n$}.
For $n \in \Nbb$, 
%let $\Sym(n)$ denote the set of real $n \times n$ symmetric matrices,
let $\Sym^{++}(n)$ denote the set of real $n \times n$ symmetric positive definite matrices.  
%For the {\it geometric mean} $G_{\alpha}(x,y) = x^{1-\alpha}y^{\alpha}$ applied to Gaussian densities
For $m_i \in \R^n, C_i \in \Sym^{++}(n)$, $i=0,1$, consider the Gaussian measures $\mu_i=\Ncal(m_i,C_i)$ with densities with respect to the Lebesgue measure
%\begin{align*}
$p_i(x) = \frac{1}{\sqrt{(2\pi)^n\det(C_i)}}\exp\left(-\frac{1}{2}(x-m_i)^TC_i^{-1}(x-m_i)\right)$.
%\end{align*}
%$i=0,1$.
Their $\alpha$-weighted geometric mixture is the Gaussian measure with density \cite[Corollary 1]{nielsen2019geometricJensenShannon}
% {\it geometric mean} $G_{\alpha}(x,y) = x^{1-\alpha}y^{\alpha}$ applied to Gaussian densities
%\begin{lemma}
%\label{lemma:geometric-interpolation-Gaussians}
%For two Gaussian densities $p_i(x) = \frac{1}{\sqrt{(2\pi)^n\det(\Sigma_i)}}\exp\left(-\frac{1}{2}(x-m_i)^T\Sigma_i^{-1}(x-m_i)\right)$,
%$i=1,2$,
\begin{align}
	\label{equation:geometric-interpolation-Gaussian-finite}
	(p_0p_1)^G_{\alpha}(x) = \Ncal(\mu_{\alpha}, C_{\alpha})(x),
	%	\\
	%\end{align}
	\text{ where }
	%\begin{align}
	C_{\alpha} &= [(1-\alpha)C_0^{-1} + \alpha C_1^{-1}]^{-1},
	%	\label{equation:C-alpha-finite}
	\nonumber
	\\
	m_{\alpha} &= C_{\alpha}[(1-\alpha)C_0^{-1}m_0 + \alpha C_1^{-1}m_1].
	%	\label{equation:m-alpha-finite}
	%\nonumber
\end{align}
%\end{lemma}
%{\bf $M$-Jensen-Shannon divergence}
With the weighted $M$-mixture of $P$ and $Q$, the following is the generalization of the Jensen-Shannon divergence
in Eq.\eqref{equation:Jensen-Shannon-original}.
%\begin{definition}
%[\textbf{$M$-Jensen-Shannon divergence}]

\textbf{$M$-Jensen-Shannon divergence}.
%	\label{definition:M-Jensen-Shannon-divergence}
Let $M$ be a given abstract mean. 
%Let $\alpha \in [0,1]$. 
Let $P,Q \in \Pcal(\Xcal,\mu)$. The $M$-Jensen-Shannon divergence between 
$P$ and $Q$ is defined to be
\begin{align}
	\JS_{M_\alpha}(P||Q) = (1-\alpha)\KL(P||(PQ)^M_{\alpha}) + \alpha \KL(Q||(PQ)^M_{\alpha}), \;\; 0 \leq \alpha \leq 1.
\end{align}
%\end{definition}
For $M_{1/2} = A_{1/2}$, we recover the original Jensen-Shannon divergence in Eq.\eqref{equation:Jensen-Shannon-original}.

{\bf Geometric Jensen-Shannon divergence between Gaussian measures on $\R^n$}.
The {\it arithmetic mean} of two Gaussian measures $P,Q$ is not a Gaussian measure and,
as of the current writing, no closed-form formula is known for 
%the Jensen-Shannon divergence 
$\JS(P||Q) = \JS_{A_{1/2}}(P||Q)$ (however, $\JS(P||Q)$ admits a closed form formula if $P,Q$ are Cauchy distributions \cite{nielsen2021FDivergengeCauchy}).
%between two Gaussian measures. 
On the other hand, 
%as shown in 
by Eq.\eqref{equation:geometric-interpolation-Gaussian-finite},
the {\it geometric mean} of two Gaussian densities is another closed form Gaussian density. Consequently, we
have the following explicit expression for the geometric Jensen-Shannon divergence between two Gaussian measures.
Let $\mu_i = \Ncal(m_i, C_i)$, $i=0,1$, with $C_i \in \Sym^{++}(n)$. 
For $M_{\alpha} = G_{\alpha}$, $\alpha \in [0,1]$, we have \cite[Corollary 1]{nielsen2019geometricJensenShannon}
\begin{align}
	\label{equation:geometric-Jensen-Shannon-Gaussian-finite}
	&\JS_{G_{\alpha}}(\mu_0||\mu_1) = 
	\frac{(1-\alpha)}{2}||C_{\alpha}^{-1/2}(m_0- m_{\alpha})||^2 
	+\frac{\alpha}{2}||C_{\alpha}^{-1/2}(m_1- m_{\alpha})||^2
	\nonumber
	\\
	&\quad -\frac{1}{2}\log\frac{\det(C_0)^{1-\alpha}\det(C_1)^{\alpha}}{\det(C_{\alpha})} 
	%\nonumber
	%\\
	%& \quad 
	+ \frac{1}{2}\trace[C_{\alpha}^{-1}((1-\alpha)C_0 + \alpha C_1) - I].
\end{align}

	\section{Background for the infinite-dimensional setting}
	\label{section:infinite}
	In this work, we generalize the
	weighted geometric mean for Gaussian measures on $\R^n$ in 
	Eq.\eqref{equation:geometric-interpolation-Gaussian-finite} and the corresponding geometric Jensen-Shannon divergence in Eq.\eqref{equation:geometric-Jensen-Shannon-Gaussian-finite} to the 
	infinite-dimensional 
	setting of Gaussian measures on a
	%n infinite-dimensional 
	%separable 
	Hilbert space $\Hcal$. 
	Throughout the following, let $(\Hcal, \la, \ra)$ be a real, separable Hilbert space, with $\dim(\Hcal) = \infty$ unless explicitly stated otherwise.
	For two 
	%separable 
	Hilbert spaces {\color{black}$(\Hcal_i, \la,\ra_i)$},$i=1,2$, let $\Lcal(\Hcal_1,\Hcal_2)$ denote the Banach space of bounded linear operators from $\Hcal_1$ to $\Hcal_2$, with operator norm $||A||=\sup_{||x||_1\leq 1}||Ax||_2$.
	For $\Hcal_1=\Hcal_2 = \Hcal$, we write $\Lcal(\Hcal)$.
	%Let $\Lcal(\H)$ denote the Banach space of bounded linear operator on $\H$. 
	%For
	%$A \in \Lcal(\H)$, $||A|| = \sup_{||x||\leq 1}||Ax||$. 
	Let $\Sym(\Hcal) \subset \Lcal(\Hcal)$ be the set of bounded, self-adjoint linear operators on $\Hcal$. Let $\Sym^{+}(\Hcal) \subset \Sym(\Hcal)$ be the set of
	self-adjoint, {\it positive} operators on $\Hcal$, i.e. $A \in \Sym^{+}(\Hcal) \equivalent A^{*}=A, \la Ax,x\ra \geq 0 \forall x \in \Hcal$, where $A^{*}$ is the adjoint of $A$, which is the unique operator in $\Lcal(\Hcal)$ satisfying $\la Ax, y\ra = \la x, A^{*}y\ra \forall x,y \in \Hcal$.
	Let $\Sym^{++}(\Hcal)\subset \Sym^{+}(\Hcal)$ be the set of self-adjoint, {\it strictly positive} operator on $\Hcal$,
	i.e $A \in \Sym^{++}(\Hcal) \equivalent A^{*}=A, \la x, Ax\ra > 0$ $\forall x\in \Hcal, x \neq 0$.
	We write $A \geq 0$ for $A \in \Sym^{+}(\Hcal)$ and $A > 0$ for $A \in \Sym^{++}(\Hcal)$.
	%If $\gamma I+A > 0$, where $I$ is the identity operator,$\gamma \in \R,\gamma > 0$, then $\gamma I+A$ is also invertible, in which case it is called
	%{\it positive definite}. 
	{\color{black}
		%Furthermore, 
		We say that $A \in  \Sym(\Hcal)$ is 
		%said to be 
		{\it positive definite} if $\exists M_A > 0$ such that $\la x, Ax\ra \geq M_A||x||^2$ $\forall x \in \Hcal$ - this
		% condition 
		is equivalent to $A$ being both strictly positive and invertible, see e.g. \cite{Petryshyn:1962}.}
	We denote by $\bP(\Hcal)$ the set of all positive definite operators on $\Hcal$.
	The Banach space $\Tr(\Hcal)$  of trace class operators on $\Hcal$ is defined by (see e.g. \cite{ReedSimon:Functional})
	%\begin{align*}
	$\Tr(\Hcal) = \{A \in \Lcal(\Hcal): ||A||_{\tr} = \sum_{k=1}^{\infty}\la e_k, (A^{*}A)^{1/2}e_k\ra < \infty\}$,
	%\\
	%\HS(\H) &= \{A \in \Lcal(\H): ||A||_{\HS}^2 = \sum_{k=1}^{\infty}||Ae_k||^2  < \infty\},
	%\end{align*}
	for any orthonormal basis {\color{black}$\{e_k\}_{k \in \Nbb} \subset \Hcal$}.
	%$\{e_k\}_{k \in \Nbb} \in \H$.
	For $A \in \Tr(\Hcal)$, its trace is defined by $\trace(A) = \sum_{k=1}^{\infty}\la e_k, Ae_k\ra$, which is independent of choice of $\{e_k\}_{k\in \Nbb}$. 
	The Hilbert space $\HS(\Hcal_1,\Hcal_2)$ of Hilbert-Schmidt operators from $\Hcal_1$ to $\Hcal_2$ is defined by 
	(see e.g. \cite{Kadison:1983})
	$\HS(\Hcal_1, \Hcal_2) = \{A \in \Lcal(\Hcal_1, \Hcal_2):||A||^2_{\HS} = \trace(A^{*}A) =\sum_{k=1}^{\infty}||Ae_k||_2^2 < \infty\}$,
	for any orthonormal basis $\{e_k\}_{k \in \Nbb}$ in $\Hcal_1$,
	%It is a Hilbert space 
	with inner product $\la A,B\ra_{\HS}=\trace(A^{*}B)$. For $\Hcal_1 = \Hcal_2 = \Hcal$, we write $\HS(\Hcal)$. 
	The set  $\Ubb(\Hcal)$ of unitary operators on $\Hcal$ is defined by $\Ubb(\Hcal) = \{U \in \Lcal(\Hcal): UU^{*} = U^{*}U = I\}$, where $I$ is the identity operator on $\Hcal$.
	%When $\dim(\Hcal) = \infty$,
	%{\color{black}$\Tr(\Hcal) \subsetneq \HS(\Hcal) \subsetneq \Lcal(\Hcal)$},
	%$\Lcal(\H) \subsetneq \trace(\H) \subsetneq \HS(\H)$ 
	%with $||A||\leq ||A||_{\HS}\leq ||A||_{\tr}$.
	
	{\bf Equivalence of Gaussian measures}. 
	%On $\R^n$, any two Gaussian densities are equivalent, that is they have the same support, which is all of $\R^n$.
	%The situation is drastically different in the infinite-dimensional setting.
	Let $Q,R \in \Sym^{++}(\Hcal) \cap \Tr(\Hcal)$.
	% be two self-adjoint, positive trace class operators on $\Hcal$ such that $\ker(Q) = \ker(R) = \{0\}$. 
	Let $m_1, m_2 \in \Hcal$. 
	A fundamental result in the theory of Gaussian measures on an infinite-dimensional Hilbert space $\Hcal$ is the Feldman-Hajek Theorem \cite{Feldman:Gaussian1958}, \cite{Hajek:Gaussian1958}, which states that 
	two Gaussian measures $\mu = \Ncal(m_1,Q)$ and 
	$\nu = \Ncal(m_2, R)$ on $\Hcal$
	%on a Hilbert space 
	are either mutually singular or
	%, denoted by $\mu \perp \nu$, or mutually 
	equivalent, {\color{black}that is either $\mu \perp \nu$ or $\mu \sim \nu$}.
	The necessary and sufficient conditions for
	the equivalence of the two Gaussian measures $\nu$ and $\mu$ are given by the following.
	\begin{theorem}
		[\cite{Bogachev:Gaussian}, Corollary 6.4.11, \cite{DaPrato:PDEHilbert}, Theorems  1.3.9 and 1.3.10]
		\label{theorem:Gaussian-equivalent}
		Let $\Hcal$ be a separable Hilbert space. Consider two Gaussian measures $\mu = \Ncal(m_1, Q)$,
		$\nu = \Ncal(m_2, R)$ on $\Hcal$. Then $\mu$ and $\nu$ are equivalent if and only if 
		the following
		conditions 
		both hold:
		%\begin{enumerate}
		%\item 
		(1) $m_2 - m_1 \in \range(Q^{1/2})$;
		%\item 
		(2) There exists  $S \in  \Sym(\Hcal) \cap \HS(\Hcal)$, $I -S > 0$,
		%without the eigenvalue $1$, 
		such that
		%\begin{align}
		%\label{equation:RQ-equivalent}
		$R = Q^{1/2}(I-S)Q^{1/2}$.
		%\end{align}
		%\end{enumerate}
		%Furthermore, if $m_1 = m_2 = 0$, then the Radon-Nikodym derivative of $\nu$ with respect to $\mu$ is given by
		%\begin{align}
		%\frac{d\nu}{d\mu}(x) = \exp\left(\frac{1}{2}\sum_{k=1}^{\infty}\left[\frac{\alpha_k}{1+\alpha_k}W_{\phi_k}^2(x) - \log(1+\alpha_k)\right]\right).
		%\end{align}
	\end{theorem}
	
	%\begin{align}
	%\KL^{\gamma}()
	%\end{align}
	%Utilizing the identity
	
	%\begin{align*}
	%(A^{-1} + B^{-1})^{-1} = [B^{-1}(A+B)A^{-1}]^{-1} = A(A+B)^{-1}B = B(A+B)^{-1}A
	%\end{align*}	
	
	%For $0 < \alpha < 1$
	%\begin{align*}
	%\Sigma_{\alpha} = (1-\alpha)^{-1}\Sigma_1 [(1-\alpha)^{-1}\Sigma_1 + \alpha^{-1}\Sigma_2]^{-1}\alpha^{-1}\Sigma_2
	%\end{align*}
	
	%\begin{align}
	%\KL^{\gamma}(P_1||P_2)
	%\end{align}

%\subsection{Kullback-Leibler divergence between equivalent Gaussian measures}

%\label{section:KL-divergence-Gaussian-Hilbert-space}
%\subsection
{\bf Kullback-Leibler (KL) divergence between equivalent Gaussian measures}.
Let $\Gauss(\Hcal)$ be the set of all Gaussian measures on $\Hcal$.
Let $\mu, \nu \in \Gauss(\Hcal)$.
If $\mu \perp \nu$, then
$\KL(\nu|| \mu) = \infty$. If $\mu \sim \nu$, then we have the following result, which plays a crucial role
in subsequent sections of the current work.
%We now recall results on Kullback-Leibler (KL) divergences between two Gaussian measures $\mu = \Ncal(m_1,Q)$ and $\nu = \Ncal(m_2,R)$ on $\Hcal$. If $\mu \perp \nu$, then
%$\KL(\nu|| \mu) = \infty$. If $\mu \sim \nu$, then we have the following result. 
%let $S \in \HS(\H)\cap \Sym(\H)$, with $I-S > 0$, be such that $R = Q^{1/2}(I-S)Q^{1/2}$, then
%(see e.g. \cite{Minh:2020regularizedDiv}),
\begin{theorem}
	[\cite{Minh:2020regularizedDiv}]
	\label{theorem:KL-gaussian}
	Let $\mu = \Ncal(m_1, Q)$, $\nu = \Ncal(m_2, R)$, with $\ker(Q) = \ker{R} = \{0\}$, and
	%Assume that 
	$\mu \sim \nu$.
	Let $S \in \Sym(\Hcal) \cap \HS(\Hcal)$, $I-S > 0$, 
	be such that $R = Q^{1/2}(I-S)Q^{1/2}$, then
	\begin{align}
		\label{equation:KL-gaussian}
		\KL(\nu ||\mu) = \frac{1}{2}||Q^{-1/2}(m_2-m_1)||^2 -\frac{1}{2}\log\dettwo(I-S).
	\end{align}
\end{theorem}
Here $\dettwo$ is the Hilbert-Carleman determinant, see e.g. \cite{Simon:1977}, with $\dettwo(I+A) = \det[(I+A)\exp(-A)]$
for $A \in \HS(\Hcal)$, where $\det$ is the Fredholm determinant,
given by $\det(I+A) = \prod_{j=1}^{\infty}(1+\lambda_k(A))$, $A \in \Tr(\Hcal)$, $\{\lambda_k(A)\}_{k=1}^{\infty}$ being the eigenvalues of $A$.
	
	\section{Geometric Jensen-Shannon Divergence Between Equivalent Gaussian Measures}
	\label{section:geometric-JS-divergence-equivalent-Gaussian}
	%\section{Geometric interpolation of two equivalent Gaussian measures}
	%\label{section:geometric-interpolation-equivalent-Gaussian-measures}
	
	{\bf From finite to infinite-dimensional setting}. 
	%Let $\Gauss(\Hcal)$ be the set of all Gaussian measures on $\Hcal$.
	A Gaussian measure $\mu = \Ncal(m, C)$ on $\Hcal$ is said to be {\it nondegenerate} if $\ker(C) = \{0\}$, i.e. $C \in \Sym^{++}(\Hcal)\cap\Tr(\Hcal)$.
	If $\dim(\Hcal) = \infty$, then there is no natural reference measure on $\Hcal$, since the Lebesgue measure does not exist,
	and consequently the concept of Gaussian density functions with respect to the Lebesgue measure on $\R^n$ does not generalize on $\Hcal$.
	Thus Eqs. \eqref{equation:geometric-interpolation-Gaussian-finite}
	%, \eqref{equation:C-alpha-finite}, \eqref{equation:m-alpha-finite}, 
	and \eqref{equation:geometric-Jensen-Shannon-Gaussian-finite}
	do not generalize to an arbitrary pair of nondegenerate measures in $\Gauss(\Hcal$). Instead, by the Feldman-Hajek Theorem, 
	we fix a measure $\mu_{*} \in \Gauss(\Hcal)$ and consider the set $\Gauss(\Hcal,\mu_{*})$ of all 
	Gaussian measures that are equivalent to $\mu_{*}$. Then the Radon-Nikodym density $\frac{d\mu}{d\mu_{*}}$ is well-defined and positive $\forall \mu \in \Gauss(\Hcal,\mu_{*})$. The geometric Jensen-Shannon divergence is then well-defined on the set
	$\Gauss(\Hcal,\mu_{*})$, which we subsequently focus on.
	Throughout the following, let $\mu_{*} = \Ncal(m_{*}, C_{*}) \in \Gauss(\Hcal)$
	% denote a fixed Gaussian measure 
	%on $\Hcal$, 
	with mean $m_{*} \in \Hcal$ and covariance operator $C_{*} \in \Sym^{++}(\Hcal) \cap \Tr(\Hcal)$.
	%$\ker(C_{*}) = \{0\}$. 
	%Let $\Gauss(\Hcal,\mu_{*)}$ denote the set of all Gaussian measures on $\Hcal$ that are equivalent to $\mu_{*}$. 
	By Theorem \ref{theorem:Gaussian-equivalent},  the set $\Gauss(\Hcal,\mu_{*)}$ of all Gaussian measures on $\Hcal$ that are equivalent to $\mu_{*}$ is given by the following
	\begin{align}
		\label{equation:Gauss-H-mustar}
		\Gauss(\Hcal, \mu_{*}) = \left\{\mu = \Ncal(m,C): m - m_{*} \in \range(C_{*}^{1/2}), \right.
		\nonumber
		\\
		\left.\quad \quad \quad \quad \quad \quad \quad C = C_{*}^{1/2}(I-S)C_{*}^{1/2}, \; S \in \SymHS(\Hcal)_{< I}\right\},
	\end{align}
	where the set $\SymHS(\Hcal)_{<I}$ is given by
	\begin{align}
		\label{equation:SymHS<I}
		\SymHS(\Hcal)_{<I} = \{S \in \Sym(\Hcal) \cap \HS(\Hcal): I-S > 0\}.
	\end{align}
	We furthermore define the following subset of $\SymHS(\Hcal)_{<I}$
	\begin{align}
		\label{equation:SymTr<I}
		\SymTr(\Hcal)_{<I} = \{S \in \Sym(\Hcal) \cap \Tr(\Hcal): I-S > 0\}.
	\end{align}
	In \cite{Minh2024:FisherRaoGaussian}, the set $\Gauss(\Hcal, \mu_{*})$, in the zero-mean setting,
	is an infinite-dimensional Hilbert manifold on which the Fisher-Rao Riemannian metric is well-defined,
	giving it an infinite-dimensional Riemannian Cartan-Hadamard manifold structure. 
	As we show below, as is done in \cite{Minh2024:FisherRaoGaussian}, our analysis and explicit expressions are first obtained on the subset
	$\SymTr(\Hcal)_{<I}$, then extended to the set $\SymHS(\Hcal)_{<I}$ by a limiting process.
	%This strategy is also employed in \cite{Minh2024:FisherRaoGaussian}.
	%\begin{align}
	%q(x) = \frac{d\nu}{d\mu_0}(x) = \exp\left[-\frac{1}{2}\sum_{k=1}^{\infty}\Psi_k(x)\right]\exp\left[-\frac{1}{2}||(I-T)^{-1/2}Q^{-1/2}(m_2 - m_{*})||^2\right],
	%\end{align}
	
	For two measures $\mu_i=\Ncal(m_i, C_i) \in \Gauss(\Hcal,\mu_{*})$, $C_i = C_{*}^{1/2}(I-S_i)C_{*}^{1/2}$, $i=0,1$, their Radon-Nikodym densities 
	$p_i = \frac{d\mu_i}{d\mu_{*}}$ with respect to $\mu_{*}$
	admit closed forms (presented in Theorem \ref{theorem:radon-nikodym-infinite} in Section \ref{section:Radon-Nikodym-density-equivalent-Gaussian-Hilbert-space}).
	%(e.g. Theorem 11 in \cite{Minh:2020regularizedDiv}).
	%given explicitly by Theorem 
	%\ref{theorem:radon-nikodym-infinite}. 
	The following gives then the closed form formula for their geometric interpolation.
	% when $S_i \in \SymTr(\Hcal)_{<I}$.
	\begin{theorem}
		[\textbf{Geometric interpolation of two equivalent Gaussian measures on Hilbert space}]
		\label{theorem:geometric-interpolation-equivalent-Gaussian-Hilbert-space}
		Let $\mu_i = \Ncal(m_i, C_i) \in \Gauss(\Hcal, \mu_{*})$, $i=0,1$, where
		$C_i = C_{*}^{1/2}(I-S_i)C_{*}^{1/2}$, $S_i \in \SymHS(\Hcal)_{<I}$. Let the corresponding 
		Radon-Nikodym densities with respect to $\mu_{*}$ be denoted by $p_i = \frac{d\mu_i}{d\mu_{*}}$. Then
		for $\alpha \in [0,1]$,
		\begin{align}
			(p_0p_1)^{G}_{\alpha}(x) = \frac{d\mu_{\alpha}}{d\mu_{*}},
			% = \frac{d\Ncal(m_{\alpha},C_{\alpha})}{d\mu_{*}}
			%	\nonumber
		\end{align}
		where $\mu_{\alpha} = \Ncal(m_{\alpha}, C_{\alpha}) \in \Gauss(\Hcal,\mu_{*})$, $C_{\alpha} = C_{*}^{1/2}(I-S_{\alpha})C_{*}^{1/2}$, with
		\begin{align}
			S_{\alpha} &= I - [(1-\alpha)(I-S_0)^{-1} + \alpha (I-S_1)^{-1}]^{-1} \in \SymHS(\Hcal)_{< I},
			\label{equation:S-alpha}
			\\
			C_{\alpha} &= C_{*}^{1/2} [(1-\alpha)(I-S_0)^{-1} + \alpha (I-S_1)^{-1}]^{-1}C_{*}^{1/2},
			\label{equation:C-alpha}
			\\
			m_{\alpha} &= m_{*} + C_{*}^{1/2}(I-S_{\alpha})[(1-\alpha)(I-S_0)^{-1}C_{*}^{-1/2}(m_0 - m_{*}) 
			\nonumber
			\\
			&\quad \quad \quad \quad \quad \quad \quad + \alpha(I-S_1)^{-1}C_{*}^{-1/2}(m_1 - m_{*})].
			\label{equation:m-alpha}
		\end{align}
	The normalizing factor is given by
	\begin{align}
		&Z_{\alpha}^G(p_0:p_1)
		\nonumber
		= 
		\int_{\Hcal}p_0^{1-\alpha}(x)p_1^{\alpha}(x) d\mu_{*}(x)
		\\
		&=\det[(1-\alpha)(I+A)^{-\alpha} + \alpha (I+A)^{1-\alpha}]^{-1/2}
		\nonumber
		\\
		& \quad \times \exp\left[-\frac{1-\alpha}{2}||(I-S_0)^{-1/2}C_{*}^{-1/2}(m_0 - m_{*})||^2\right]
		\nonumber
		\\
		& \quad \times 
		\exp\left[-\frac{\alpha}{2}||(I-S_1)^{-1/2}C_{*}^{-1/2}(m_1 - m_{*})||^2\right]
		\nonumber
		\\
		& \quad \times \exp\left(\frac{1}{2}||(I-S_{\alpha})^{-1/2}C_{*}^{-1/2}(m_{\alpha} - m_{*})||^2\right).
		\label{equation:normalizing-factor-HS}
	\end{align}
		Here $I + A = (I-S_1)^{-1/2}(I-S_0)(I-S_1)^{-1/2} > 0$, with $A =(I-S_1)^{-1/2}(S_1-S_0)(I-S_1)^{-1/2} \in \Sym(\Hcal)\cap \HS(\Hcal).$
		In particular, for $S_0, S_1 \in \SymTr(\Hcal)_{< I}$,
		\begin{align}
			(p_0p_1)^{G}_{\alpha}(x) &	=[\det(I-S_{\alpha})]^{-1/2}
			\nonumber
			\\
			& \quad \times \exp\left\{-\frac{1}{2}\la C_{*}^{-1/2}(x-m_{*}), S_{\alpha}(I-S_{\alpha})^{-1}C_{*}^{-1/2}(x-m_{*})\ra \right\}
			\nonumber
			\\
			&\quad \times \exp(\la C_{*}^{-1/2}(x-m_{*}), (I-S_{\alpha})^{-1}C_{*}^{-1/2}(m_{\alpha} - m_{*})\ra)
			\nonumber
			\\
			&\quad \times \exp\left[-\frac{1}{2}||(I-S_{\alpha})^{-1/2}C_{*}^{-1/2}(m_{\alpha} - m_{*})||^2\right].
			\label{equation:geometric-interpolation-Gaussian-S-Tr}
		\end{align}
Let $P_N = \sum_{k=1}^Ne_k \otimes e_k$, $N \in \Nbb$, be the orthogonal projection onto the $N$-dimensional subspace of $\Hcal$ spanned
by $\{e_k\}_{k=1}^N$, where $\{e_k\}_{k=1}^{\infty}$ are the orthonormal eigenvectors of $C_{*}$.
In the expression in Eq.\eqref{equation:geometric-interpolation-Gaussian-S-Tr},
\begin{align}
	&\la C_{*}^{-1/2}(x-m_{*}), S_{\alpha}(I-S_{\alpha})^{-1}C_{*}^{-1/2}(x-m_{*})\ra 
	\nonumber
	\\
	& \doteq \lim_{N \approach \infty} \la C_{*}^{-1/2}P_N(x-m_{*}), S_{\alpha}(I-S_{\alpha})^{-1}C_{*}^{-1/2}P_N(x-m_{*})\ra 
	\\
	&
	\la C_{*}^{-1/2}(x-m_{*}), (I-S_{\alpha})^{-1}C_{*}^{-1/2}(m_{\alpha} - m_{*})\ra 
	\nonumber
	\\
	&\doteq \lim_{N \approach \infty} \la C_{*}^{-1/2}P_N(x-m_{*}), (I-S_{*})^{-1}C_{*}^{-1/2}(m_{\alpha} - m_{*})\ra,
\end{align}
with the limits being in the 
%$\Lcal^1(\Hcal,\mu)$ and 
$\Lcal^2(\Hcal, \mu_{*})$ sense.
	\end{theorem}

With $\mu_{\alpha}$, $\alpha \in [0,1]$, given explicitly by Theorem \ref{theorem:geometric-interpolation-equivalent-Gaussian-Hilbert-space},
we next obtain the geometric Jensen-Shannon divergence between $\mu_0$ and $\mu_1$, as follows.
	\begin{theorem}
		[\textbf{Geometric Jensen-Shannon divergence between equivalent Gaussian measures}]
		\label{theorem:geometric-JS-equivalent-Gaussian-measures-Hilbert-space}
		Let $\mu_i = \Ncal(m_i, C_i) \in \Gauss(\Hcal, \mu_{*})$, $i=0,1$, with $C_i= C_{*}^{1/2}(I-S_i)C_{*}^{1/2}$, $S_i \in \SymHS(\Hcal)_{< I}$. Let $m_{\alpha} \in \Hcal, S_{\alpha} \in \SymHS(\Hcal)_{<I}, C_{\alpha} \in \Sym^{+}(\Hcal) \cap \Tr(\Hcal)$ be as in Theorem \ref{theorem:geometric-interpolation-equivalent-Gaussian-Hilbert-space}. Then
		\begin{align}
			&\JS_{G_{\alpha}}(\mu_0||\mu_1) = (1-\alpha)\KL(\mu_0||\mu_{\alpha}) + \alpha \KL(\mu_1||\mu_{\alpha})
			\\
			& = \frac{(1-\alpha)}{2}||C_{\alpha}^{-1/2}(m_0- m_{\alpha})||^2 - \frac{1-\alpha}{2}\log\dettwo[(I -S_{\alpha})^{-1}
			(I-S_0)]
			\nonumber
			\\
			&\quad +\frac{\alpha}{2}||C_{\alpha}^{-1/2}(m_1- m_{\alpha})||^2 -\frac{\alpha}{2}\log\dettwo[(I -S_{\alpha})^{-1}
			(I-S_1)].
		\end{align}
		In particular, for $S_0, S_1 \in \SymTr(\Hcal)_{< I}$,
		\begin{align}
			\JS_{G_{\alpha}}(\mu_0||\mu_1) &=\frac{(1-\alpha)}{2}||C_{\alpha}^{-1/2}(m_0- m_{\alpha})||^2 +\frac{\alpha}{2}||C_{\alpha}^{-1/2}(m_1- m_{\alpha})||^2
			\nonumber
			\\
			&\quad - \frac{1}{2}\log\frac{\det(I-S_0)^{1-\alpha}\det(I-S_1)^{\alpha}}{\det(I-S_{\alpha})}
			\nonumber
			\\
			&\quad  +\frac{1}{2}\trace[(I-S_{\alpha})^{-1}(S_{\alpha}-(1-\alpha)S_0- \alpha S_1)].
			\label{equation:geometric-JS-Gaussian-Hilbert-trace-class-S}
		\end{align}
	\end{theorem}
	
	%\begin{comment}
		{\bf Finite-dimensional setting}. 
		The following verifies 
	%	One can directly verify 
		that for $\mu_i = \Ncal(m_i, C_i) \in \Gauss(\R^n)$, $m_i \in \R^n$, $C_i \in \Sym^{++}(n)$, $i=0,1$, from Theorem \ref{theorem:geometric-JS-equivalent-Gaussian-measures-Hilbert-space} we recover the finite-dimensional formula
		for $\JS_{G_{\alpha}}(\mu_0 ||\mu_1)$ in Eq.\eqref{equation:geometric-Jensen-Shannon-Gaussian-finite} (\cite{nielsen2019geometricJensenShannon}, Corollary 1), which can also be computed directly from the expression of the KL divergence between Gaussian densities on $\R^n$.
		\begin{corollary}
			\label{corollary:geometric-JS-Gaussian-finite}
			Let $\mu_i = \Ncal(m_i, C_i) \in \Gauss(\R^n)$, $i=0,1$. Then
			\begin{align}
					&\JS_{G_{\alpha}}(\mu_0 ||\mu_1) = \frac{(1-\alpha)}{2}||C_{\alpha}^{-1/2}(m_0- m_{\alpha})||^2 +\frac{\alpha}{2}||C_{\alpha}^{-1/2}(m_1- m_{\alpha})||^2
					\nonumber
					\\
					&\quad -\frac{1}{2}\log\frac{\det(C_0)^{1-\alpha}\det(C_1)^{\alpha}}{\det(C_{\alpha})} 
					+
					\frac{1}{2}\trace[C_{\alpha}^{-1}((1-\alpha)C_0 + \alpha C_1) - I].
				\end{align}
		\end{corollary}
	%\end{comment}
	
	%\section{Regularized divergence between positive definite trace class operators}
	\section{Regularized Geometric Jensen-Shannon Divergence Between Gaussian Measures On Hilbert Space}
	\label{section:regularized-JS-divergence}
	
	Theorem \ref{theorem:geometric-JS-equivalent-Gaussian-measures-Hilbert-space} is valid and finite for any pair of {\it equivalent} Gaussian measures
	$\mu_0, \mu_1$ on $\Hcal$. 
	In this section, we describe a regularized formulation of $\JS_{G_{\alpha}}(\mu_0||\mu_1)$
	that is valid and finite for {\it any pair} of Gaussian measures $\mu_0, \mu_1 \in \Gauss(\Hcal)$.
	It is based on the generalization of the Log-Det divergences on $\Sym^{++}(n)$
	%for SPD matrices in Section \ref{section:finite} 
	to the infinite-dimensional setting of 
	positive definite unitized trace class operators on a Hilbert space \cite{Minh:LogDet2016}.
	%This formulation was first proposed in \cite{Minh:LogDet2016} for the Alpha Log-Det divergences, generalizing the finite-dimensional setting in \cite{Chebbi:2012Means}. It was subsequently generalized to the Alpha-Beta Log-Det divergences between positive definite trace class operators in \cite{Minh:2019AlphaBeta}, corresponding to the finite-dimensional 
	%version in \cite{cichocki15}, and 
	%later to the entire
	%and the Hilbert manifold of positive definite Hilbert-Schmidt operators
	%in \cite{Minh:Positivity2020}. 
	We first note that for $\mu_i = \Ncal(m_i,C_i)$ on $\R^n$, $C_i \in \Sym^{++}(n)$,
	\begin{align}
		\label{equation:KL-Gaussian-finite}
		&	\KL(\mu_1 ||\mu_2)
		%= d_{\KL}[(m_1, C_1 + \gamma I), (m_2, C_2 + \gamma I)]
		=\frac{1}{2}\la m_1 - m_2, C_2^{-1}(m_1 - m_2)\ra 
		+ \frac{1}{2}d^1_{\logdet}(C_1 , C_2), \;\; 
		\\
		\text{where }&d^{1}_{\logdet}(C_1,C_2) = \trace(C_2^{-1}C_1 - I) - \log\det(C_2^{-1}C_1).
		\label{equation:d-1-logdet-finite}
	\end{align}
	%
	%Similar to the case of the affine-invariant Riemannian metric on $\Sym^{++}(n)$, 
	
	\begin{comment}
		{\bf Challenge in the infinite-dimensional setting}.% We first note that the Alpha Log-Det divergences for $\Sym^{++}(n)$, as given in Eqs.
		%\eqref{equation:logdet-alpha-finite}, \eqref{equation:alpha+1-finite}, and \eqref{equation:alpha-1-finite}, are {\it not} directly generalizable to
		The Log-Det divergence $d^1_{\logdet}$ in Eq.\eqref{equation:d-1-logdet-finite} is {\it not} directly generalizable to
		the case $A,B$ are positive trace class operators on an infinite-dimensional Hilbert space (see \cite{Minh:LogDet2016,Minh:2024InfiniteDistanceSurvey} for more detailed discussions).
		To see this, consider any compact $A \in \Sym^{++}(\Hcal)$, with eigenvalues $\{\lambda_k\}_{k=1}^{\infty}$ and corresponding orthonormal eigenvectors $\{\phi_k\}_{k=1}^{\infty}$,
		then $\lambda_k > 0$  $\forall k \in \Nbb$ and $\lim_{k \approach \infty}\lambda_k = 0$.
		Let  $\phi_k\otimes \phi_k:\Hcal \mapto \Hcal$, $k \in \Nbb$, be the rank-one operator defined by $(\phi_k \otimes \phi_k)w = \la \phi_k, w\ra\phi_k$ $\forall w \in \Hcal$.
		Then both the inverse $A^{-1} = \sum_{k=1}^{\infty}\lambda_k^{-1}\phi_k \otimes \phi_k$ and the principal logarithm
		$\log(A) = \sum_{k=1}^{\infty}\log(\lambda_k)\phi_k \otimes \phi_k$, are unbounded and a direct generalization of the $\log\det$ function on $\Sym^{++}(n)$ would give $\log\det(A) = \trace\log(A) = 
		\sum_{k=1}^{\infty}\log(\lambda_k) = -\infty$.
	\end{comment}
	
	{\bf Extended (unitized) trace class operators}. The Log-Det divergence $d^1_{\logdet}$ in Eq.\eqref{equation:d-1-logdet-finite} is {\it not} directly generalizable to
	the case $A,B \in \Sym^{++}(\Hcal)\cap \Tr(\Hcal)$ when $\dim(\Hcal) = \infty$, 
	%are positive trace class operators on an infinite-dimensional Hilbert space, 
	since for any compact operator $A \in \Sym^{++}(\Hcal)$, both $A^{-1}$ and $\log(A)$ are unbounded (see \cite{Minh:LogDet2016,Minh:2024InfiniteDistanceSurvey} for more detail). 
	%The above challenge 
	This can be resolved via the concepts of {\it extended (or unitized) trace class operators}, along with the {\it extended trace} and {\it extended Fredholm determinant}, as follows \cite{Minh:LogDet2016}.
Let $I$ be the identity operator on $\Hcal$.
Consider the algebra of {\it extended (or unitized) trace class operators}
%which is 
defined by
\begin{equation}
	\Tr_X(\Hcal) = \{A + \gamma I \; : \; A \in \Tr(\Hcal), \;\gamma \in \R\}.
	%, \text{where $I$ is the identity operator}.
\end{equation}
%where $I$ is the identity operator.
%with $I$ being the identity operator.
%This set can be equipped with 
This
set 
is a Banach algebra under the {\it extended trace norm}
\begin{equation}
	||A+\gamma I||_{\etr} = ||A||_{\trace} + |\gamma| = \trace |A| + |\gamma|.
\end{equation}
%Under this norm, $\Tr_X(\H)$ becomes a Banach algebra. 
For 
%an operator 
{$(A+\gamma I) \in \Tr_X(\Hcal)$}, we define
its {\it extended trace} by
%to be
\begin{equation}
	\etr(A+\gamma I) = \tr(A) + \gamma, \;\;\;\text{with $\etr(I) = 1$, in contrast to $\trace(I) = \infty$}.
\end{equation}
%For {$\dim(\H) < \infty$}, we set
%$||\cdot||_{\etr} = ||\cdot||_{\tr}$ and $\trX(\cdot) = \trace(\cdot)$.
%
% {$||A+\gamma I||_{\etr} = ||A+\gamma I||_{\trace} = \trace(|A+\gamma I|)$} and {$\etr(A+\gamma I) = \tr(A+\gamma I)$}.

{\bf Positive definite unitized trace-class operators}.
%The following set of positive definite operators, which 
These are scalar perturbations of self-adjoint trace class operators on $\Hcal$, %can now be 
defined as
%
%ready to define the set of positive definite unitized trace class operators
\begin{align}
	\PC_1(\Hcal) &= \bP(\Hcal) \cap \Tr_X(\Hcal)
	%\nonumber
	%\\
	%& 
	= \{A+\gamma I > 0 \; : \; A^{*} = A, A \in \Tr(\Hcal), \gamma \in \R\}. \nonumber
\end{align}
\textbf{Extended Fredholm determinant} \cite{Minh:LogDet2016}.
%\label{def:det-positive}
(i) 
%Assume that 
Let {$(A+\gamma I) \in \PC_1(\Hcal)$}. Then $\log(A+\gamma I)$ is well-defined and furthermore $\log(A+\gamma I) \in \Tr_X(\Hcal)$.
The extended Fredholm determinant of {$(A+\gamma I)$} is defined to be
\begin{equation}
	\label{equation:extended-Fredholm-det-positive}
	\detX(A+\gamma I) = \exp(\etr[\log(A+\gamma I)]).
\end{equation}
(ii) More generally, for $(A+\gamma I) \in \Tr_X(\Hcal)$, Eq.\eqref{equation:extended-Fredholm-det-positive} can be generalized to 
\begin{align}
	\detX(A+\gamma I) = \gamma \det[I+(A/\gamma)],
\end{align}
where $\det$ is the classical Fredholm determinant \cite{Fredholm:1903,gohberg2012traces,Simon:1977}.

	{\bf Infinite-dimensional Log-Det divergence}.
	The extended trace and extended Fredholm determinant lead to the following infinite-dimensional generalization of the %Alpha 
	Log-Det divergence (see \cite{Minh:LogDet2016} for the general definition of $d^{\alpha}_{\logdet}$, $-1\leq \alpha \leq 1$)
	\begin{align}
		&	d^{1}_{\logdet}[(A+\gamma I), (B+\mu I)] = (\frac{\gamma}{\mu}-1)\log\frac{\gamma}{\mu}
		%\left(\frac{\gamma}{\mu}-1\right)\log\frac{\gamma}{\mu} 
		\nonumber
		\\
		&\;\;\;\;\;+ \trX[(B+\mu I)^{-1}(A+\gamma I) - I] - \frac{\gamma}{\mu}\log\detX[(B+\mu I)^{-1}(A+\gamma I)].
		\label{equation:alpha+1}
	\end{align}
	In the case $\gamma = \mu$, $d^{\alpha}_{\logdet}[(A+\gamma I), (B+ \gamma I)]$ assumes a much simpler form, which directly generalizes the finite-dimensional formula
	\begin{align}
		d^{1}_{\logdet}[(A+\gamma I), (B+\gamma I)] &= 
		%\left(\frac{\gamma}{\mu}-1\right)\log\frac{\gamma}{\mu} 
		%\nonumber
		%\\
		%&\;\;\;\;\;+ 
		\trX[(B+\gamma I)^{-1}(A+\gamma I) - I]
		\nonumber 
		\\
		&\quad 
		-\log\detX[(B+\gamma I)^{-1}(A+\gamma I)].
	\end{align}
	{\bf Regularized geometric Jensen-Shannon divergence}.
	%Motivated by the regularized $\KL$ divergence defined Eq.\eqref{equation:regularized-KL-infinite} and the limiting behavior in Eq.\eqref{equation:limit-regularized-KL}, 
	With the infinite-dimensional Log-Det divergence,
	we now formulate the regularized geometric Jensen-Shannon divergence between any pair of Gaussian measures in $\Gauss(\Hcal)$. Let
	$\mu_i = \Ncal(m_i, C_i) \in \Gauss(\Hcal)$, $i=0,1$. Let $\gamma>0, \gamma \in \R$ be fixed. Define the following regularized terms
	\begin{align}
		\label{equation:C-alpha-gamma}
		C_{\alpha, \gamma} &= [(1-\alpha)(C_0 + \gamma I)^{-1} + \alpha (C_1 + \gamma I)^{-1}]^{-1}, \; 0 \leq \alpha \leq 1,
		\\
		m_{\alpha,\gamma} &= C_{\alpha,\gamma}[(1-\alpha)(C_0+\gamma I)^{-1}m_0 + \alpha (C_1+\gamma I)^{-1}m_1].
		\label{equation:m-alpha-gamma}
	\end{align}
	The following shows that $C_{\alpha,\gamma}$ has the form $C_{\alpha,\gamma} = \gamma I +A \in \PC_1(\Hcal)$, for $A \in \Sym(\Hcal) \cap \Tr(\Hcal)$.
	\begin{lemma}
		\label{lemma:C-alpha-gamma-form}
		Let $\gamma \in \R, \gamma > 0$ be fixed. Let $C_{\alpha,\gamma}$ be as defined in Eq.\eqref{equation:C-alpha-gamma}. Then 
		$C_{\alpha,\gamma} = \gamma I + A \in \PC_1(\Hcal)$, where
		$A = \gamma (I-B)^{-1/2}B(I-B)^{-1/2}\in \Sym(\Hcal) \cap \Tr(\Hcal)$, $B = (1-\alpha)C_0(C_0+\gamma I)^{-1} + \alpha C_1(C_1+\gamma I)^{-1}\in \Sym(\Hcal) \cap \Tr(\Hcal)$.
		% $A = [(1-\alpha)C_0(C_0+\gamma I)^{-1} +\alpha C_1(C_1+\gamma I)^{-1}][(1-\alpha)(C_0 + \gamma I)^{-1} + \alpha (C_1 + \gamma I)^{-1}]^{-1}\in \Sym(\Hcal) \cap \Tr(\Hcal)$.
	\end{lemma}
	With $C_{\alpha,\gamma} \in \PC_1(\Hcal)$, the quantity $d^1_{\logdet}[(C+\mu I), C_{\alpha,\gamma}]$ is then well-defined
	$\forall (C+\mu I) \in \PC_1(\Hcal)$, so that the following definition is well-justified.
	\begin{definition}
		[\textbf{Regularized geometric Jensen-Shannon divergence between Gaussian measures on Hilbert space}]
		\label{definition:regularized-geometric-JS-Gaussian-Hilbert-space}
		Let
		$\mu_i = \Ncal(m_i, C_i) \in \Gauss(\Hcal)$, $i=0,1$. Let $\gamma>0, \gamma \in \R$ be fixed. Let $0 \leq \alpha \leq 1$ be fixed. Let $C_{\alpha,\gamma}$, $m_{\alpha,\gamma}$ be as defined in Eqs.\eqref{equation:C-alpha-gamma},
		\eqref{equation:m-alpha-gamma}, respectively.
		The regularized geometric Jensen-Shannon divergence between $\mu_0$ and $\mu_1$ is defined to be
		\begin{align}
			&\JS_{G_{\alpha}}^{\gamma}(\mu_0||\mu_1) = 
			\frac{1-\alpha}{2}||C_{\alpha, \gamma}^{-1/2}(m_0 - m_{\alpha,\gamma})||^2 + \frac{\alpha}{2}||C_{\alpha, \gamma}^{-1/2}(m_1 - m_{\alpha,\gamma})||^2 
			\nonumber
			\\
			&\quad +\frac{(1-\alpha)}{2}d^1_{\logdet}[(C_0+\gamma I), C_{\alpha,\gamma}]
			%\nonumber
			%\\
			%&\quad 
			+ \frac{\alpha}{2} d^1_{\logdet}[(C_1+\gamma I), C_{\alpha,\gamma}].
		\end{align}
	\end{definition}
	\begin{proposition}
		\label{proposition:property-regularized-JS-divergence}
		Let $\gamma\in \R, \gamma > 0$ be fixed. Let $0 < \alpha < 1$. Then
		\begin{align}
			&\JS_{G_{\alpha}}^{\gamma}(\mu_0||\mu_1) \geq 0 \;\;\forall \mu_0,\mu_1 \in \Gauss(\Hcal),
			\\
			&\JS_{G_{\alpha}}^{\gamma}(\mu_0||\mu_1) = 0 \equivalent \mu_0 = \mu_1 \equivalent m_0 = m_1, C_0 = C_1.
		\end{align}
	\end{proposition}
	The following result shows that
	%, as with the regularized $\KL$ divergence, 
	for two equivalent
	Gaussian measures $\mu_0 \sim \mu_1$ with zero mean,
	%$\Ncal(m_1,C_1) \sim \Ncal(m_2,C_2)$, then 
	we recover the exact geometric $\JS$ divergence as
	% the regularization parameter 
	$\gamma \approach 0^{+}$.
	% that is
	\begin{theorem}
		[\textbf{Limiting behavior of the regularized geometric Jensen-Shannon divergence}]
		\label{theorem:limit-regularized-JS-Gaussian}
		Let
		$\mu_i = \Ncal(0, C_i) \in \Gauss(\Hcal)$, $i=0,1$. Let $0 \leq \alpha \leq 1$ be fixed. Assume that $\mu_0 \sim \mu_1$. Then
		\begin{align}
			\lim_{\gamma \approach 0^{+}}\JS_{G_{\alpha}}^{\gamma}(\mu_0||\mu_1) = \JS_{G_{\alpha}}(\mu_0||\mu_1).
		\end{align}
	\end{theorem}

\section{Proofs of Main Results}
\label{section:proofs}

We first note the following.
% results.
%\begin{lemma}
%	\label{lemma:logdet-I-S-finite}
%	Let $S \in \SymTr(\Hcal)_{< I}$. Then 
%	$\det(I-S)> 0$ and $\log\det(I-S)$ is well-defined and finite.
%\end{lemma}
%\begin{proof}
By Lemma 3 in \cite{Minh:LogDet2016}, if $S \in \SymTr(\Hcal)_{< I}$, then $\log(I-S) \in \Sym(\Hcal)\cap \Tr(\Hcal)$.
Let $\{\alpha_k\}_{k=1}^{\infty}$ be the  eigenvalues of $S$, then
\begin{align*}
	\left|\sum_{k=1}^{\infty}\log(1-\alpha_k)\right| = |\trace(\log(I-S))| \leq ||\log(I-S)||_{\tr} < \infty. 
\end{align*}
It follows that for $S \in \SymTr(\Hcal)_{< I}$,
\begin{align*}
	\det(I-S) &= \prod_{k=1}^{\infty}(1-\alpha_k) = \exp(\trace(\log(I-S))) > 0,
	\\
	|\log\det(I-S)| &= \left|\sum_{k=1}^{\infty}\log(1-\alpha_k)\right| = |\trace(\log(I-S))|< \infty.
\end{align*}
%	\qed
%\end{proof}
%By Lemma \ref{lemma:logdet-I-S-finite}, 
Thus all expressions of the form $\det(I-S)$ and $\log\det(I-S)$, for $S \in \SymTr(\Hcal)_{<I}$, are well-defined and finite. %Lemma \ref{lemma:logdet-I-S-finite} 
This also implies that for $S \in \SymHS(\Hcal)_{< I}$, we have $\dettwo(I-S) = \det[(I-S)\exp(S)] > 0$, so that $\log\dettwo(I-S)$ is well-defined and finite.
Furthermore, we have the following
\begin{proposition}
	[\cite{Minh:2022Kullback}, Theorem 12]
	\label{proposition:logdet2-I-S-finite}
	Let $S \in \SymHS(\Hcal)_{< I}$. Then
	\begin{align}
		|\log\dettwo(I-S)| = -\log\dettwo(I-S) \leq ||(I-S)^{-1}||\;||S||^2_{\HS}.
	\end{align}
	%	Consequently, 
	%	\begin{align}
		%	\dettwo(I-S) > 0.
		%	\end{align}
\end{proposition}

\subsection{Radon-Nikodym density between two equivalent Gaussian measures}
\label{section:Radon-Nikodym-density-equivalent-Gaussian-Hilbert-space}

In this section, we review the explicit expressions for the Radon-Nikodym density between two equivalent Gaussian measures on 
$\Hcal$. 
These are employed to derive the explicit expression for the geometric interpolation of equivalent Gaussian measures in Theorem \ref{theorem:geometric-interpolation-equivalent-Gaussian-Hilbert-space}.
The following material is based on \cite{Bogachev:Gaussian}, \cite{DaPrato:PDEHilbert}, and 
the exposition in \cite{Minh:2020regularizedDiv} and \cite{Minh2024:FisherRaoGaussian}.

For the expression of the Radon-Nikodym density between two equivalent Gaussian measures on 
$\Hcal$, 
%in the following
we utilize the concept of {\it white noise mapping}, see e.g. \cite{DaPrato:2006,DaPrato:PDEHilbert}.
%In the following, 
For $\mu = \Ncal(m, Q)$, $\ker(Q) = \{0\}$, we  define
%\begin{align}
$\Lcal^2(\Hcal, \mu) = \Lcal^2(\Hcal, \Bsc(\Hcal),\mu) = \Lcal^2(\Hcal, \Bsc(\Hcal), \Ncal(m,Q))$.
%\end{align}
Consider the following mapping
\begin{align}
	&W:Q^{1/2}(\Hcal) \subset \Hcal \mapto \Lcal^2(\Hcal,\mu), \;\; z  \in Q^{1/2}(\Hcal) \mapto W_z \in \Lcal^2(\Hcal, \mu),
	\\
	%\end{align}
	%defined by
	%\begin{align}
	&W_z(x) = \la x -m, Q^{-1/2}z\ra,  \;\;\; z \in Q^{1/2}(\Hcal), x \in \Hcal.
	%W_z(x) = \la x, Q^{-1/2}z\ra,
\end{align}
For any pair $z_1, z_2 \in Q^{1/2}(\Hcal)$, it follows from the definition of covariance operator
that $\la W_{z_1}, W_{z_2}\ra_{\Lcal^2(\Hcal,\mu)} = \la z_1, z_2\ra_{\Hcal}$ 
%\begin{align}
%	\la W_{z_1}, W_{z_2}\ra_{\Lcal^2(\Hcal,\mu)}
% \int_{\H}W_{z_1}(x)W_{z_2}(x)\Ncal(m,Q)(dx)
%\\
%&
%	&= \int_{\Hcal}\la x -m, Q^{-1/2}z_1\ra\la x-m, Q^{-1/2}z_2\ra\Ncal(m, Q)(dx)
%	\nonumber
%	\\
%	& = \la Q(Q^{-1/2}z_1), Q^{-1/2}z_2\ra 
%\nonumber
%\\
%&
%	= \la z_1, z_2\ra_{\Hcal}.
%\end{align}
so that the map $W:Q^{1/2}(\Hcal) \mapto \Lcal^2(\Hcal, \mu)$ is an isometry, that is
%\begin{align}
$||W_z||_{\Lcal^2(\Hcal,\mu)} = ||z||_{\Hcal}$, $z \in Q^{1/2}(\Hcal)$.
%\end{align}
Since $\ker(Q) = \{0\}$, the subspace $Q^{1/2}(\Hcal)$ is dense in $\Hcal$ and the map $W$ can be uniquely extended to all of $\Hcal$, as follows.
For any $z \in \Hcal$, let $\{z_n\}_{n\in \Nbb}$ be a sequence in $Q^{1/2}(\Hcal)$ with $\lim_{n \approach \infty}||z_n -z||_{\Hcal} = 0$.
Then $\{z_n\}_{n \in \Nbb}$ is a Cauchy sequence in $\Hcal$, so that by isometry, $\{W_{z_n}\}_{n\in \Nbb}$ is also
a Cauchy sequence in $\Lcal^2(\Hcal, \mu)$, thus converging to a unique element in $\Lcal^2(\Hcal, \mu)$.
Thus
%for any $z \in \H$, 
we can define
% the map
\begin{align}
	W: \Hcal \mapto \Lcal^2(\Hcal, \mu),  \;\;\; z \in \Hcal \mapto \Lcal^2(\Hcal, \mu)
\end{align}
by the following unique limit in $\Lcal^2(\Hcal, \mu)$
\begin{align}
	W_z(x) = \lim_{n \approach \infty}W_{z_n}(x) = \lim_{n \approach \infty}\la x-m, Q^{-1/2}z_n\ra.
\end{align}
The map $W: \Hcal \mapto \Lcal^2(\Hcal, \mu)$ is called the {\it white noise mapping}
associated with the measure $\mu = \Ncal(m,Q)$.

Let $\mu = \Ncal(m_1, Q)$, $\nu = \Ncal(m_2,R)$ be equivalent, with $R = Q^{1/2}(I-S)Q^{1/2}$, $S \in \SymHS(\Hcal)_{<I}$.
Let $\{\alpha_k\}_{k \in \Nbb}$ be the eigenvalues of $S$, with corresponding orthonormal eigenvectors $\{\phi_k\}_{k \in \Nbb}$, which form an orthonormal basis in $\Hcal$.
The following result expresses the Radon-Nikodym density $\frac{d\nu}{d\mu}$ in terms of the $\alpha_k$'s and
$\phi_k$'s. Here the white noise mapping is $W: \Hcal \mapto \Lcal^2(\Hcal, \mu) = \Lcal^2(\Hcal, \Ncal(m_1,Q))$,
%$W_{\phi_k}$ is defined with respect to the measure $\mu = \Ncal(m_1,Q)$,
with $\{W_{\phi_k}\}_{k=1}^{\infty}$ forming an orthonormal sequence in $\Lcal^2(\Hcal,\mu)$.
%, as in Section

%is the main result in this section.

\begin{theorem}
	[\cite{Minh:2020regularizedDiv}, Theorem 11 and Corollary 2]
	\label{theorem:radon-nikodym-infinite}
	Let $\mu = \Ncal(m_1, Q)$, $\nu = \Ncal(m_2,R)$, with $m_2 - m_1 \in \range(Q^{1/2})$, $R = Q^{1/2}(I-S)Q^{1/2}$, $S \in \SymHS(\Hcal)_{<I}$.
	The Radon-Nikodym density $\frac{d\nu}{d\mu}$
	%of $\nu$ with respect to $\mu$ 
	is given by
	\begin{align}
		\label{equation:RN-infinite}
		\frac{d\nu}{d\mu}(x) = \exp\left[-\frac{1}{2}\sum_{k=1}^{\infty}\Phi_k(x)\right]\exp\left[-\frac{1}{2}||(I-S)^{-1/2}Q^{-1/2}(m_2 - m_1)||^2\right],
		%\left[\frac{\alpha_k}{1-\alpha_k}W^2_{\phi_k}(x-m_1) + \frac{1}{1-\alpha_k}W_{\phi_k}(x-m_1)W_{\phi_k}(m_1 - m_2) + \log(1-\alpha_k)\right]\right\},
	\end{align}
	where for each $k \in \Nbb$
	\begin{align}
		\label{equation:Phik}
		\Phi_k = \frac{\alpha_k}{1-\alpha_k}W^2_{\phi_k} - \frac{2}{1-\alpha_k}\la Q^{-1/2}(m_2-m_1), \phi_k\ra W_{\phi_k}+ \log(1-\alpha_k).
	\end{align}
	The series $\sum_{k=1}^{\infty}\Phi_k$ converges in 
	$\Lcal^2(\Hcal,\mu)$ and
	% the function 
	$s(x) = \exp\left[-\frac{1}{2}\sum_{k=1}^{\infty}\Phi_k(x)\right] \in \Lcal^1(\Hcal, \mu)$.
\end{theorem}
For the case $m_1 = m_2 = 0$, this is Corollary 6.4.11 in \cite{Bogachev:Gaussian}.
In particular, for the case $S \in \Sym(\Hcal) \cap \Tr(\Hcal)$, $I -S > 0$, we have
\begin{align}
	\label{equation:RN-traceclass}
	\frac{d\nu}{d\mu}(x) &= [\det(I-S)]^{-1/2}
	\\
	& \times \exp\left\{-\frac{1}{2}\la Q^{-1/2}(x-m_1), S(I-S)^{-1}Q^{-1/2}(x-m_1)\ra \right\}
	\nonumber
	\\
	&\times \exp(\la Q^{-1/2}(x-m_1), (I-S)^{-1}Q^{-1/2}(m_2 - m_1)\ra)
	\nonumber
	\\
	&\times \exp\left[-\frac{1}{2}||(I-S)^{-1/2}Q^{-1/2}(m_2 - m_1)||^2\right].
	\nonumber
\end{align}
For the case $Q=R \equivalent S=0$, i.e. two measures have the same covariance operator), we have
\begin{align}
\label{equation:RN-same-covariance}
\frac{d\nu}{d\mu}(x) &= \exp(\la Q^{-1/2}(x-m_1), Q^{-1/2}(m_2 - m_1)\ra)
\nonumber
\\
&\quad \times \exp\left[-\frac{1}{2}||Q^{-1/2}(m_2 - m_1)||^2\right].
\end{align}
Let $P_N = \sum_{k=1}^Ne_k \otimes e_k$, $N \in \Nbb$, be the orthogonal projection onto the $N$-dimensional subspace of $\Hcal$ spanned
by $\{e_k\}_{k=1}^N$, where $\{e_k\}_{k=1}^{\infty}$ are the orthonormal eigenvectors of $Q$.
In the above expression,
\begin{align}
	&\la Q^{-1/2}(x-m_1), S(I-S)^{-1}Q^{-1/2}(x-m_1)\ra 
	\nonumber
	\\
	& \doteq \lim_{N \approach \infty} \la Q^{-1/2}P_N(x-m_1), S(I-S)^{-1}Q^{-1/2}P_N(x-m_1)\ra 
	\\
	&
	\la Q^{-1/2}(x-m_1), (I-S)^{-1}Q^{-1/2}(m_2 - m_1)\ra 
	\nonumber
	\\
	&\doteq \lim_{N \approach \infty} \la Q^{-1/2}P_N(x-m_1), (I-S)^{-1}Q^{-1/2}(m_2 - m_1)\ra,
\end{align}
with the limits being in the 
%$\Lcal^1(\Hcal,\mu)$ and 
$\Lcal^2(\Hcal, \mu)$ sense.
% (see Proposition \ref{proposition:RN-S-trace-class-quadratic-form-convergence}).
%, respectively.
For the case $m_1=m_2 = 0$, we obtain the expression given in Proposition 1.3.11 in \cite{DaPrato:PDEHilbert}
\begin{align}
	\label{equation:RN-m1m2-0}
	\frac{d\nu}{d\mu}(x) = [\det(I-S)]^{-1/2}\exp\left\{-\frac{1}{2}\la Q^{-1/2}x, S(I-S)^{-1}Q^{-1/2}x\ra \right\}.
\end{align}
It follows that for $S \in \SymHS(\Hcal)_{<I}$,
\begin{align}
	\log\left\{{\frac{d\nu}{d\mu}(x)}\right\} = -\frac{1}{2}\sum_{k=1}^{\infty}\Phi_k(x)-\frac{1}{2}||(I-S)^{-1/2}Q^{-1/2}(m_2 - m_1)||^2.
\end{align}
In the case $m_1 = m_2 = 0$,
\begin{align}
	\label{equation:log-Radon-Nikodym-S-HS-zero-mean}
	\log\left\{{\frac{d\nu}{d\mu}(x)}\right\} = -\frac{1}{2}\sum_{k=1}^{\infty}\Phi_k(x) = -\frac{1}{2}\sum_{k=1}^{\infty}\left[\frac{\alpha_k}{1-\alpha_k}W^2_{\phi_k}(x) + \log(1-\alpha_k)\right].
\end{align}
In particular, for $S \in \SymTr(\Hcal)_{<I}$, with $m_1 = m_2 = 0$,
\begin{align}
	\label{equation:log-Radon-Nikodym-S-trace-class-zero-mean}
	\log\left\{{\frac{d\nu}{d\mu}(x)}\right\} = - \frac{1}{2}\log\det(I-S) - \frac{1}{2}\la Q^{-1/2}x, S(I-S)^{-1}Q^{-1/2}x\ra.
\end{align}

We first prove the following result, which is the generalization of Proposition 12 
in \cite{Minh2024:FisherRaoGaussian}, which treats zero-mean Gaussian measures on $\Hcal$, to the general Gaussian setting
involving both means and covariance operators.

\begin{proposition}
	\label{proposition:inner-logRN-mu-star}
	Let $S_{\nu} \in \SymHS(\Hcal)_{<I}$ be fixed but arbitrary. 
	Let $\nu = \Ncal(m_{\nu}, C_{\nu}) \in \Gauss(\Hcal, \mu_{*})$, where $m_{\nu} \in \Hcal$, $C_{\nu} = C_{*}^{1/2}(I-S_{\nu})C_{*}^{1/2}$.
	Let $S, S_1,S_2 \in \SymHS(\Hcal)_{< I}$. Let
	$\mu = \Ncal(m,C)$, $\mu_i = \Ncal(m_i, C_i) \in \Gauss(\Hcal, \mu_{*})$, $C= C_{*}^{1/2}(I-S)C_{*}^{1/2}$, $C_i = C_{*}^{1/2}(I-S_i)C_{*}^{1/2}$, $i=1,2$.
	Then, with $u_{\nu} = C_{*}^{-1/2}(m_{\nu} - m_{*})$, $u_i = C_{*}^{-1/2}(m_i - m_{*})$, $i=1,2$,
	\begin{align}
		&\left\la \log\left\{{\frac{d\mu_1}{d\mu_{*}}(x)}\right\}, \log\left\{{\frac{d\mu_2}{d\mu_{*}}(x)}\right\}\right\ra_{\Lcal^2(\Hcal,\nu)} 
		\\
		&
		= \frac{1}{2}\trace[(I-S_{\nu})S_1(I-S_1)^{-1}(I-S_{\nu})S_2(I-S_2)^{-1}]
		\nonumber
		%\\
		%&\quad + \la S_1(I-S_1)^{-1}u_{\nu}, (I-S_{\nu})S_2(I-S_2)^{-1}u_{\nu}\ra
		%\nonumber
		\\
		& \quad +\left\la (I-S_{\nu})^{1/2}(I-S_1)^{-1}u_1 - (I-S_{\nu})^{1/2}S_1(I-S_1)^{-1}u_{\nu}, \right.
		\nonumber
		\\
		&\quad \quad \quad \left.(I-S_{\nu})^{1/2}(I-S_2)^{-1}u_2 - (I-S_{\nu})^{1/2}S_2(I-S_2)^{-1}u_{\nu}\right\ra
		\nonumber
		\\
		& \quad + \frac{1}{4}\left[\la u_{\nu}, S_1(I-S_1)^{-1}u_{\nu}\ra-2\la (I-S_1)^{-1}u_1, u_{\nu}\ra \right.
		\nonumber
		\\
		&\quad \quad \quad \left. -  \trace(S_{\nu}S_1(I-S_1)^{-1})-	\log\dettwo[(I-S_1)^{-1} \right]
		\nonumber
		\\
		& \quad \quad \times \left[\la u_{\nu}, S_2(I-S_2)^{-1}u_{\nu}\ra-2\la (I-S_2)^{-1}u_2, u_{\nu}\ra \right.
		\nonumber
		\\
		&\quad \quad \quad \left. - \trace(S_{\nu}S_2(I-S_2)^{-1})-	\log\dettwo[(I-S_2)^{-1}\right].
		\nonumber
	\end{align} 
	In particular, for $\mu_1 = \mu_2 = \mu = \Ncal(m,C)$, $m_1 = m_2 = m$, $u_1 = u_2 = u = C_{*}^{-1/2}(m-m_{*})$, $S_1 = S_2 = S$,
	\begin{align}
		&\left\|\log\left\{\frac{d\mu}{d\mu_{*}}(x)\right\}\right\|^2_{\Lcal^2(\Hcal,\nu)} 
		\\
			&
		= \frac{1}{2}||(I-S_{\nu})^{1/2}S(I-S)^{-1}(I-S_{\nu})^{1/2}||^2_{\HS}
		\nonumber
		%\\
		%&\quad + \la S_1(I-S_1)^{-1}u_{\nu}, (I-S_{\nu})S_2(I-S_2)^{-1}u_{\nu}\ra
		%\nonumber
		\\
		& \quad +||(I-S_{\nu})^{1/2}(I-S)^{-1}u - (I-S_{\nu})^{1/2}S(I-S)^{-1}u_{\nu}||^2
		\nonumber
		\\
		& \quad + \frac{1}{4}\left[\la u_{\nu}, S(I-S)^{-1}u_{\nu}\ra-2\la (I-S)^{-1}u, u_{\nu}\ra \right.
		\nonumber
		\\
		&\quad \quad \quad \left. -  \trace(S_{\nu}S(I-S)^{-1})-	\log\dettwo[(I-S)^{-1} \right]^2.
		\nonumber
	\end{align}
	Consequently,
	\begin{align}
		&\left\| \log\left\{{\frac{d\mu_1}{d\mu_{*}}(x)}\right\}- \log\left\{{\frac{d\mu_2}{d\mu_{*}}(x)}\right\}\right\|^2_{\Lcal^2(\Hcal,\nu)}
		\\
		&= \frac{1}{2}||(I-S_{\nu})^{1/2}[S_1(I-S_1)^{-1}-S_2(I-S_2)^{-1}](I-S_{\nu})^{1/2}||^2_{\HS}
		\nonumber
		\\
		& \quad + ||(I-S_{\nu})^{1/2}[(I-S_1)^{-1}u_1 - (I-S_2)^{-1}u_2] 
		\nonumber
		\\
		&\quad \quad -(I-S_{\nu})^{1/2}[S_1(I-S_1)^{-1}-S_2(I-S_2)^{-1}]u_{\nu}||^2
		\nonumber
		\\
		& \quad + \frac{1}{4}\left(\log\dettwo(I-S_1)^{-1}- \log\dettwo(I-S_2)^{-1} \right.
		\nonumber
		\\
		&\quad \quad+ \trace[S_{\nu}[S_1(I-S_1)^{-1}-S_2(I-S_2)^{-1}]]
		\nonumber
		\\
		&\quad \quad - \la u_{\nu}, [S_1(I-S_1)-S_2(I-S_2)^{-1}]u_{\nu}\ra 
		\nonumber
		\\
		& \quad \quad \left.+ 2 \la (I-S_1)^{-1}u_1 - (I-S_2)^{-1}u_2, u_{\nu}\ra\right)^2.
		\nonumber
	\end{align}
	Furthermore,
	\begin{align}
		\label{equation:upperbound-logRN-L2}
		&\left\| \log\left\{{\frac{d\mu_1}{d\mu_{*}}(x)}\right\}- \log\left\{{\frac{d\mu_2}{d\mu_{*}}(x)}\right\}\right\|^2_{\Lcal^2(\Hcal,\nu)}
		%\nonumber
		\\
		&\leq \left(\frac{1}{2}||I-S_{\nu}||^2 + 2||I-S_{\nu}||\;||u_{\nu}||^2 + [||S_1||_{\HS} + ||S_2||_{\HS} + ||S_1S_2||_{\HS}]^2\right. 
		\nonumber
		\\
		&\quad \quad \left.+ ||S_{\nu}||_{\HS}^2 + ||u_{\nu}||^4 + 8||u_1||^2||u_{\nu}||^2 + 4||u_1||^2\right)
		\nonumber
		\\
		& \quad \quad \times ||(I-S_1)^{-1}||^2||(I-S_2)^{-1}||^2||S_1-S_2||^2_{\HS}
		\nonumber
		\\
		& \quad +[4||I-S_{\nu}||+8||u_{\nu}||^2] ||(I-S_2)^{-1}||^2||u_1 - u_2||^2. 
		\nonumber
	\end{align}
\end{proposition}

From the inequality in \eqref{equation:upperbound-logRN-L2}, we immediately obtain the following.
\begin{corollary}
	\label{corollary:log-Radon-Nikodym-L2-convergence}
	Let $\mu = \Ncal(m,C)$, $\{\mu_k = \Ncal(m_k,C_k)\}_{k \in \Nbb} \in \Gauss(\Hcal,\mu_{*})$, 
	with $C = C_{*}^{1/2}(I-S)C_{*}^{1/2}$, 
	$C_k = C_{*}^{1/2}(I-S_k)C_{*}^{1/2}$, $S,S_k \in \SymHS(\Hcal)_{< I}$, $k \in \Nbb$. Assume that $\lim_{k \approach \infty}||m_k - m|| = 0$, $\lim_{k \approach \infty}||S_k - S||_{\HS} = 0$. Let $\nu \in \Gauss(\Hcal, \mu_{*})$ be fixed but arbitrary.
	Then 
	\begin{align}
		\lim_{k \approach \infty}\left\| \log\left\{{\frac{d\mu_k}{d\mu_{*}}(x)}\right\}- \log\left\{{\frac{d\mu}{d\mu_{*}}(x)}\right\}\right\|_{\Lcal^2(\Hcal,\nu)} = 
		0.
		%\lim_{k \approach \infty}\left\| \log\left\{{\frac{d\mu_k}{d\mu}(x)}\right\}\right\|_{\Lcal^2(\Hcal,\nu)} = 0.
	\end{align}
	In particular, with $\nu = \mu$, 
	\begin{align}
		\lim_{k \approach \infty}\left\| \log\left\{{\frac{d\mu_k}{d\mu_{*}}(x)}\right\}- \log\left\{{\frac{d\mu}{d\mu_{*}}(x)}\right\}\right\|_{\Lcal^2(\Hcal,\mu)} = 
		0.
		%\lim_{k \approach \infty}\left\| \log\left\{{\frac{d\mu_k}{d\mu}(x)}\right\}\right\|_{\Lcal^2(\Hcal,\mu)} = 0.
	\end{align}
	%	Consequently,
	%	\begin{align}
		%		\lim_{k \approach \infty}\left\| \left\{{\frac{d\mu_k}{d\mu_{*}}(x)}\right\}- \left\{{\frac{d\mu}{d\mu_{*}}(x)}\right\}\right\|^2_{\Lcal^2(\Hcal,\mu_{*})} = 0
		%	\end{align}
\end{corollary}	

For the proof of Proposition \ref{proposition:inner-logRN-mu-star}, we need the following technical results.

\begin{lemma}
	[\cite{Minh2024:FisherRaoGaussian}, Lemma 22]
	\label{lemma:HS-inner-expansion}
	Let $S,T \in \Sym(\Hcal) \cap \HS(\Hcal)$, with eigenvalues $\{\alpha_k\}_{k \in \Nbb}$, $\{\beta_k\}_{k \in \Nbb}$
	and corresponding orthonormal eigenvectors $\{\phi_k\}_{k \in \Nbb}$, $\{\psi_k\}_{k \in \Nbb}$, respectively.
	Let $A \in \Sym(\Hcal)$.
	Then
	\begin{align}
		%\la A^{1/2}SA^{1/2},A^{1/2}TA^{1/2}\ra_{\HS} = 
		\trace(ASAT) =\trace(SATA) = \sum_{k,j=1}^{\infty}\alpha_k \beta_j \la A\phi_k, \psi_j\ra^2.
	\end{align}
	If, in addition, $A \in \Sym(\Hcal)\cap \HS(\Hcal)$, then $\trace(AS) = \sum_{k=1}^{\infty}\alpha_k\la \phi_k, A\phi_k\ra$.
\end{lemma}

\begin{lemma}
	\label{lemma:inner-UV-expansion}
	Let $S,T \in \Sym(\Hcal)$ be compact, with eigenvalues $\{\alpha_k\}_{k \in \Nbb}$, $\{\beta_k\}_{k \in \Nbb}$
	and corresponding orthonormal eigenvectors $\{\phi_k\}_{k \in \Nbb}$, $\{\psi_k\}_{k \in \Nbb}$, respectively.
	Let $A \in \Sym(\Hcal)$.
	Then $\forall u,v \in \Hcal$,
	\begin{align}
		\la Su, ATv\ra = \sum_{k,j=1}^{\infty}\alpha_k\beta_j \la u,\phi_k\ra\la v, \psi_j\ra \la \phi_k, A\psi_j\ra.
	\end{align}
	%	Let $A \in \Sym(\Hcal)$.
\end{lemma}
\begin{proof} With $\{\phi_k\}_{k \in \Nbb}$, $\{\psi_k\}_{k \in \Nbb}$ being both orthonormal bases in $\Hcal$,
	\begin{align*}
		&\sum_{k,j=1}^{\infty}\alpha_k\beta_j \la u,\phi_k\ra\la v, \psi_j\ra \la \phi_k, A\psi_j\ra = \sum_{k=1}^{\infty}
		\alpha_k \la u,\phi_k\ra\sum_{j=1}^{\infty}\la v, T\psi_j\ra \la A\phi_k, \psi_j\ra
		\\
		&= \sum_{k=1}^{\infty}
		\la u,S\phi_k\ra\la Tv, A\phi_k\ra = \sum_{k=1}^{\infty}
		\la Su,\phi_k\ra\la ATv, \phi_k\ra = \la Su, ATv\ra.
	\end{align*}
	\qed
\end{proof}

\begin{lemma}
	\label{lemma:quadratic-form-expansion}
	Let $S \in \Sym(\Hcal)$ be compact, with eigenvalues $\{\alpha_k\}_{k \in \Nbb}$
	and corresponding orthonormal eigenvectors $\{\phi_k\}_{k \in \Nbb}$. Then $\forall u \in \Hcal$,
	\begin{align}
		\la u, Su\ra = \sum_{k=1}^{\infty}\alpha_k \la u, \phi_k\ra^2.
	\end{align}
\end{lemma}
\begin{proof} With $\{\phi_k\}_{k \in \Nbb}$ being an orthonormal basis in $\Hcal$,
	\begin{align*}
		\sum_{k=1}^{\infty}\alpha_k \la u, \phi_k\ra^2 = \sum_{k=1}^{\infty}\la u, S\phi_k\ra \la u, \phi_k \ra= \sum_{k=1}^{\infty}\la Su, \phi_k\ra \la u, \phi_k\ra = \la u, Su\ra.
	\end{align*}
	\qed
\end{proof}

\begin{lemma}[\cite{Minh2024:FisherRaoGaussian}, Lemma 23]
	\label{lemma:logdet-2-inverse}
	Let $A \in \SymHS(\Hcal)_{< I}$ with eigenvalues $\{\lambda_k\}_{k=1}^{\infty}$. Then
	\begin{align}
		\log\dettwo[(I-A)^{-1}] &= -\sum_{k=1}^{\infty}\left[\frac{\alpha_k}{1-\alpha_k} + \log(1-\alpha_k)\right]
		\\
		&= - [\log\dettwo(I-A) + \trace(A^2(I-A)^{-1})].
	\end{align}
	In particular, for $A \in \SymTr(\Hcal)_{< I}$, $\log\dettwo[(I-A)^{-1}] = -\log\det(I-A) - \trace[A(I-A)^{-1}]$.
\end{lemma}

\begin{lemma}
	[\cite{Minh:2020regularizedDiv}, Lemma 24]
	\label{lemma:integral-a2b-Gaussian-Hilbert}
	Let $\Ncal(m,Q) \in \Gauss(\Hcal)$ be fixed but arbitrary. 
	For any pair $a,b \in \Hcal$,
	\begin{align}
		\int_{\Hcal}\la x-m, a\ra^2\la x-m, b\ra \Ncal(m,Q)(dx) = 0.
	\end{align}
\end{lemma}

\begin{lemma}
	\label{lemma:integral-Wa2Wb-Gaussian-Hilbert}
	Let $\nu = \Ncal(m_{\nu}, C_{\nu}) \in \Gauss(\Hcal, \mu_{*})$, $m_{\nu} \in \Hcal$, $C_{\nu} = C_{*}^{1/2}(I-S_{\nu})C_{*}^{1/2}$,
	$S_{\nu} \in \SymHS(\Hcal)_{< I}$. Then for any $a,b \in \Hcal$,
	\begin{align}
		\int_{\Hcal}W_a^2(x)W_b(x)d\nu(x)& = \la  (I-S_{\nu})a, a\ra\la C_{*}^{-1/2}(m_{\nu}-m_{*}), b\ra 
		\nonumber
		\\
		& \quad+  2\la C_{*}^{-1/2}(m_{\nu} - m_{*}), a\ra\la (I-S_{\nu})a, b\ra 
		\nonumber
		\\
		& \quad + \la C_{*}^{-1/2}(m_{\nu} - m_{*}), a\ra^2\la C_{*}^{-1/2}(m_{\nu}-m_{*}), b\ra.
	\end{align}
\end{lemma}
\begin{proof}
	For $a,b \in \range(C_{*}^{1/2})$, we have
	\begin{align*}
		&J = \int_{\Hcal}W_a^2(x)W_b(x)d\nu(x) 
		\\
		&= \int_{\Hcal}\la x-m_{*}, C_{*}^{-1/2}a\ra^2 \la x-m_{*}, C_{*}^{-1/2}b\ra\Ncal(m_{\nu},C_{\nu})(dx)
		\\
		& = \int_{\Hcal}\la x-m_{\nu} + m_{\nu} - m_{*}, C_{*}^{-1/2}a\ra^2 \la x-m_{\nu} + m_{\nu} - m_{*},C_{*}^{-1/2}b\ra \Ncal(m_{\nu},C_{\nu})(dx)
		\\
		& = \int_{\Hcal}[\la x- m_{\nu}, C_{*}^{-1/2}a\ra^2 + 2 \la x-m_{\nu}, C_{*}^{-1/2}a\ra\la m_{\nu} - m_{*}, C_{*}^{-1/2}a\ra 
		%\\
		%+ \la m_{\nu} - m_{*}, C_{*}^{-1/2}a\ra^2]
		\\
		&+ \la m_{\nu} - m_{*}, C_{*}^{-1/2}a\ra^2][\la x-m_{\nu}, C_{*}^{-1/2}b\ra + \la m_{\nu}-m_{*}, C_{*}^{-1/2}b\ra] \Ncal(m_{\nu},C_{\nu})(dx)
		\\
		& = \int_{\Hcal}\la x- m_{\nu}, C_{*}^{-1/2}a\ra^2\la x-m_{\nu}, C_{*}^{-1/2}b\ra \Ncal(m_{\nu},C_{\nu})(dx)
		\\
		&+ \la m_{\nu}-m_{*}, C_{*}^{-1/2}b\ra\int_{\Hcal}\la x- m_{\nu}, C_{*}^{-1/2}a\ra^2  \Ncal(m_{\nu},C_{\nu})(dx)
		\\
		&+ 2\la m_{\nu} - m_{*}, C_{*}^{-1/2}a\ra\int_{\Hcal}\la x-m_{\nu}, C_{*}^{-1/2}a\ra \la x-m_{\nu}, C_{*}^{-1/2}b\ra\Ncal(m_{\nu},C_{\nu})(dx)
		\\
		&+ 2\la m_{\nu} - m_{*}, C_{*}^{-1/2}a\ra\la m_{\nu}-m_{*}, C_{*}^{-1/2}b\ra\int_{\Hcal} \la x-m_{\nu}, C_{*}^{-1/2}a\ra \Ncal(m_{\nu},C_{\nu})(dx)
		\\
		&+ \la m_{\nu} - m_{*}, C_{*}^{-1/2}a\ra^2\int_{\Hcal}\la x-m_{\nu}, C_{*}^{-1/2}b\ra\Ncal(m_{\nu},C_{\nu})(dx)
		\\
		&+\la m_{\nu} - m_{*}, C_{*}^{-1/2}a\ra^2\la m_{\nu}-m_{*}, C_{*}^{-1/2}b\ra.
	\end{align*}
	With $C_{\nu} = C_{*}^{1/2}(I-S_{\nu})C_{*}^{1/2}$, using Lemma \ref{lemma:integral-a2b-Gaussian-Hilbert}, we obtain
	\begin{align*}
		J &= \la C_{*}^{-1/2}(m_{\nu}-m_{*}), b\ra \la  C_{*}^{1/2}(I-S_{\nu})C_{*}^{1/2}(C_{*}^{-1/2}a), C_{*}^{-1/2}a\ra
		\\
		&\quad + 2\la C_{*}^{-1/2}(m_{\nu} - m_{*}), a\ra \la C_{*}^{1/2}(I-S_{\nu})C_{*}^{1/2}(C_{*}^{-1/2}a), C_{*}^{-1/2}b\ra 
		\\
		&\quad + \la C_{*}^{-1/2}(m_{\nu} - m_{*}), a\ra^2\la C_{*}^{-1/2}(m_{\nu}-m_{*}), b\ra
		\\
		& = \la  (I-S_{\nu})a, a\ra\la C_{*}^{-1/2}(m_{\nu}-m_{*}), b\ra  +  2\la C_{*}^{-1/2}(m_{\nu} - m_{*}), a\ra\la (I-S_{\nu})a, b\ra 
		\\
		& \quad + \la C_{*}^{-1/2}(m_{\nu} - m_{*}), a\ra^2\la C_{*}^{-1/2}(m_{\nu}-m_{*}), b\ra.
	\end{align*}
	Since $\range(C_{*}^{1/2})$ is dense in $\Hcal$,
	the result for the general case $a,b \in \Hcal$ is then obtained via a limiting argument.
	\qed
\end{proof}

\begin{proof}[\textbf{of Proposition \ref{proposition:inner-logRN-mu-star}}]
	Let $\{\alpha_k\}_{k \in \Nbb}$ and $\{\beta_k\}_{k \in \Nbb}$ be the eigenvalues of $S_1$ and $S_2$, with corresponding orthonormal eigenvectors
	$\{\phi_k\}_{k \in \Nbb}$ and $\{\psi_k\}_{k \in \Nbb}$, respectively.
	Consider the white noise mapping $W: \Hcal \mapto \Lcal^2(\Hcal, \mu_{*})$. Let $u_1 = C_{*}^{-1/2}(m_1-m_{*})$,
	$u_2 =  C_{*}^{-1/2}(m_1-m_{*})$. Then
	the Radon-Nikodym densities $\frac{d\mu_1}{d\mu_{*}}$ and $\frac{d\mu_2}{d\mu_{*}}$ are given by
	\begin{align*}
		&\log\left\{{\frac{d\mu_1}{d\mu_{*}}(x)}\right\} = 
		%-\frac{1}{2}||(I-S)^{-1/2}C_{*}^{-1/2}(m_1 - m_{*})||^2
		%\\
		%& \quad 
		-\frac{1}{2}\sum_{k=1}^{\infty}\left[\frac{\alpha_k}{1-\alpha_k}W^2_{\phi_k}(x)
		 - \frac{2}{1-\alpha_k}\la u_1, \phi_k\ra W_{\phi_k} 
		+\log(1-\alpha_k)\right].
		\\
		&\log\left\{{\frac{d\mu_2}{d\mu_{*}}(x)}\right\} =
		%-\frac{1}{2}||(I-T)^{-1/2}C_{*}^{-1/2}(m_2 - m_{*})||^2
		%\\
		%&\quad 
		-\frac{1}{2}\sum_{k=1}^{\infty}\left[\frac{\beta_k}{1-\beta_k}W^2_{\psi_k}(x) 
		-\frac{2}{1-\beta_k}\la u_2, \psi_k\ra W_{\psi_k}
		+ \log(1-\beta_k)\right].
	\end{align*}
	By Lemma 19 in \cite{Minh:2020regularizedDiv}, with $u_{\nu} = C_{*}^{-1/2}(m_{\nu} - m_{*})$, for any $a,b \in \Hcal$,
	\begin{align*}
		&\int_{\Hcal}W_a(x)d\nu(x) = \la u_{\nu}, a\ra,
		\\
		&\int_{\Hcal}W_a(x)W_b(x)d\nu(x) = \la a, (I-S_{\nu})b\ra + \la u_{\nu}, a\ra \la u_{\nu}, b\ra,
		\\
		&\int_{\Hcal}W^2_{a}(x)d\nu(x) = \la a, (I-S_{\nu})a\ra + \la u_{\nu}, a\ra^2.
		%\la C_{*}^{-1/2}(m_{\nu} - m_{*}), a\ra^2.
	\end{align*}
	By Lemma 22 in \cite{Minh:2020regularizedDiv}, for any pair $a,b \in \Hcal$, with $u_{\nu} = C_{*}^{-1/2}(m_{\nu} - m_{*})$,
	\begin{align*}
		\int_{\Hcal}W^2_a(x)W^2_b(x)d\nu(x) &=\la a, (I-S_{\nu})a\ra \la b, (I-S_{\nu})b\ra + 2 \la a,(I-S_{\nu})b\ra^2
		\\
		& \quad + \la u_{\nu}, b\ra^2 \la a, (I-S_{\nu})a\ra
		%\\
		%& 
		+ 4\la u_{\nu}, a\ra \la u_{\nu}, b\ra\la a, (I-S_{\nu})b\ra
		\\
		& \quad + \la u_{\nu}, a\ra^2 \la b, (I-S_{\nu})b\ra
		%\\
		%& 
		+ \la u_{\nu}, a\ra^2\la u_{\nu}, b\ra^2.
	\end{align*}
	By Lemma \ref{lemma:integral-Wa2Wb-Gaussian-Hilbert}, for any pair $a,b \in \Hcal$, with $u_{\nu} = C_{*}^{-1/2}(m_{\nu} - m_{*})$,
	\begin{align*}
		\int_{\Hcal}W_a^2(x)W_b(x)d\nu(x)& = \la  (I-S_{\nu})a, a\ra\la u_{\nu}, b\ra 
		+  2\la u_{\nu}, a\ra\la (I-S_{\nu})a, b\ra 
		+ \la u_{\nu}, a\ra^2\la u_{\nu}, b\ra.
	\end{align*}
	%Recall that $\int_{\Hcal}W^2_{\phi_k}d\mu_0(x) = ||\phi_k||^2 = 1$, 
	%$\int_{\Hcal}W^2_{\psi_k}d\mu_0(x) = ||\psi_k||^2 = 1$. 
	For each pair $k, j \in \Nbb$, with $u_i = C_{*}^{-1/2}(m_i - m_{*})$, $i=1,2$,
	\begin{align*}
		J_{kj} &= \int_{\Hcal}\left[\frac{\alpha_k}{1-\alpha_k}W^2_{\phi_k}(x)  - \frac{2}{1-\alpha_k}\la u_1, \phi_k\ra W_{\phi_k}(x) + \log(1-\alpha_k)\right]
		\\
		& \quad \quad \quad \times \left[\frac{\beta_j}{1-\beta_j}W^2_{\psi_j}(x) - \frac{2}{1-\beta_j}\la u_2, \psi_j\ra W_{\psi_j}(x)+ \log(1-\beta_j)\right]d\nu(x)
		\\
		& = J_{kj,1} + J_{kj,2},
	\end{align*}
	where
	\begin{align*}
	J_{kj,1} &= \int_{\Hcal}\left[\frac{\alpha_k}{1-\alpha_k}W^2_{\phi_k}(x)   + \log(1-\alpha_k)\right]
	%\\
	%& \quad \quad \quad \times 
	\left[\frac{\beta_j}{1-\beta_j}W^2_{\psi_j}(x) + \log(1-\beta_j)\right]d\nu(x),
	\\
	J_{kj,2}&=-\frac{2}{1-\beta_j}\int_{\Hcal}\left[\frac{\alpha_k}{1-\alpha_k}W^2_{\phi_k}(x)+ \log(1-\alpha_k)\right]\la u_2, \psi_j\ra W_{\psi_j}(x)d\nu(x)
	\\
	&\quad + \frac{4}{(1-\alpha_k)(1-\beta_j)}\int_{\Hcal}\la u_1, \phi_k\ra W_{\phi_k}(x)\la u_2, \psi_j\ra W_{\psi_j}(x)d\nu(x)
	\\
	& \quad - \frac{2}{1-\alpha_k} \int_{\Hcal}\left[\frac{\beta_j}{1-\beta_j}W^2_{\psi_j}(x) + \log(1-\beta_j)\right]\la u_1, \phi_k\ra W_{\phi_k}d\nu(x).
	\end{align*}
	For each fixed pair $N,M \in \Nbb$, define 
	\begin{align*}
		f &= \frac{1}{2}\sum_{k=1}^{\infty}\left[\frac{\alpha_k}{1-\alpha_k}W^2_{\phi_k}(x) -\frac{2}{1-\alpha_k}\la u_1, \phi_k\ra W_{\phi_k}+ \log(1-\alpha_k)\right],
		\\
		f_N &= \frac{1}{2}\sum_{k=1}^{N}\left[\frac{\alpha_k}{1-\alpha_k}W^2_{\phi_k}(x) -\frac{2}{1-\alpha_k}\la u_1, \phi_k\ra W_{\phi_k}+ \log(1-\alpha_k)\right],
		\\
		g &= \frac{1}{2}\sum_{k=1}^{\infty}\left[\frac{\beta_k}{1-\beta_k}W^2_{\psi_k}(x) -\frac{2}{1-\beta_k}\la u_2, \psi_k\ra W_{\psi_k}+ \log(1-\beta_k)\right],\;
		\\
		g_M &= \frac{1}{2}\sum_{k=1}^{M}\left[\frac{\beta_k}{1-\beta_k}W^2_{\psi_k}(x) -\frac{2}{1-\beta_k}\la u_2, \psi_k\ra W_{\psi_k}+ \log(1-\beta_k)\right].
	\end{align*}
	By Propositions 8 and 10 in \cite{Minh:2020regularizedDiv}, $f,g \in \Lcal^2(\Hcal,\nu)$, with $\lim\limits_{N \approach \infty}||f_N - f||_{\Lcal^2(\Hcal,\nu)} = 0$, $\lim\limits_{M \approach \infty}||g_M - g||_{\Lcal^2(\Hcal,\nu)} = 0$. By H\"older's inequality,
	$\lim\limits_{N,M \approach \infty}||f_Ng_M - fg||_{\Lcal^1(\Hcal,\nu)} = 0$.
	It follows that
	\begin{align*}
		J& = \left\la \log\left\{{\frac{d\mu_1}{d\mu_{*}}(x)}\right\}, \log\left\{{\frac{d\mu_2}{d\mu_{*}}(x)}\right\}\right\ra_{\Lcal^2(\Hcal,\nu)}
		%\\
		= \int_{\Hcal}f(x)g(x)\nu(dx) 
		\\
		&= \lim_{N, M \approach \infty}\int_{\Hcal}f_N(x)g_M(x)\nu(dx) = \frac{1}{4}\sum_{k,j=1}^{\infty}J_{kj}
		= \frac{1}{4}\sum_{k,j=1}^{\infty}[J_{kj,1} + J_{kj,2}].
	\end{align*}
	Let us now calculate $J_{kj,1}$ and $J_{kj,2}$. For each pair $k,j \in \Nbb$,
	\begin{align*}
		&J_{kj,1} = \int_{\Hcal}\left[\frac{\alpha_k}{1-\alpha_k}W^2_{\phi_k}(x)   + \log(1-\alpha_k)\right]
		%\\
		%& \quad \quad \quad \times 
		\left[\frac{\beta_j}{1-\beta_j}W^2_{\psi_j}(x) + \log(1-\beta_j)\right]d\nu(x)
		\\
		&= \frac{\alpha_k}{1-\alpha_k}\frac{\beta_j}{1-\beta_j}[\la \phi_k, (I-S_{\nu})\phi_k\ra \la \psi_j, (I-S_{\nu})\psi_j\ra + 2 \la \phi_k,(I-S_{\nu})\psi_j\ra^2] 
		\\
		& \quad + \frac{\alpha_k}{1-\alpha_k}\frac{\beta_j}{1-\beta_j}\la u_{\nu}, \psi_j\ra^2 \la \phi_k, (I-S_{\nu})\phi_k\ra
		\\
		& \quad + 4\frac{\alpha_k}{1-\alpha_k}\frac{\beta_j}{1-\beta_j}\la u_{\nu}, \phi_k\ra\la u_{\nu}, \psi_j\ra\la \phi_k, (I-S_{\nu})\psi_j\ra
		\\
		& \quad + \frac{\alpha_k}{1-\alpha_k}\frac{\beta_j}{1-\beta_j}\la u_{\nu}, \phi_k\ra^2 \la \psi_j, (I-S_{\nu})\psi_j\ra
		%\\
		%& \quad 
		+ \frac{\alpha_k}{1-\alpha_k}\frac{\beta_j}{1-\beta_j}\la u_{\nu}, \phi_k\ra^2\la u_{\nu}, \psi_j\ra^2
		\\
		&\quad + \frac{\alpha_k}{1-\alpha_k}\log(1-\beta_j)[\la \phi_k, (I-S_{\nu})\phi_k\ra + \la u_{\nu}, \phi_k\ra^2]
		\\
		&\quad + \frac{\beta_j}{1-\beta_j}\log(1-\alpha_k) [\la \psi_j, (I-S_{\nu})\psi_j\ra + \la u_{\nu}, \psi_j\ra^2]
		%\\
		%&\quad 
		+ \log(1-\alpha_k)\log(1-\beta_j).
	\end{align*}
	Rearranging the terms in $J_{kj,1}$ gives, with $u_{\nu} = C_{*}^{-1/2}(m_{\nu} - m_{*})$,
\begin{align*}
	J_{kj,1} &= 2\frac{\alpha_k}{1-\alpha_k}\frac{\beta_j}{1-\beta_j}[\la \phi_k, (I-S_{\nu})\psi_j\ra^2] 
	\\
	& \quad + 4\frac{\alpha_k}{1-\alpha_k}\frac{\beta_j}{1-\beta_j}\la u_{\nu}, \phi_k\ra\la u_{\nu}, \psi_j\ra\la \phi_k, (I-S_{\nu})\psi_j\ra
	\\
	&\quad + \left[\frac{\alpha_k}{1-\alpha_k}(\la \phi_k, (I-S_{\nu})\phi_k\ra + \la u_{\nu}, \phi_k\ra^2) + \log(1-\alpha_k)\right]
	\\
	& \quad \times 
	\left[\frac{\beta_j}{1-\beta_j}(\la \psi_j, (I-S_{\nu})\psi_j\ra + \la u_{\nu}, \psi_j\ra^2)+ \log(1-\beta_j)\right].
\end{align*}
Equivalently,
\begin{align*}
	J_{kj,1}	& = 2\frac{\alpha_k}{1-\alpha_k}\frac{\beta_j}{1-\beta_j}[\la \phi_k, (I-S_{\nu})\psi_j\ra^2] 
	\\
	& \quad + 4\frac{\alpha_k}{1-\alpha_k}\frac{\beta_j}{1-\beta_j}\la u_{\nu}, \phi_k\ra\la u_{\nu}, \psi_j\ra\la \phi_k, (I-S_{\nu})\psi_j\ra
	\\
	&\quad + \left[- \frac{\alpha_k}{1-\alpha_k}\la \phi_k, S_{\nu}\phi_k\ra + \frac{\alpha_k}{1-\alpha_k}\la u_{\nu}, \phi_k\ra^2+  \frac{\alpha_k}{1-\alpha_k} + \log(1-\alpha_k)\right]
	\\
	& \quad \times \left[- \frac{\beta_j}{1-\beta_j}\la \psi_j, S_{\nu}\psi_j\ra + \frac{\beta_j}{1-\beta_j}\la u_{\nu}, \psi_j\ra^2+ \frac{\beta_j}{1-\beta_j} + \log(1-\beta_j)\right].
\end{align*}
	%\begin{align*}
	%\left[\frac{\alpha_k}{1-\alpha_k}\right]\left[\frac{\beta_j}{1-\beta_j}\right]
	%\end{align*}
Let us sum the terms in $J_{kj,1}$ separately. By Lemma \ref{lemma:HS-inner-expansion},
		\begin{align}
		&\sum_{k,j=1}^{\infty}\frac{\alpha_k}{1-\alpha_k}\frac{\beta_j}{1-\beta_j}[\la \phi_k, (I-S_{\nu})\psi_j\ra^2] 
		\nonumber
		\\
		&= \trace[(I-S_{\nu})S_1(I-S_1)^{-1}(I-S_{\nu})S_2(I-S_2)^{-1}].
		\label{equation:Jkj-sum-1-1}
		\\
		& \sum_{k=1}^{\infty}\frac{\alpha_k}{1-\alpha_k}\la \phi_k, S_{\nu}\phi_k\ra = \trace(S_{\nu}S_1(I-S_1)^{-1}),
		\label{equation:Jkj-sum-1-2}
		\\
		& \sum_{j=1}^{\infty}\frac{\beta_j}{1-\beta_j}\la \psi_j, S_{\nu}\psi_j\ra  = \trace(S_{\nu}S_2(I-S_2)^{-1}).
		\label{equation:Jkj-sum-1-3}
		\end{align}
		By Lemma \ref{lemma:inner-UV-expansion},
		\begin{align}
			&\sum_{k,j=1}^{\infty}\frac{\alpha_k}{1-\alpha_k}\frac{\beta_j}{1-\beta_j}\la u_{\nu}, \phi_k\ra\la u_{\nu}, \psi_j\ra\la \phi_k, (I-S_{\nu})\psi_j\ra
			\nonumber
			\\
			&= \la S_1(I-S_1)^{-1}u_{\nu}, (I-S_{\nu})S_2(I-S_2)^{-1}u_{\nu}\ra.
			\label{equation:Jkj-sum-2}
		\end{align}
		By Lemma \ref{lemma:quadratic-form-expansion},
		\begin{align}
			&\sum_{k=1}^{\infty}\frac{\alpha_k}{1-\alpha_k}\la u_{\nu}, \phi_k\ra^2 
			%\nonumber
			%\\
			%&
			= 
			\la u_{\nu}, S_1(I-S_1)^{-1}u_{\nu}\ra,
			\label{equation:Jkj-sum-3-1}
			\\
		&	\sum_{j=1}^{\infty}\frac{\beta_j}{1-\beta_j}\la u_{\nu}, \psi_j\ra^2
		%\nonumber
		%\\
		%&
		=	\la u_{\nu}, S_2(I-S_2)^{-1}u_{\nu}\ra.
		\label{equation:Jkj-sum-3-2}
		\end{align}
		By Lemma \ref{lemma:logdet-2-inverse},
		\begin{align}
	\sum_{k=1}^{\infty}\left[\frac{\alpha_k}{1-\alpha_k} + \log(1-\alpha_k)\right] = -	\log\dettwo[(I-S_1)^{-1}],
	\label{equation:Jkj-sum-4-1}
	\\
	 \sum_{j=1}^{\infty}\left[\frac{\beta_j}{1-\beta_j} + \log(1-\beta_j)\right] = -	\log\dettwo[(I-S_2)^{-1}].
	 \label{equation:Jkj-sum-4-2}
		\end{align}
		Combining
		% the expressions in 
		Eqs.\eqref{equation:Jkj-sum-1-1}, \eqref{equation:Jkj-sum-1-2}, \eqref{equation:Jkj-sum-1-3}, \eqref{equation:Jkj-sum-2}, \eqref{equation:Jkj-sum-3-1}, \eqref{equation:Jkj-sum-3-2}, \eqref{equation:Jkj-sum-4-1}, \eqref{equation:Jkj-sum-4-2}, we obtain
		\begin{align}
		%J &= \left\la \log\left\{{\frac{d\mu_1}{d\mu_{*}}(x)}\right\}, \log\left\{{\frac{d\mu_2}{d\mu_{*}}(x)}\right\}\right\ra_{\Lcal^2(\Hcal,\nu)} = 
		&\frac{1}{4}\sum_{k,j=1}^{\infty}J_{kj,1} 
		%\\
		%&
		= \frac{1}{2}\trace[(I-S_{\nu})S_1(I-S_2)^{-1}(I-S_{\nu})S_1(I-S_2)^{-1}]
		\nonumber
		\\
		&\quad + \la S_1(I-S_1)^{-1}u_{\nu}, (I-S_{\nu})S_2(I-S_2)^{-1}u_{\nu}\ra
		\nonumber
		\\
		& \quad + \frac{1}{4}\left[\la u_{\nu}, S_1(I-S_1)^{-1}u_{\nu}\ra - \trace(S_{\nu}S_1(I-S_1)^{-1})-	\log\dettwo[(I-S_1)^{-1} \right]
		\nonumber
		\\
		& \quad \quad \times \left[\la u_{\nu}, S_2(I-S_2)^{-1}u_{\nu}\ra-\trace(S_{\nu}S_2(I-S_2)^{-1})-	\log\dettwo[(I-S_2)^{-1}\right].
		\label{equation:sum-Jkj-1-total}
		\end{align}
	\begin{comment}
	\begin{align*}
		%&\left\la \log\left\{{\frac{d\mu}{d\mu_0}(x)}\right\}, \log\left\{{\frac{d\nu}{d\mu_0}(x)}\right\}\right\ra_{\Lcal^2(\Hcal,\mu_{*})}
		%\\
		%= \int_{\Hcal}f(x)g(x)\mu_{*}(dx) 
		%\\
		%&= \lim_{N, M \approach \infty}\int_{\Hcal}f_N(x)g_M(x)\mu_{*}(dx)
		%&
		&= \frac{1}{2}\sum_{k,j=1}^{\infty}\frac{\alpha_k}{1-\alpha_k}\frac{\beta_j}{1-\beta_j}\la \phi_k, (I-S_{*})\psi_j\ra^2
		\\
		&\quad + \frac{1}{4}\sum_{k=1}^{\infty}\left[- \frac{\alpha_k}{1-\alpha_k}\la \phi_k, S_{*}\phi_k\ra + \frac{\alpha_k}{1-\alpha_k} + \log(1-\alpha_k)\right]
		\\
		&\quad \quad \times \sum_{j=1}^{\infty}\left[- \frac{\beta_j}{1-\beta_j}\la \psi_j, S_{*}\psi_j\ra +\frac{\beta_j}{1-\beta_j} + \log(1-\beta_j)\right]
		\\
		& = \frac{1}{2}\trace[(I-S_{*})S(I-S)^{-1}(I-S_{*})T(I-T)^{-1}]
		%\frac{1}{2}\la (I-S_{*})^{1/2}S(I-S)^{-1}(I-S_{*})^{1/2}, (I-S_{*})^{1/2}T(I-T)^{-1}(I-S_{*})^{1/2}\ra_{\HS}
		\\
		& \quad + \frac{1}{4}(\log\dettwo[(I-S)^{-1}] + \trace[S^{*}S(I-S)^{-1}])
		\\
		& \quad \quad \times (\log\dettwo[(I-T)^{-1}] + \trace[S^{*}T(I-T)^{-1}]).
	\end{align*}
	where the last line follows from Lemmas \ref{lemma:HS-inner-expansion} and \ref{lemma:logdet-2-inverse}. Thus
	\end{comment}
Next, we calculate $J_{kj,2}$. For each pair $k,j \in \Nbb$,
	\begin{align*}
	& J_{kj,2}=-\frac{2}{1-\beta_j}\int_{\Hcal}\left[\frac{\alpha_k}{1-\alpha_k}W^2_{\phi_k}(x)+ \log(1-\alpha_k)\right]\la u_2, \psi_j\ra W_{\psi_j}(x)d\nu(x)
	\\
	&+ \frac{4}{(1-\alpha_k)(1-\beta_j)}\int_{\Hcal}\la u_1, \phi_k\ra W_{\phi_k}(x)\la u_2, \psi_j\ra W_{\psi_j}(x)d\nu(x)
	\\
	& - \frac{2}{1-\alpha_k} \left[\frac{\beta_j}{1-\beta_j}W^2_{\psi_j}(x) + \log(1-\beta_j)\right]\la u_1, \phi_k\ra W_{\phi_k}d\nu(x)
	\\
	&= - \frac{2\alpha_k\la u_2, \psi_j\ra[\la  (I-S_{\nu})\phi_k, \phi_k\ra\la u_{\nu}, \psi_j\ra +  2\la u_{\nu}, \phi_k\ra\la (I-S_{\nu})\phi_k, \psi_j\ra + \la u_{\nu}, \phi_k\ra^2\la u_{\nu}, \psi_j\ra]}{(1-\alpha_k)(1-\beta_j)}
	\\
	& \quad - \frac{2\la u_2, \psi_j\ra \la u_{\nu}, \psi_j\ra \log(1-\alpha_k)}{1-\beta_j}
	\\
	& \quad + \frac{4\la u_1, \phi_k\ra \la u_2, \psi_j\ra [\la \phi_k, (I-S_{\nu})\psi_j\ra + \la u_{\nu}, \phi_k\ra \la u_{\nu}, \psi_j\ra]}{(1-\alpha_k)(1-\beta_j)}
	\\
	& \quad - \frac{2\beta_j\la u_1, \phi_k\ra[\la  (I-S_{\nu})\psi_j, \psi_j\ra\la u_{\nu}, \phi_k\ra +  2\la u_{\nu}, \psi_j\ra\la (I-S_{\nu})\psi_j, \phi_k\ra + \la u_{\nu}, \psi_j\ra^2\la u_{\nu}, \phi_k\ra]}{(1-\alpha_k)(1-\beta_j)}
	\\
	& \quad - \frac{2\la u_1, \phi_k\ra \la u_{\nu}, \phi_k\ra\log(1-\beta_j)}{1-\alpha_k}.	
\end{align*}
Rearranging the terms in $J_{kj,2}$ gives
\begin{align*}
	&J_{kj,2} = -2 \frac{\la u_2, \psi_j\ra\la u_{\nu}, \psi_j\ra}{1-\beta_j}\left[\frac{\alpha_k\la  (I-S_{\nu})\phi_k, \phi_k\ra}{(1-\alpha_k)} +  \log(1-\alpha_k) + \frac{\alpha_k}{1-\alpha_k}\la u_{\nu}, \phi_k\ra^2\right] 
	\\
	& \quad - \frac{4\alpha_k\la u_2, \psi_j\ra\la u_{\nu}, \phi_k\ra\la (I-S_{\nu})\phi_k, \psi_j\ra }{(1-\alpha_k)(1-\beta_j)}
	\\
	& \quad + \frac{4\la u_1, \phi_k\ra \la u_2, \psi_j\ra [\la \phi_k, (I-S_{\nu})\psi_j\ra + \la u_{\nu}, \phi_k\ra \la u_{\nu}, \psi_j\ra]}{(1-\alpha_k)(1-\beta_j)}
	\\
	& \quad - 2\frac{\la u_1, \phi_k\ra\la u_{\nu}, \phi_k\ra}{1-\alpha_k}\left[\frac{\beta_j\la  (I-S_{\nu})\psi_j, \psi_j\ra}{(1-\beta_j)} + \log(1-\beta_j) + \frac{\beta_j}{1-\beta_j }\la u_{\nu}, \psi_j\ra^2\right]
	\\
	& \quad - \frac{4\beta_j\la u_1, \phi_k\ra\la u_{\nu}, \psi_j\ra\la (I-S_{\nu})\psi_j, \phi_k\ra ]}{(1-\alpha_k)(1-\beta_j)}.
\end{align*}
Summing each of the terms in $J_{kj,2}$ individually, we get for the first term,
\begin{align}
	&\sum_{k,j=1}^{\infty}\frac{\la u_2, \psi_j\ra\la u_{\nu}, \psi_j\ra}{1-\beta_j}\left[\frac{\alpha_k\la  (I-S_{\nu})\phi_k, \phi_k\ra}{(1-\alpha_k)} +  \log(1-\alpha_k) + \frac{\alpha_k}{1-\alpha_k}\la u_{\nu}, \phi_k\ra^2\right] 
	\nonumber
	\\
	& = \la (I-S_2)^{-1}u_2,u_{\nu}\ra
	\nonumber
	\\
	& \quad \times \sum_{k=1}^{\infty}\left[-\frac{\alpha_k}{1-\alpha}_k\la S_{\nu}\phi_k, \phi_k\ra +  \frac{\alpha_k}{1-\alpha_k}+\log(1-\alpha_k) + \frac{\alpha_k}{1-\alpha_k}\la u_{\nu}, \phi_k\ra^2\right] 
	\nonumber
	\\
	& = -\la (I-S_2)^{-1}u_2,u_{\nu}\ra  [\trace(S_{\nu}S_1(I-S_1)^{-1})+\log\dettwo(I-S_1)^{-1} ]
	\nonumber
	\\
	&\quad + \la (I-S_2)^{-1}u_2,u_{\nu}\ra \la u_{\nu}, S_1(I-S_1)^{-1}u_{\nu}\ra].
	\label{equation:Jkj2-sum-1}
\end{align}
For the second term,
\begin{align}
	&\sum_{k,j=1}^{\infty}\frac{\alpha_k\la u_2, \psi_j\ra\la u_{\nu}, \phi_k\ra\la (I-S_{\nu})\phi_k, \psi_j\ra }{(1-\alpha_k)(1-\beta_j)}
	\nonumber
	\\
	& = \la S_1(I-S_1)^{-1}u_{\nu}, (I-S_{\nu})(I-S_2)^{-1}u_2\ra.
	\label{equation:Jkj2-sum-2}
\end{align}
For the third term,
\begin{align}
	&\sum_{k,j=1}^{\infty}\frac{\la u_1, \phi_k\ra \la u_2, \psi_j\ra [\la \phi_k, (I-S_{\nu})\psi_j\ra + \la u_{\nu}, \phi_k\ra \la u_{\nu}, \psi_j\ra]}{(1-\alpha_k)(1-\beta_j)}
	\nonumber
	\\
	& = \la (I-S_1)^{-1}u_1, (I-S_{\nu})(I-S_2)^{-1}u_2\ra + \la (I-S_1)^{-1}u_1,u_{\nu}\ra \la (I-S_2)^{-1}u_2,u_{\nu}\ra. 
	\label{equation:Jkj2-sum-3}
\end{align}
For the fourth term,
\begin{align}
	&\sum_{k,j=1}^{\infty}\frac{\la u_1, \phi_k\ra\la u_{\nu}, \phi_k\ra}{1-\alpha_k}\left[\frac{\beta_j\la  (I-S_{\nu})\psi_j, \psi_j\ra}{(1-\beta_j)} + \log(1-\beta_j) + \frac{\beta_j}{1-\beta_j }\la u_{\nu}, \psi_j\ra^2\right]
	\nonumber
	\\
	& = - \la (I-S_1)^{-1}u_1, u_{\nu}\ra [\trace(S_{\nu}S_2(I-S_2)^{-1}) + \log\dettwo(I-S_2)^{-1}]
	\nonumber
	\\
	& \quad + \la (I-S_1)^{-1}u_1, u_{\nu}\ra\la u_{\nu}, S_2(I-S_2)^{-1}u_{\nu}\ra.
	\label{equation:Jkj2-sum-4}
\end{align}
For the fifth term,
\begin{align}
	\label{equation:Jkj2-sum-5}
	&\sum_{k,j=1}^{\infty}\frac{\beta_j\la u_1, \phi_k\ra\la u_{\nu}, \psi_j\ra\la (I-S_{\nu})\psi_j, \phi_k\ra ]}{(1-\alpha_k)(1-\beta_j)}
	\nonumber
	\\
	& = \la S_2(I-S_2)^{-1}u_{\nu}, (I-S_{\nu})(I-S_1)^{-1}u_1\ra.
\end{align}
Combining Eqs.\eqref{equation:Jkj2-sum-1}, \eqref{equation:Jkj2-sum-2}, \eqref{equation:Jkj2-sum-3}, \eqref{equation:Jkj2-sum-4}, \eqref{equation:Jkj2-sum-5}, we obtain
\begin{align}
	&\frac{1}{4}\sum_{k,j=1}^{\infty}J_{kj,2} 
	\nonumber
	= 
	\frac{1}{2}\la (I-S_2)^{-1}u_2,u_{\nu}\ra  [\trace(S_{\nu}S_1(I-S_1)^{-1})+\log\dettwo(I-S_1)^{-1} ]
	\\
	& - \frac{1}{2}\la (I-S_2)^{-1}u_2,u_{\nu}\ra \la u_{\nu}, S_1(I-S_1)^{-1}u_{\nu}\ra
	\nonumber
	\\
	& - \la S_1(I-S_1)^{-1}u_{\nu}, (I-S_{\nu})(I-S_2)^{-1}u_2\ra
	\nonumber
	\\
	& + \la (I-S_1)^{-1}u_1, (I-S_{\nu})(I-S_2)^{-1}u_2\ra + \la (I-S_1)^{-1}u_1,u_{\nu}\ra \la (I-S_2)^{-1}u_2,u_{\nu}\ra
	\nonumber
	\\
	&+\frac{1}{2}\la (I-S_1)^{-1}u_1, u_{\nu}\ra [\trace(S_{\nu}S_2(I-S_2)^{-1}) + \log\dettwo(I-S_2)^{-1}]
	\nonumber
	\\
	&-\frac{1}{2}\la (I-S_1)^{-1}u_1, u_{\nu}\ra\la u_{\nu}, S_2(I-S_2)^{-1}u_{\nu}\ra
	\nonumber
	\\
	&-\la S_2(I-S_2)^{-1}u_{\nu}, (I-S_{\nu})(I-S_1)^{-1}u_1\ra.
	\label{equation:sum-Jkj-2-total}
\end{align}
Combining Eqs.\eqref{equation:sum-Jkj-1-total} and \eqref{equation:sum-Jkj-2-total}, we obtain
\begin{align*}
&J = 
\left\la \log\left\{{\frac{d\mu_1}{d\mu_{*}}(x)}\right\}, \log\left\{{\frac{d\mu_2}{d\mu_{*}}(x)}\right\}\right\ra_{\Lcal^2(\Hcal,\nu)}
\\
&=
\frac{1}{4}\sum_{k,j=1}^{\infty}J_{kj} = \frac{1}{4}\sum_{k,j=1}^{\infty}J_{kj,1} + \frac{1}{4}\sum_{k,j=1}^{\infty}J_{kj,2}
\\
&
= \frac{1}{2}\trace[(I-S_{\nu})S_1(I-S_1)^{-1}(I-S_{\nu})S_2(I-S_2)^{-1}]
\nonumber
%\\
%&\quad + \la S_1(I-S_1)^{-1}u_{\nu}, (I-S_{\nu})S_2(I-S_2)^{-1}u_{\nu}\ra
%\nonumber
\\
& \quad +\left\la (I-S_{\nu})^{1/2}(I-S_1)^{-1}u_1 - (I-S_{\nu})^{1/2}S_1(I-S_1)^{-1}u_{\nu}, \right.
\\
&\quad \quad \quad \left.(I-S_{\nu})^{1/2}(I-S_2)^{-1}u_2 - (I-S_{\nu})^{1/2}S_2(I-S_2)^{-1}u_{\nu}\right\ra
\\
& \quad + \frac{1}{4}\left[\la u_{\nu}, S_1(I-S_1)^{-1}u_{\nu}\ra-2\la (I-S_1)^{-1}u_1, u_{\nu}\ra \right.
\\
&\quad \quad \quad \left. -  \trace(S_{\nu}S_1(I-S_1)^{-1})-	\log\dettwo[(I-S_1)^{-1} \right]
\nonumber
\\
& \quad \quad \times \left[\la u_{\nu}, S_2(I-S_2)^{-1}u_{\nu}\ra-2\la (I-S_2)^{-1}u_2, u_{\nu}\ra \right.
\\
&\quad \quad \quad \left. - \trace(S_{\nu}S_2(I-S_2)^{-1})-	\log\dettwo[(I-S_2)^{-1}\right].
%\label{equation:sum-Jkj-1-total}
\end{align*}
In particular, for $S_1 = S_2 = S \in \SymHS(\Hcal)_{< I}$, $m_1 = m_2 = m$, $\mu_1 = \mu_2 = \mu= \Ncal(m,C)$, with $C= C_{*}^{1/2}(I-S)C_{*}^{1/2}$, $u = C_{*}^{-1/2}(m- m_{*})$,
\begin{align*}
	&\left\| \log\left\{{\frac{d\mu}{d\mu_{*}}(x)}\right\}\right\|^2_{\Lcal^2(\Hcal,\nu)}
	\\
%	&=
%	\frac{1}{4}\sum_{k,j=1}^{\infty}J_{kj} = \frac{1}{4}\sum_{k,j=1}^{\infty}J_{kj,1} + \frac{1}{4}\sum_{k,j=1}^{\infty}J_{kj,2}
%	\\
	&
	= \frac{1}{2}||(I-S_{\nu})^{1/2}S(I-S)^{-1}(I-S_{\nu})^{1/2}||^2_{\HS}
	\nonumber
	%\\
	%&\quad + \la S_1(I-S_1)^{-1}u_{\nu}, (I-S_{\nu})S_2(I-S_2)^{-1}u_{\nu}\ra
	%\nonumber
	\\
	& \quad +||(I-S_{\nu})^{1/2}(I-S)^{-1}u - (I-S_{\nu})^{1/2}S(I-S)^{-1}u_{\nu}||^2
	\\
	& \quad + \frac{1}{4}\left[\la u_{\nu}, S(I-S)^{-1}u_{\nu}\ra-2\la (I-S)^{-1}u, u_{\nu}\ra \right.
	\\
	&\quad \quad \quad \left. -  \trace(S_{\nu}S(I-S)^{-1})-	\log\dettwo[(I-S)^{-1} \right]^2.
	%\label{equation:sum-Jkj-1-total}
\end{align*}
It follows then that		
	\begin{align*}
		&K = \left\| \log\left\{{\frac{d\mu_1}{d\mu_{*}}(x)}\right\}- \log\left\{{\frac{d\mu_2}{d\mu_{*}}(x)}\right\}\right\|^2_{\Lcal^2(\Hcal,\nu)}
		\\
		&= \frac{1}{2}||(I-S_{\nu})^{1/2}[S_1(I-S_1)^{-1}-S_2(I-S_2)^{-1}](I-S_{\nu})^{1/2}||^2_{\HS}
		\\
		& \quad + ||(I-S_{\nu})^{1/2}[(I-S_1)^{-1}u_1 - (I-S_2)^{-1}u_2] 
		\\
		&\quad \quad -(I-S_{\nu})^{1/2}[S_1(I-S_1)^{-1}-S_2(I-S_2)^{-1}]u_{\nu}||^2
		\\
		& \quad + \frac{1}{4}\left(\log\dettwo(I-S_1)^{-1}- \log\dettwo(I-S_2)^{-1} \right.
		\\
		&\quad \quad+ \trace[S_{\nu}[S_1(I-S_1)^{-1}-S_2(I-S_2)^{-1}]]
		\\
		&\quad \quad - \la u_{\nu}, [S_1(I-S_1)-S_2(I-S_2)^{-1}]u_{\nu}\ra 
		\\
		& \quad \quad \left.+ 2 \la (I-S_1)^{-1}u_1 - (I-S_2)^{-1}u_2, u_{\nu}\ra\right)^2.
	\end{align*}
Using the Cauchy-Schwarz inequality, we obtain
	\begin{align*}
		K&   \leq \frac{1}{2}||(I-S_{\nu})^{1/2}[S_1(I-S_1)^{-1}-S_2(I-S_2)^{-1}](I-S_{\nu})^{1/2}||^2_{\HS}
		\\
		&\quad + 2||(I-S_{\nu})^{1/2}[(I-S_1)^{-1}u_1 - (I-S_2)^{-1}u_2]||^2
		\\
		& \quad + 2||(I-S_{\nu})^{1/2}[S_1(I-S_1)^{-1}-S_2(I-S_2)^{-1}]u_{\nu}||^2
		\\
		&\quad + (\log\dettwo(I-S_1)^{-1}- \log\dettwo(I-S_2)^{-1})^2 
		\\
		&\quad +(\trace[S_{\nu}[S_1(I-S_1)^{-1}-S_2(I-S_2)^{-1}]])^2
		\\
		&\quad + \la u_{\nu}, [S_1(I-S_1)-S_2(I-S_2)^{-1}]u_{\nu}\ra^2
		\\
		& \quad + 4 \la (I-S_1)^{-1}u_1 - (I-S_2)^{-1}u_2, u_{\nu}\ra^2.
	\end{align*}
	\begin{comment}
	\begin{align*}
		&\left\| \log\left\{{\frac{d\mu}{d\mu_0}(x)}\right\}- \log\left\{{\frac{d\nu}{d\mu_0}(x)}\right\}\right\|^2_{\Lcal^2(\Hcal,\mu_{*})}
		\\
		&= \frac{1}{2}||(I-S_{*})^{1/2}[S(I-S)^{-1}-T(I-T)^{-1}](I-S_{*})^{1/2}||^2_{\HS}
		\\
		&  + \frac{1}{4}(\log\dettwo(I-S)^{-1}- \log\dettwo(I-T)^{-1} + \trace[S^{*}[S(I-S)^{-1}-T(I-T)^{-1}]])^2
		\\
		& \leq \frac{1}{2}||(I-S_{*})^{1/2}[S(I-S)^{-1}-T(I-T)^{-1}](I-S_{*})^{1/2}||^2_{\HS}
		\\
		&\quad + \frac{1}{2}(\log\dettwo(I-S)^{-1}- \log\dettwo(I-T)^{-1})^2 
		+\frac{1}{2}(\trace[S^{*}[S(I-S)^{-1}-T(I-T)^{-1}]])^2.
	\end{align*}
	\end{comment}
	Using $S_{1}(I-S_1)^{-1} - S_2(I-S_2)^{-1} = (I-S_1)^{-1}(S_1-S_2)(I-S_2)^{-1}$, we obtain
	\begin{align}
		&||(I-S_{\nu})^{1/2}[S_1(I-S_1)^{-1} - S_2(I-S_2)^{-1}](I-S_{\nu})^{1/2}||_{\HS}
		\nonumber
		\\
			&\leq ||I-S_{\nu}||\;||(I-S_1)^{-1}||\;||(I-S_2)^{-1}||\;||S_1-S_2||_{\HS},
		\label{equation:K-upperbound-1}
		\\ 
		&||(I-S_{\nu})^{1/2}[S_1(I-S_1)^{-1}-S_2(I-S_2)^{-1}]u_{\nu}||
		\nonumber
		\\
		& \leq ||(I-S_{\nu})^{1/2}||(I-S_1)^{-1}||\;||(I-S_2)^{-1}||\;||S_1-S_2||_{\HS}\;||u_{\nu}||,
		\label{equation:K-upperbound-2}
		\\
		&|\trace[S_{\nu}[S_1(I-S_1)^{-1}-S_2(I-S_2)^{-1}]| = |\la S_{\nu}, S_1(I-S_1)^{-1}-S_2(I-S_2)^{-1}\ra_{\HS}|
		\nonumber
		\\
		& \leq ||S_{\nu}||_{\HS}||(I-S_1)^{-1}||\;||(I-S_2)^{-1}||\;||S_1-S_2||_{\HS},
		\label{equation:K-upperbound-3}
		\\
		&|\la u_{\nu}, [S_1(I-S_1)-S_2(I-S_2)^{-1}]u_{\nu}\ra|
		\nonumber
		\\
		& \leq ||(I-S_1)^{-1}||\;||(I-S_2)^{-1}||\;||S_1-S_2||_{\HS}\;||u_{\nu}||^2.
		\label{equation:K-upperbound-4}
	\end{align}
	By Theorem 12 in \cite{Minh:2022KullbackGaussian},
	\begin{align}
		&|\log\dettwo[(I-S_1)^{-1}] - \log\dettwo[(I-S_2)^{-1}]| 	
		\label{equation:K-upperbound-5}
		%\nonumber
		\\
		&\leq ||(I-S_1)^{-1}||\;||(I-S_2)^{-1}||[||S_1||_{\HS} + ||S_2||_{\HS} + ||S_1S_2||_{\HS}]||S_1-S_2||_{\HS}.
	\nonumber
	\end{align}
	Write $(I-S_1)^{-1} - (I-S_2)^{-1} = (I-S_1)^{-1}(S_1 - S_2)(I-S_2)^{-1}$, we obtain
\begin{align*}
&||(I-S_1)^{-1}u_1 - (I-S_2)^{-1}u_2||
\\
&=||[(I-S_1)^{-1}-(I-S_2)^{-1}]u_1 + (I-S_2)^{-1}(u_1 - u_2)||
\\
& \leq ||[(I-S_1)^{-1}-(I-S_2)^{-1}]||\;||u_1|| + ||(I-S_2)^{-1}||\;||u_1-u_2||
\\
& \leq ||(I-S_1)^{-1}||\;||(I-S_2)^{-1}||\;||S_1-S_2||\;||u_1|| + ||(I-S_2)^{-1}||\;||u_1-u_2||.
\end{align*}
It follows that
\begin{align}
&||(I-S_{\nu})^{1/2}[(I-S_1)^{-1}u_1 - (I-S_2)^{-1}u_2]||^2
\nonumber
\\
&\leq  2||I-S_{\nu}||\;||(I-S_1)^{-1}||^2||(I-S_2)^{-1}||^2||S_1-S_2||^2||u_1||^2
\nonumber
\\
& \quad + 2||I-S_{\nu}||\;||(I-S_2)^{-1}||^2||u_1-u_2||^2,
\label{equation:K-upperbound-6}
\\
&\la (I-S_1)^{-1}u_1 - (I-S_2)^{-1}u_2, u_{\nu}\ra^2 \leq
|| (I-S_1)^{-1}u_1 - (I-S_2)^{-1}u_2||^2||u||^2_{\nu} 
\nonumber
\\
&\leq 
2||(I-S_1)^{-1}||^2||(I-S_2)^{-1}||^2||S_1-S_2||^2||u_1||^2||u_{\nu}||^2
\nonumber
\\
& +2||(I-S_2)^{-1}||^2||u_1-u_2||^2||u_{\nu}||^2.
\label{equation:K-upperbound-7}
\end{align}
Combining Eqs.\eqref{equation:K-upperbound-1},\eqref{equation:K-upperbound-2},\eqref{equation:K-upperbound-3},\eqref{equation:K-upperbound-4},\eqref{equation:K-upperbound-5},\eqref{equation:K-upperbound-6},\eqref{equation:K-upperbound-7}, we obtain
	\begin{align*}
		K &= \left\| \log\left\{{\frac{d\mu_1}{d\mu_{*}}(x)}\right\}- \log\left\{{\frac{d\mu_2}{d\mu_{*}}(x)}\right\}\right\|^2_{\Lcal^2(\Hcal,\nu)}
		\\
		&\leq \left(\frac{1}{2}||I-S_{\nu}||^2 + 2||I-S_{\nu}||\;||u_{\nu}||^2 + [||S_1||_{\HS} + ||S_2||_{\HS} + ||S_1S_2||_{\HS}]^2\right. 
		\\
		&\quad \quad \left.+ ||S_{\nu}||_{\HS}^2 + ||u_{\nu}||^4 + 8||u_1||^2||u_{\nu}||^2 + 4||u_1||^2\right)
		\\
		& \quad \quad \times ||(I-S_1)^{-1}||^2||(I-S_2)^{-1}||^2||S_1-S_2||^2_{\HS}
		\\
		& \quad +[4||I-S_{\nu}||+8||u_{\nu}||^2] ||(I-S_2)^{-1}||^2||u_1 - u_2||^2.
	\end{align*}
	This completes the proof.
	\qed
\end{proof}

\subsection{Proof for the geometric interpolation between two equivalent Gaussian measures}
\label{section:proof-geometic-interpolation-Gaussian}

In the following, we consider two cases. First, we prove Theorem \ref{theorem:geometric-interpolation-equivalent-Gaussian-Hilbert-space}  for the case $S_i \in \SymTr(\Hcal)_{< I}$, $i=0,1$.
Then we proceed to the general case $S_i \in \SymHS(\Hcal)_{< I}$, $i=0,1$, via a limiting process.
We need the following technical results. The first is a result on Gaussian integrals on $\Hcal$ from \cite{DaPrato:PDEHilbert}.

\begin{theorem}
	[\cite{DaPrato:PDEHilbert}, Proposition 1.2.8]
	\label{theorem:exponential-quadratic-Gaussian}
	Let $\Ncal(0,Q) \in \Gauss(\Hcal)$.
	Assume that $M$ is a self-adjoint operator on $\Hcal$ such that $\la Q^{1/2}MQ^{1/2} x, x \ra < ||x||^2$ $\forall 
	x \in \Hcal, x \neq 0$. Let $b \in \Hcal$. Then 
	\begin{align}
		&\int_{\Hcal}\exp\left(\frac{1}{2}\la M y, y\ra + \la b,y\ra\right)\Ncal(0,Q)(dy)
		%\nonumber
		\\
		& = [\det(I - Q^{1/2}MQ^{1/2})]^{-1/2}\exp\left(\frac{1}{2}||(I-Q^{1/2}MQ^{1/2})^{-1/2}Q^{1/2}b||^2\right).
		\nonumber
	\end{align}
\end{theorem}

The following result is standard, for the proof see e.g. \cite{Minh:2021RiemannianEstimation}. 
\begin{lemma}
\label{lemma:trace-class-orthogonal-projection}
Let $A \in \Tr(\Hcal)$. Let $\{e_k\}_{k \in \Nbb}$ be any
orthonormal basis in $\Hcal$. Let $N \in \Nbb$. Consider the orthogonal projection
$P_N = \sum_{k=1}^Ne_k \otimes e_k$. Then
\begin{align}
\lim_{N \approach \infty}||P_NAP_N-A||_{\trace} = \lim_{N \approach \infty}||P_NA-A||_{\tr} = \lim_{N\approach \infty}||AP_N - A||_{\tr} = 0.
\end{align}
\end{lemma}

%The following result follows from the triangle inequality of the norm.
%\begin{lemma}
%\label{lemma:operator-inverse-convergence}
%Let $A,B,\{A_k, B_k\}_{k \in \Nbb} \in \Sym(\Hcal)$ be such that $I+A>0$, $I+A_k > 0$ $\forall k \in \Nbb$. Assume that
%$\lim_{k \approach \infty}||A_k-A|| = 0$, $\lim_{k \approach \infty}||B_k - B|| = 0$. Then
%\begin{align}
%\lim_{k \approach \infty}||(I+A_k)^{-1}B_k - (I+A)^{-1}B|| = 0.
%\end{align}
%\end{lemma}

\begin{proposition}
	[\cite{Minh2024:FisherRaoGaussian}, Proposition 8]
	\label{proposition:RN-S-trace-class-quadratic-form-convergence}
	Let $\mu_0 = \Ncal(m_0, C_0)$ be fixed, with the associated white noise mapping $W:\Hcal \mapto \Lcal^2(\Hcal,\mu_0)$.
	Assume that $T \in \Sym(\Hcal) \cap \Tr(\Hcal)$, with eigenvalues $\{\alpha_k\}_{k \in \Nbb}$ and
	corresponding orthonormal eigenvectors $\{\phi_k\}_{k \in \Nbb}$. 	Let $P_N = \sum_{j=1}^Ne_j \otimes e_j$, with $\{e_j\}_{j\in \Nbb}$ being the orthonormal eigenvectors of $C_0$. Then in the $\Lcal^2(\Hcal,\mu_0)$ sense,
	%\begin{align}
	%	\lim_{N \approach \infty}\sum_{k=1}^{\infty}{\alpha_k}W^2_{P_N\phi_k} = \sum_{k=1}^{\infty}{\alpha_k}W^2_{\phi_k}.
	%\end{align}
%, then, consequently, with the limit being taken in the $\Lcal^2(\Hcal,\mu_{0})$ sense, 
	%consequently, with the limit being taken in the $\Lcal^2(\Hcal,\mu_0)$ sense,
	\begin{align}
		&\sum_{k=1}^{\infty}{\alpha_k}W^2_{\phi_k}(x)
		=\lim_{N \approach \infty}\sum_{k=1}^{\infty}{\alpha_k}W^2_{P_N\phi_k}
		\\
		%= \sum_{k=1}^{\infty}\frac{\alpha_k}{1-\alpha_k}\lim_{N \approach \infty}W^2_{P_N\phi_k}(x)
		%	= \lim_{N \approach \infty} \sum_{k=1}^{\infty}\frac{\alpha_k}{1-\alpha_k}W^2_{P_N\phi_k}(x)
		%\\
		& = \lim_{N \approach \infty}\la C_0^{-1/2}P_N(x-m_0), TC_0^{-1/2}P_N(x-m_0)\ra
		\\
		&\doteq\la C_0^{-1/2}(x-m_0), TC_0^{-1/2}(x-m_0)\ra.
	\end{align}
\end{proposition}

We now prove the following result, whose statement is well-defined by the limits in Proposition \ref{proposition:RN-S-trace-class-quadratic-form-convergence}. It is a generalization of Proposition 6 in \cite{Minh:2020regularizedDiv}.

\begin{proposition}
	\label{proposition:exponetial-quadratic-whitenoise-Gaussian}
	Let $\mu = \Ncal(m,C) \in \Gauss(\Hcal)$, $C \in \Sym^{++}(\Hcal) \cap \Tr(\Hcal)$.
	Assume that $A \in \SymTr(\Hcal)_{< I}$, $b \in \Hcal$. Consider
	the orthogonal projection $P_N = \sum_{k=1}^N e_k \otimes e_k$, $N \in \Nbb$, where $\{e_k\}_{k \in \Nbb}$
	are the orthonormal eigenvectors of $C$.
	Define
	\begin{align}
	&\la C^{-1/2}(x-m), AC^{-1/2}(x-m)\ra + \la C^{-1/2}(x-m),b\ra 
	%\nonumber
	\\
	&\doteq \lim_{N \approach \infty}\la C^{-1/2}P_N(x-m), AC^{-1/2}P_N(x-m)\ra + \la C^{-1/2}P_N(x-m),b\ra,
	\nonumber
	\end{align}
	where the limit is in the $\Lcal^2(\Hcal,\mu)$ sense.
	Then
	\begin{align}
		&\int_{\Hcal}\exp\left(\frac{1}{2}\la C^{-1/2}(x-m), AC^{-1/2}(x-m)\ra + \la C^{-1/2}(x-m),b\ra\right)\Ncal(m,C)(dx)
		\nonumber
		\\
		& = [\det(I- A)]^{-1/2}\exp\left(\frac{1}{2}||(I-A)^{-1/2}b||^2\right).	\label{equation:exponetial-quadratic-whitenoise-Gaussian}
	\end{align}
\end{proposition}
%In the following, 
%we make use of
{\bf Special cases}. Let $A = c_1 g \otimes g$, $b = c_2g$, where $c_1,c_2 \in \R$, $g \in \Hcal$, with $c_1||g||^2 < 1$. Then Eq.\eqref{equation:exponetial-quadratic-whitenoise-Gaussian} gives
	\begin{align}
	&\int_{\Hcal}\exp\left[\frac{1}{2}c_1W_{g}^2(x) + c_2W_{g}(x) \right]\Ncal(m,C)(dx)
	\nonumber
	\\
	&=\frac{1}{(1-c_1||g||^2)^{1/2}}\exp\left(\frac{c_2^2||g||^2}{2(1-c_1||g||^2)}\right).
	%\nonumber
\end{align} 
This is Proposition 6 in \cite{Minh:2020regularizedDiv}. In particular, with $c_1 = 0$,
\begin{align}
	\int_{\Hcal}\exp[c_2W_{g}(x)]\Ncal(m,C)(dx) = \exp\left(\frac{c_2^2}{2}||g||^2\right).
\end{align}
With $c_2 = 1$, the above formula gives Proposition 1.2.7 in \cite{DaPrato:PDEHilbert}.

The proof of Proposition \ref{proposition:exponetial-quadratic-whitenoise-Gaussian} utilizes the Vitali Convergence Theorem on {\it uniformly integrable} sequences of functions,
%which we now recall 
(see e.g. \cite{Folland:Real,Rudin:RealComplex}).
Let 
$(\Xcal, \Fcal, \mu)$ be a positive measurable space. A sequence of functions $\{f_n\}_{n\in \Nbb} \in \Lcal^1(\Xcal,\mu)$ 
is said to be {\it uniformly integrable} if $\forall \epsilon > 0$ $\exists \delta > 0$ such that
\begin{align}
	\sup_{n \in \Nbb}\int_{E}|f_n|d\mu < \epsilon \;\;\; \text{whenever} \;\;\; \mu(E) < \delta, E \in \Fcal.
\end{align}

\begin{theorem}
	[\textbf{Vitali's Convergence Theorem}]
	Assume that $(\Xcal, \Fcal, \mu)$ is a positive measurable space with $\mu(\Xcal) < \infty$.
	Let $\{f_n\}_{n\in \Nbb}$ be a sequence of functions that are uniformly integrable on $\Xcal$, with
	$f_n \approach f$ a.e. and $|f| < \infty$ a.e.. Then $f \in \Lcal^1(\Xcal, \mu)$ and $||f_n - f||_{\Lcal^1(\Xcal, \mu)} \approach 0$.
\end{theorem}
\begin{proof}
	[\textbf{of Proposition \ref{proposition:exponetial-quadratic-whitenoise-Gaussian}}]
	It suffices to prove Eq.\eqref{equation:exponetial-quadratic-whitenoise-Gaussian} for $m=0$.
Define the sequence $\{f_N = \exp\left(\frac{1}{2}\la C^{-1/2}P_Nx, AC^{-1/2}P_Nx\ra + \la C^{-1/2}P_Nx,b\ra\right)\}_{N \in \Nbb}$.
We show that the sequence $\{f_N\}_{N \in \Nbb}$ is uniformly integrable on $(\Hcal,\mu)$, then apply Vitali's convergence theorem.
% is applicable.

By assumption, $A \in \SymTr(\Hcal)_{< I}$, i.e. $A \in \Tr(\Hcal)$ and $I-A > 0$. We first show that $\exists p > 1$ such that
$I -pA > 0$. Let $\{\lambda_k\}_{k\in \Nbb}$ be the eigenvalues of $A$, then $\lambda_k < 1$ $\forall k\in \Nbb$  and $\lim_{k \approach \infty}\lambda_k = 0$. Let $\lambda_{\max}$ be the largest eigenvalue. If $\lambda_{\max} \leq 0$, then
$1-p\lambda_k \geq 1 $ $\forall p > 0$. If $0< \lambda_{\max} <1$, then let $1 < p < \frac{1}{\lambda_{\max}}$, so that
$1-p\lambda_k \geq 1 - p\lambda_{\max} > 0$ $\forall k \in \Nbb$.
Thus it follows that $I-pA > 0$, hence $pA \in \SymTr(\Hcal)_{< I}$.
%\begin{align*}
%	&\la C^{-1/2}P_Nx, AC^{-1/2}P_Nx\ra + \la C^{-1/2}P_Nx,b\ra 
	%\\
	%&= 
%	=\la x, C^{-1/2}P_NAC^{-1/2}P_Nx\ra 
%	+ \la x, C^{-1/2}P_Nb\ra
%\end{align*}

Let $M = C^{-1/2}P_NAC^{-1/2}P_N$. Since $C^{-1/2}P_N \in \Sym(\Hcal)$ and $C^{-1/2}P_NC^{1/2} = P_N$, we have $M \in \Sym(\Hcal)$ and $C^{1/2}MC^{1/2} = P_NAP_N$. Then $P_NAP_N \in \SymTr(\Hcal)_{< I}$ and $pP_NAP_N \in \SymTr(\Hcal)_{< I}$.
Thus by Theorem \ref{theorem:exponential-quadratic-Gaussian},
\begin{align*}
&J_{N,p} \doteq \int_{\Hcal}\exp\left(\frac{p}{2}\la C^{-1/2}P_Nx, pAC^{-1/2}P_Nx\ra + p\la C^{-1/2}P_Nx,b\ra\right) \Ncal(0,C)(dx)
\\
& = \int_{\Hcal}\exp\left(\frac{1}{2}\la x, p C^{-1/2}P_NAC^{-1/2}P_Nx\ra + \la x,pC^{-1/2}P_Nb\ra\right) \Ncal(0,C)(dx)
\\
& = \int_{\Hcal}\exp\left(\frac{1}{2}\la x, pMx\ra + \la x,pC^{-1/2}P_Nb\ra\right)\Ncal(0,C)(dx)
\\
& = [\det(I- pP_NAP_N)]^{-1/2}\exp\left(\frac{p^2}{2}||(I-pP_NAP_N)^{-1/2}P_Nb||^2\right).
\end{align*}
Taking limit as $N \approach \infty$, applying Lemma \ref{lemma:trace-class-orthogonal-projection} and the continuity of the Fredholm determinant in the trace class norm, we obtain
\begin{align}
&\lim_{N \approach \infty}J_{N,p} = \lim_{N \approach \infty}[\det(I- pP_NAP_N)]^{-1/2}\exp\left(\frac{p^2}{2}||(I-pP_NAP_N)^{-1/2}P_Nb||^2\right)
\nonumber
\\
& = [\det(I- pA)]^{-1/2}\exp\left(\frac{p^2}{2}||(I-pA)^{-1/2}b||^2\right) < \infty.
\label{equation:limit-JNp}
\end{align}
Let $q$ be the H\"older conjugate of $p$, i.e. $\frac{1}{p} + \frac{1}{q} = 1$.
By H\"older's inequality, for any set $A \in \Bsc(\Hcal)$,
\begin{align*}
	J_N &\doteq \int_{A}\exp\left(\frac{1}{2}\la C^{-1/2}P_Nx, AC^{-1/2}P_Nx\ra + \la C^{-1/2}P_Nx,b\ra\right) \Ncal(0,C)(dx)
	\\
	&= \int_{\Hcal}\onebf_A\exp\left(\frac{1}{2}\la C^{-1/2}P_Nx, AC^{-1/2}P_Nx\ra + \la C^{-1/2}P_Nx,b\ra\right) \Ncal(0,C)(dx)
	\\
	& \leq \mu(A)^{1/q}
	\\
	& \times \left(\int_{\Hcal}\exp\left(\frac{p}{2}\la C^{-1/2}P_Nx, AC^{-1/2}P_Nx\ra + p\la C^{-1/2}P_Nx,b\ra\right) \Ncal(0,C)(dx)\right)^{1/p}.
\end{align*}
Combining this inequality with the limit in Eq.\eqref{equation:limit-JNp}, it follows that $\forall \epsilon > 0$, $\exists \delta > 0$ such that $\forall A \in \Bsc(\Hcal)$ with $\mu(A) < \delta$,
\begin{align*}
\sup_{N \in \Nbb}\int_{A}\exp\left(\frac{1}{2}\la C^{-1/2}P_Nx, AC^{-1/2}P_Nx\ra + \la C^{-1/2}P_Nx,b\ra\right) \Ncal(0,C)(dx) < \epsilon.
\end{align*}
Thus $\{f_N = \exp\left(\frac{1}{2}\la C^{-1/2}P_Nx, AC^{-1/2}P_Nx\ra + \la C^{-1/2}P_Nx,b\ra\right)\}_{N \in \Nbb}$ is uniformly integrable on $(\Hcal,\mu)$. Since
$\lim_{N \approach \infty}\la C^{-1/2}P_N(x-m), AC^{-1/2}P_N(x-m)\ra + \la C^{-1/2}P_N(x-m),b\ra = \la C^{-1/2}(x-m), AC^{-1/2}(x-m)\ra + \la C^{-1/2}(x-m),b\ra$ in $\Lcal^2(\Hcal,\mu)$, there is a subsequence 
of $\{f_N\}_{N \in \Nbb}$ that converges a.e.
to  $\exp\left(\frac{1}{2}\la C^{-1/2}x, AC^{-1/2}x\ra + \la C^{-1/2}x,b\ra\right)$.
Applying Vitali's convergence theorem to this subsequence then gives
\begin{align*}
&J \doteq \int_{\Hcal}\exp\left(\frac{1}{2}\la C^{-1/2}x, AC^{-1/2}x\ra + \la C^{-1/2}x,b\ra\right)N(0,C)(dx)
\\
&= \lim_{N \approach \infty}\int_{\Hcal}\exp\left(\frac{1}{2}\la C^{-1/2}P_Nx, AC^{-1/2}P_Nx\ra + \la C^{-1/2}P_Nx,b\ra\right)\Ncal(0,C)(dx)
\\
& = \lim_{N \approach \infty}[\det(I-P_NAP_N)]^{-1/2}\exp\left(\frac{1}{2}||(I-P_NAP_N)^{-1/2}P_Nb||^2\right)
\\
& = [\det(I-A)]^{-1/2}\exp\left(\frac{1}{2}||(I-A)^{-1/2}b||^2\right).
\end{align*}
\qed
\end{proof}

The following is a special case of Corollary 3 in \cite{Minh:Positivity2020}.
\begin{lemma}
	\label{lemma:A-alpha-I+A-HS-Tr}
	Let $0 \leq \alpha \leq 1$.
	Let $A \in \Sym(\Hcal) \cap \HS(\Hcal)$, $I+A > 0$. Then
	\begin{align}
		&(1-\alpha)(I+A)^{-\alpha} + \alpha (I+A)^{1-\alpha} \in \PC_1(\Hcal), 
		\\
		\text{with } &(1-\alpha)(I+A)^{-\alpha} + \alpha (I+A)^{1-\alpha} - I \in \Sym(\Hcal) \cap \Tr(\Hcal).
	\end{align}
	\end{lemma}
As a consequence of Lemma \ref{lemma:A-alpha-I+A-HS-Tr}, we have $0 < \det[(1-\alpha)(I+A)^{-\alpha} + \alpha (I+A)^{1-\alpha}] < \infty$ for $0 \leq \alpha \leq 1$
and $A \in \Sym(\Hcal) \cap \HS(\Hcal)$, $I+A > 0$.
\begin{lemma}
[\cite{Minh:2019AlphaBeta}, Lemma 10]
\label{lemma:I+A-power-Tr}
Let $A \in \Sym(\Hcal) \cap \Tr(\Hcal)$, with $I+A > 0$. Let $\beta \in \R$ be fixed but arbitrary.
Then $(I+A)^{\beta} \in \PC_1(\Hcal)$, with $(I+A)^{\beta} - I \in \Sym(\Hcal) \cap \Tr(\Hcal)$.
\end{lemma}
As a consequence of Lemma \ref{lemma:I+A-power-Tr}, for $I+A \in \PC_1(\Hcal)$, we have $\det[(I+A)^{\beta}]$ is well-defined $\forall \beta \in \R$. Furthermore, from definition of $\det(I+A)$,
\begin{align}
	\det[(I+A)^{\beta}] = [\det(I+A)]^{\beta} > 0, \;\;\forall \beta \in \R.
\end{align}
\begin{proposition}
	[\textbf{Normalizing factor}]
	\label{proposition:normalizing factor}
	Assume the hypothesis of Theorem \ref{theorem:geometric-interpolation-equivalent-Gaussian-Hilbert-space}.
	Assume further that $S_i \in \SymTr(\Hcal)_{< I}$, $i=0,1$. Then the normalizing factor $Z_{\alpha}^G(p_0:p_1)$ is given by
	\begin{align}
		&Z_{\alpha}^G(p_0:p_1) \doteq \int_{\Hcal}p_0^{1-\alpha}(x)p_1^{\alpha}(x) \; d\mu_{*}(x)
		\nonumber
		\\
		& =  [\det(I-S_0)]^{-\frac{1-\alpha}{2}}[\det(I-S_1)]^{-\frac{\alpha}{2}}
		\nonumber
		\\
		& \quad \times \exp\left[-\frac{1-\alpha}{2}||(I-S_0)^{-1/2}C_{*}^{-1/2}(m_0 - m_{*})||^2\right]
		\nonumber
		\\
		& \quad \times 
		\exp\left[-\frac{\alpha}{2}||(I-S_1)^{-1/2}C_{*}^{-1/2}(m_1 - m_{*})||^2\right]
		\nonumber
		\\
		&\quad \times \det(I-S_{\alpha})^{1/2}\exp\left(\frac{1}{2}||(I-S_{\alpha})^{-1/2}C_{*}^{-1/2}(m_{\alpha} - m_{*})||^2\right).
		\label{equation:normalizing-factor-Tr}
	\end{align}
Equivalently, with $I + A = (I-S_1)^{-1/2}(I-S_0)(I-S_1)^{-1/2} > 0$, with $A =(I-S_1)^{-1/2}(S_1-S_0)(I-S_1)^{-1/2}$, then $Z_{\alpha}^G(p_0:p_1)$ is given by
\begin{align}
	&Z_{\alpha}^G(p_0:p_1)
	\nonumber
	\\
	%& = 
	%\int_{\Hcal}p_0^{1-\alpha}(x)p_1^{\alpha}(x) d\mu_{*}(x)
	&= \det[(1-\alpha)(I+A)^{-\alpha} + \alpha (I+A)^{1-\alpha}]^{-1/2}
	\nonumber
	\\
	& \quad \times \exp\left[-\frac{1-\alpha}{2}||(I-S_0)^{-1/2}C_{*}^{-1/2}(m_0 - m_{*})||^2\right]
	\nonumber
	\\
	& \quad \times 
	\exp\left[-\frac{\alpha}{2}||(I-S_1)^{-1/2}C_{*}^{-1/2}(m_1 - m_{*})||^2\right]
	\nonumber
	\\
	& \quad \times \exp\left(\frac{1}{2}||(I-S_{\alpha})^{-1/2}C_{*}^{-1/2}(m_{\alpha} - m_{*})||^2\right).
	\label{equation:normalizing-factor-HS-1}
\end{align}
Furthermore, this expression is valid for $S_i \in  \SymHS(\Hcal)_{< I}$, $i=0,1$.	
\end{proposition}
\begin{proof}
	By definition,
	\begin{align}
		&Z_{\alpha}^G(p_0:p_1) = \int_{\Hcal}p_0^{1-\alpha}(x)p_1^{\alpha}(x) \; d\mu_{*}(x)
		\nonumber
		\\
		& =  [\det(I-S_0)]^{-\frac{1-\alpha}{2}}[\det(I-S_1)]^{-\frac{\alpha}{2}}
		\nonumber
		\\
		& \quad \times \exp\left[-\frac{1-\alpha}{2}||(I-S_0)^{-1/2}C_{*}^{-1/2}(m_0 - m_{*})||^2\right]
		\nonumber
		\\
		& \quad \times 
		\exp\left[-\frac{\alpha}{2}||(I-S_1)^{-1/2}C_{*}^{-1/2}(m_1 - m_{*})||^2\right]
		\nonumber
		\\
		&\quad \times J,
		\label{equation:normalizing-1}
		%\int_{\Hcal}\exp\left\{-\frac{1}{2}\la C_{*}^{-1/2}(x-m_{*}), [(1-\alpha)S_0(I-S_0)^{-1} + \alpha S_1(I-S_1)^{-1}]\right.
		%\\
		%&\quad \quad \quad \quad \quad C_{*}^{-1/2}(x-m_{*})\ra 
		%\right\}
		%\\
		%& \quad \quad \quad \quad \quad +(\la C_{*}^{-1/2}(x-m_{*}), (1-\alpha)(I-S_0)^{-1}C_{*}^{-1/2}(m_0 - m_{*})\ra)
		%\\
		%& \quad \quad \quad \quad \quad \left.+ (\la C_{*}^{-1/2}(x-m_{*}), \alpha(I-S_1)^{-1}C_{*}^{-1/2}(m_1 - m_{*})\ra) \right\}.
	\end{align}
	where $J$ is the following integral over $\Hcal$ with respect to $\mu_{*} = \Ncal(m_{*}, C_{*})$
	\begin{align*}
		J &= \int_{\Hcal}\exp\left\{-\frac{1}{2}\la C_{*}^{-1/2}(x-m_{*}), [(1-\alpha)S_0(I-S_0)^{-1} + \alpha S_1(I-S_1)^{-1}]\right.
		\\
		&\quad \quad \quad \quad \quad C_{*}^{-1/2}(x-m_{*})\ra 
		%\right\}
		\\
		& \quad \quad \quad \quad \quad +(\la C_{*}^{-1/2}(x-m_{*}), (1-\alpha)(I-S_0)^{-1}C_{*}^{-1/2}(m_0 - m_{*})\ra)
		\\
		& \quad \quad \quad \quad \quad \left.+ (\la C_{*}^{-1/2}(x-m_{*}), \alpha(I-S_1)^{-1}C_{*}^{-1/2}(m_1 - m_{*})\ra) \right\}d\mu_{*}(x)
		\\
		&= \int_{\Hcal}\exp\left\{-\frac{1}{2}\la C_{*}^{-1/2}x, [(1-\alpha)S_0(I-S_0)^{-1} + \alpha S_1(I-S_1)^{-1}]
		%\right.
		%\\
		%&\quad \quad \quad \quad \quad 
		C_{*}^{-1/2}x\ra 
		%\right\}
		\right.
		\\
		& \quad \quad \quad \quad \quad +(\la C_{*}^{-1/2}x, (1-\alpha)(I-S_0)^{-1}C_{*}^{-1/2}(m_0 - m_{*})\ra)
		\\
		& \quad \quad \quad \quad \quad \left.+ (\la C_{*}^{-1/2}x, \alpha(I-S_1)^{-1}C_{*}^{-1/2}(m_1 - m_{*})\ra) \right\}d\Ncal(0,C_{*})(x).
	\end{align*}
Define the following
\begin{align*}
A &= -[(1-\alpha)S_0(I-S_0)^{-1} + \alpha S_1(I-S_1)^{-1}],
\\
b&= (1-\alpha)(I-S_0)^{-1}C_{*}^{-1/2}(m_0 - m_{*}) + \alpha(I-S_1)^{-1}C_{*}^{-1/2}(m_1 - m_{*}).
\end{align*}
Since $S_0, S_1 \in \SymTr(\Hcal)_{<I}$, we have $A \in \Sym(\Hcal) \cap \Tr(\Hcal)$. Furthermore,
\begin{align*}
I - A = (1-\alpha)(I-S_0)^{-1} + \alpha (I-S_1)^{-1} > 0.
\end{align*}
Thus $A \in \SymTr(\Hcal)_{< I}$.
By Proposition \ref{proposition:exponetial-quadratic-whitenoise-Gaussian},
\begin{align}
J = [\det(I-A)]^{-1/2}\exp\left(\frac{1}{2}||(I-A)^{-1/2}b||^2\right).
\label{equation:normalize-J-A}
\end{align}
Define $S_{\alpha} \in \Tr(\Hcal)$ by
\begin{align*}
	&S_{\alpha}(I-S_{\alpha})^{-1} = (1-\alpha)S_0(I-S_0)^{-1} + \alpha S_1(I-S_1)^{-1}\\
	&\equivalent (I-S_{\alpha})^{-1} - I = (1-\alpha)[(I-S_0)^{-1}-I] + \alpha [(I-S_1)^{-1}-I]\\
	&\equivalent (I-S_{\alpha})^{-1} = (1-\alpha)(I-S_0)^{-1} + \alpha (I-S_1)^{-1}
	\\
	& \equivalent (I-S_{\alpha}) = [(1-\alpha)(I-S_0)^{-1} + \alpha (I-S_1)^{-1}]^{-1}
	\\
	& \equivalent S_{\alpha} = I - [(1-\alpha)(I-S_0)^{-1} + \alpha (I-S_1)^{-1}]^{-1}.
\end{align*}
With this definition of $S_{\alpha}$, we have
\begin{align*}
S_{\alpha}(I-S_{\alpha})^{-1} = -A \equivalent I-A = I+S_{\alpha}(I-S_{\alpha})^{-1} = (I-S_{\alpha})^{-1}.
\end{align*}
It follows that
\begin{align}
\det(I-A) = \det(I-S_{\alpha})^{-1}.
\label{equation:normalize-detI-A-S-alpha}
\end{align}
Define $m_{\alpha} \in \Hcal$ by
\begin{align*}
	%	&(I-S_{\alpha})^{-1}C_{*}^{-1/2}(m_{\alpha} - m_{*}) 
	%	\\
	%	&= (1-\alpha)(I-S_0)^{-1}C_{*}^{-1/2}(m_0 - m_{*}) + \alpha(I-S_1)^{-1}C_{*}^{-1/2}(m_1 - m_{*})
	%	\\
	%	&\equivalent C_{*}^{-1/2}(m_{\alpha}-m_{*})
	%	\\
	%	& = (I-S_{\alpha})[(1-\alpha)(I-S_0)^{-1}C_{*}^{-1/2}(m_0 - m_{*}) + \alpha(I-S_1)^{-1}C_{*}^{-1/2}(m_1 - m_{*})]
	%\end{align*}
	%\begin{align*}
	%	m_{\alpha} - m_{*} = C_{*}^{1/2}(I-S_{\alpha})[(1-\alpha)(I-S_0)^{-1}C_{*}^{-1/2}(m_0 - m_{*}) + \alpha(I-S_1)^{-1}C_{*}^{-1/2}(m_1 - m_{*})]
	%\\
	%	&\equivalent
	&m_{\alpha} = m_{*} + C_{*}^{1/2}(I-S_{\alpha})[(1-\alpha)(I-S_0)^{-1}C_{*}^{-1/2}(m_0 - m_{*}) 
	\\
	& \quad \quad \quad \quad \quad \quad \quad + \alpha(I-S_1)^{-1}C_{*}^{-1/2}(m_1 - m_{*})].
\end{align*}
Then clearly $m_{\alpha} - m_{*} \in \range(C_{*}^{1/2})$ and
\begin{align*}
	&(I-S_{\alpha})^{-1}C_{*}^{-1/2}(m_{\alpha} - m_{*}) 
	\\
	&= (1-\alpha)(I-S_0)^{-1}C_{*}^{-1/2}(m_0 - m_{*}) + \alpha(I-S_1)^{-1}C_{*}^{-1/2}(m_1 - m_{*}) =b.
\end{align*}
It follows that
\begin{align}
\exp\left(\frac{1}{2}||(I-A)^{-1/2}b||^2\right) = \exp\left(\frac{1}{2}||(I-S_{\alpha})^{-1/2}C_{*}^{-1/2}(m_{\alpha} - m_{*})||^2 \right).
\label{equation:normalize-J-b}
\end{align}
Combining Eqs.\eqref{equation:normalize-J-A}, \eqref{equation:normalize-detI-A-S-alpha}, \eqref{equation:normalize-J-b} gives
\begin{align}
J = \det(I-S_{\alpha})^{1/2}\exp\left(\frac{1}{2}||(I-S_{\alpha})^{-1/2}C_{*}^{-1/2}(m_{\alpha} - m_{*})||^2 \right).
\label{equation:normalizing-J}
\end{align}
	Combining Eqs.\eqref{equation:normalizing-1} and \eqref{equation:normalizing-J},
	the overall normalizing factor is thus
	\begin{align*}
		&Z_{\alpha}^G(p_0:p_1) = \int_{\Hcal}p_0^{1-\alpha}(x)p_1^{\alpha}(x) d\mu_{*}(x)
		\\
		& =  [\det(I-S_0)]^{-\frac{1-\alpha}{2}}[\det(I-S_1)]^{-\frac{\alpha}{2}}
		\\
		& \quad \times \exp\left[-\frac{1-\alpha}{2}||(I-S_0)^{-1/2}C_{*}^{-1/2}(m_0 - m_{*})||^2\right]
		\\
		& \quad \times 
		\exp\left[-\frac{\alpha}{2}||(I-S_1)^{-1/2}C_{*}^{-1/2}(m_1 - m_{*})||^2\right]
		\\
		&\quad \times \det(I-S_{\alpha})^{1/2}\exp\left(\frac{1}{2}||(I-S_{\alpha})^{-1/2}C_{*}^{-1/2}(m_{\alpha} - m_{*})||^2\right).
	\end{align*}
	This gives the first expression for $Z_{\alpha}^G(p_0:p_1)$ in Eq.\eqref{equation:normalizing-factor-Tr}.
With $S_{\alpha} = I - [(1-\alpha)(I-S_0)^{-1} + \alpha (I-S_1)^{-1}]^{-1}$, we have
\begin{align*}
	&[\det(I-S_0)]^{-\frac{1-\alpha}{2}}[\det(I-S_1)]^{-\frac{\alpha}{2}}\det(I-S_{\alpha})^{1/2}
	\\
	& = [\det(I-S_0)]^{-\frac{1-\alpha}{2}}[\det(I-S_1)]^{-\frac{\alpha}{2}}\det[(1-\alpha)(I-S_0)^{-1} + \alpha (I-S_1)^{-1}]^{-1/2}
	\\
	 & = [\det(I-S_0)]^{\frac{\alpha}{2}}[\det(I-S_1)]^{-\frac{\alpha}{2}}\det[(1-\alpha)I+ \alpha (I-S_0)(I-S_1)^{-1}]^{-1/2}
	\\
	& = [\det(I-S_0)]^{\frac{\alpha}{2}}[\det(I-S_1)]^{-\frac{\alpha}{2}}
	\\
	& \quad \times \det[(1-\alpha)I+ \alpha (I-S_1)^{-1/2}(I-S_0)(I-S_1)^{-1/2}]^{-1/2}
	\\
	& = [\det((I-S_1)^{-1/2}(I-S_0)(I-S_1)^{-1/2})]^{\alpha/2}
	\\
	& \quad \times \det[(1-\alpha)I+ \alpha (I-S_1)^{-1/2}(I-S_0)(I-S_1)^{-1/2}]^{-1/2}
	\\
	& = \det[(I+A)^{\alpha/2}]\det[(1-\alpha) I + \alpha (I+A)]^{-1/2} 
	\\
	&= \det[(1-\alpha)(I+A)^{-\alpha} + \alpha (I+A)^{1-\alpha}]^{-1/2}.
\end{align*}
Here $I + A = (I-S_1)^{-1/2}(I-S_0)(I-S_1)^{-1/2} > 0$, with $A =(I-S_1)^{-1/2}(S_1-S_0)(I-S_1)^{-1/2}$.
% \in \Sym(\Hcal)\cap \HS(\Hcal)$. 
This gives the second expression for $Z_{\alpha}^G(p_0:p_1)$ in Eq.\eqref{equation:normalizing-factor-HS}.
%By Lemma \ref{lemma:sum-HS-Tr}, for $S_0,S_1 \in \SymHS(\Hcal)_{< I}$, we have
%$(1-\alpha)(I-S_0)^{-\alpha}(I-S_1)^{\alpha} + \alpha(I-S_0)^{1-\alpha}(I-S_1)^{-(1-\alpha)} = I+A$, 
%with $A \in \Tr(\Hcal)$. Thus $\det(I+A)$ is well-defined and furthermore $\det(I+A)> 0$ from the expression on the last line.
For $S_0,S_1 \in \SymHS(\Hcal)_{< I}$, we have $A \in \Sym(\Hcal)\cap \HS(\Hcal)$.
By Lemmas \ref{lemma:A-alpha-I+A-HS-Tr} and \ref{lemma:I+A-power-Tr}, with $A \in \Sym(\Hcal)\cap \HS(\Hcal)$ and $I+A > 0$, we have
$(1-\alpha)(I+A)^{-\alpha} + \alpha (I+A)^{1-\alpha}\in \PC_1(\Hcal)$ and $\det[(1-\alpha)(I+A)^{-\alpha} + \alpha (I+A)^{1-\alpha}]^{-1/2} > 0$.
Hence Eq.\eqref{equation:normalizing-factor-HS-1} is valid for $S_0,S_1 \in  \SymHS(\Hcal)_{< I}$.
\qed
\end{proof}

\begin{proof}
	[\textbf{of Theorem \ref{theorem:geometric-interpolation-equivalent-Gaussian-Hilbert-space} - The case $S_i \in \SymTr(\Hcal)_{< I}$, $i=0,1$}]
	By Theorem \ref{theorem:radon-nikodym-infinite}, for $S_i \in \SymTr(\Hcal)_{< I}$, $i=0,1$,
	\begin{align*}
		%	\label{equation:RN-traceclass}
		p_i(x)= \frac{d\mu_i}{d\mu_{*}}(x) &= [\det(I-S_i)]^{-1/2}
		\\
		& \quad \times \exp\left\{-\frac{1}{2}\la C_{*}^{-1/2}(x-m_{*}), S_i(I-S_i)^{-1}C_{*}^{-1/2}(x-m_{*})\ra \right\}
		\nonumber
		\\
		&\quad \times \exp(\la C_{*}^{-1/2}(x-m_{*}), (I-S_i)^{-1}C_{*}^{-1/2}(m_i - m_{*})\ra)
		\nonumber
		\\
		&\quad \times \exp\left[-\frac{1}{2}||(I-S_i)^{-1/2}C_{*}^{-1/2}(m_i - m_{*})||^2\right].
		\nonumber
	\end{align*}
	Thus for $\alpha \in [0,1]$,
	\begin{align*}
		&p_0^{1-\alpha}(x)p_1^{\alpha}(x) = 
		[\det(I-S_0)]^{-\frac{1-\alpha}{2}}[\det(I-S_1)]^{-\frac{\alpha}{2}}
		\\
		& \quad \times \exp\left\{-\frac{1}{2}\la C_{*}^{-1/2}(x-m_{*}), [(1-\alpha)S_0(I-S_0)^{-1} + \alpha S_1(I-S_1)^{-1}]\right.
		\\
		&\quad \quad \quad \quad \quad \left.C_{*}^{-1/2}(x-m_{*})\ra \right\}
		\\
		& \quad \times \exp(\la C_{*}^{-1/2}(x-m_{*}), (1-\alpha)(I-S_0)^{-1}C_{*}^{-1/2}(m_0 - m_{*})\ra)
		\\
		& \quad \times  \exp(\la C_{*}^{-1/2}(x-m_{*}), \alpha(I-S_1)^{-1}C_{*}^{-1/2}(m_1 - m_{*})\ra)
		\\
		&\quad \times \exp\left[-\frac{1-\alpha}{2}||(I-S_0)^{-1/2}C_{*}^{-1/2}(m_0 - m_{*})||^2\right]
		\\
		&\quad \times \exp\left[-\frac{\alpha}{2}||(I-S_1)^{-1/2}C_{*}^{-1/2}(m_1 - m_{*})||^2\right].
	\end{align*}
	By Proposition \ref{proposition:normalizing factor}, the normalizing factor is given by
	\begin{align*}
	&Z_{\alpha}^G(p_0:p_1) = \int_{\Hcal}p_0^{1-\alpha}(x)p_1^{\alpha}(x) d\mu_{*}(x)
\\
& =  [\det(I-S_0)]^{-\frac{1-\alpha}{2}}[\det(I-S_1)]^{-\frac{\alpha}{2}}
\\
& \quad \times \exp\left[-\frac{1-\alpha}{2}||(I-S_0)^{-1/2}C_{*}^{-1/2}(m_0 - m_{*})||^2\right]
\\
& \quad \times 
\exp\left[-\frac{\alpha}{2}||(I-S_1)^{-1/2}C_{*}^{-1/2}(m_1 - m_{*})||^2\right]
\\
&\quad \times \det(I-S_{\alpha})^{1/2}\exp\left(\frac{1}{2}||(I-S_{\alpha})^{-1/2}C_{*}^{-1/2}(m_{\alpha} - m_{*})||^2\right).
	\end{align*}
	By definition, the geometric interpolation between $\mu_0$ and $\mu_1$, for $\alpha \in [0,1]$, is
	\begin{align*}
	&(p_0p_1)^G_{\alpha}(x) = \frac{p_0^{1-\alpha}(x)p_1^{\alpha}(x)}{Z_{\alpha}^G(p_0:p_1) }
	\\
	& = \det(I-S_{\alpha})^{-1/2}\exp\left(-\frac{1}{2}||(I-S_{\alpha})^{-1/2}C_{*}^{-1/2}(m_{\alpha} - m_{*})||^2\right)
	\\
	&\quad \times \exp\left\{-\frac{1}{2}\la C_{*}^{-1/2}(x-m_{*}), S_{\alpha}(I-S_{\alpha})^{-1}C_{*}^{-1/2}(x-m_{*})\ra \right\}
	\\
	&\quad \times \exp(\la C_{*}^{-1/2}(x-m_{*}), (I-S_{\alpha})^{-1}C_{*}^{-1/2}(m_{\alpha} - m_{*})\ra)
	\\
	& 
	= \frac{d\mu_{\alpha}}{d\mu_{*}}(x),
%	& \quad \times \exp(\la C_{*}^{-1/2}(x-m_{*}), (1-\alpha)(I-S_0)^{-1}C_{*}^{-1/2}(m_0 - m_{*})\ra)
	\\
%	& \quad \times  \exp(\la C_{*}^{-1/2}(x-m_{*}), \alpha(I-S_1)^{-1}C_{*}^{-1/2}(m_1 - m_{*})\ra).
	\end{align*}
	where $\mu_{\alpha} = \Ncal(m_{\alpha}, C_{\alpha}) \in \Gauss(\Hcal, \mu_{*})$.
	%
	%The desired expression for $S_{\alpha}$ is obtained by setting
	%
	%Since $S_0, S_1 \in \SymTr(\Hcal)$, it follows that $S_{\alpha} \in \SymTr(\Hcal)$ also.
	The corresponding covariance operator $C_{\alpha}$ is given by
	\begin{align*}
		C_{\alpha} = C_{*}^{1/2}(I-S_{\alpha})C_{*}^{1/2} = C_{*}^{1/2} [(1-\alpha)(I-S_0)^{-1} + \alpha (I-S_1)^{-1}]^{-1}C_{*}^{1/2}.
	\end{align*}
	\begin{comment}
	The desired expression for $m_{\alpha} \in \Hcal$, assuming $m_{\alpha}-m_{*} \in \range(C_{*}^{1/2})$, is obtained by setting
	\begin{align*}
		&(I-S_{\alpha})^{-1}C_{*}^{-1/2}(m_{\alpha} - m_{*}) 
		\\
		&= (1-\alpha)(I-S_0)^{-1}C_{*}^{-1/2}(m_0 - m_{*}) + \alpha(I-S_1)^{-1}C_{*}^{-1/2}(m_1 - m_{*})
		%	\\
		%	&\equivalent C_{*}^{-1/2}(m_{\alpha}-m_{*})
		%	\\
		%	& = (I-S_{\alpha})[(1-\alpha)(I-S_0)^{-1}C_{*}^{-1/2}(m_0 - m_{*}) + \alpha(I-S_1)^{-1}C_{*}^{-1/2}(m_1 - m_{*})]
		%\end{align*}
		%\begin{align*}
		%	m_{\alpha} - m_{*} = C_{*}^{1/2}(I-S_{\alpha})[(1-\alpha)(I-S_0)^{-1}C_{*}^{-1/2}(m_0 - m_{*}) + \alpha(I-S_1)^{-1}C_{*}^{-1/2}(m_1 - m_{*})]
		\\
		&\equivalent
		m_{\alpha} = m_{*} + C_{*}^{1/2}(I-S_{\alpha})[(1-\alpha)(I-S_0)^{-1}C_{*}^{-1/2}(m_0 - m_{*}) 
		\\
		& \quad \quad \quad \quad \quad \quad \quad + \alpha(I-S_1)^{-1}C_{*}^{-1/2}(m_1 - m_{*})].
	\end{align*}
	\end{comment}
	\qed
\end{proof}

\begin{proposition}
\label{proposition:KL-log-Radon-Nikodym-convergence}
Let $\Xcal$ be a Polish space and $\Bsc(\Xcal)$ its Borel $\sigma$-algebra.
%Let $\mu\in \Pcal(\Xcal)$ be a fixed probability measure. 
Let
$\mu, \nu \in \Pcal(\Xcal)$ be such that $\nu << \mu$. Then
\begin{align}
\KL(\nu||\mu) \leq \left\|\log\left(\frac{d\nu}{d\mu}\right) \right\|_{\Lcal^2(\Hcal,\nu)}.
\end{align}
Consequently,
\begin{align}
\left\|\log\left(\frac{d\nu}{d\mu}\right) \right\|_{\Lcal^2(\Hcal,\nu)} = 0 \equivalent \KL(\nu||\mu ) = 0 \equivalent \nu = \mu.
\end{align}
In particular, let $\{\mu_k\}_{k \in \Nbb} \in \Pcal(\Xcal)$ be such that $\mu_k \sim \mu$ $\forall k \in \Nbb$. Then
\begin{align}
\lim_{k \approach \infty}\left\|\log\left(\frac{d\mu_k}{d\mu}\right) \right\|_{\Lcal^2(\Hcal,\mu)} = 0 \imply \lim_{k \approach \infty}\KL(\mu||\mu_k) = 0.
\end{align}
%Then
%\begin{align}
%\lim_{k \approach \infty}\KL(\mu_k||\mu) = 0.
%\end{align}
\end{proposition}
\begin{proof}
%\begin{align*}
%\KL(P||Q) = \int_{\Xcal}\log\left(\frac{dP}{dQ}(x)\right)dP(x)
%\end{align*}
Since $\nu(\Xcal) =1$, by the Cauchy-Schwarz inequality,
\begin{align*}
[\KL(\nu||\mu)]^2 &= \left[\int_{\Xcal}\log\left(\frac{d\nu}{d\mu}\right)d\nu(x)\right]^2
\leq \int_{\Xcal}\left[\log\left(\frac{d\nu}{d\mu}\right)\right]^2d\nu(x)
\\
& = \left\|\log\left(\frac{d\nu}{d\mu}\right) \right\|^2_{\Lcal^2(\Hcal,\nu)},
\end{align*}
from which the desired result follows.\qed
\end{proof}

\begin{proposition}
	[\cite{Minh:2022KullbackGaussian}, Theorem 11]
	\label{proposition:limit-logdet-A-B-alpha}
	Let $0 \leq \alpha \leq 1$.
	Let $A\in \Sym(\Hcal) \cap \HS(\Hcal)$, $I+A > 0$. Then
	\begin{align}
		0 &\leq \log\det[(1-\alpha)(I+A)^{-\alpha} + \alpha (I+A)^{1-\alpha}] 
		\nonumber
		\\
		&\leq \alpha(1-\alpha)||(I+A)^{-1}||\;||A||^2_{\HS}.
	\end{align}
	Let $B\in \Sym(\Hcal) \cap \HS(\Hcal)$, $I+B > 0$. Then
	\begin{align}
		&\left|\log\det[(1-\alpha)(I+A)^{-\alpha} + \alpha (I+A)^{1-\alpha}] \right.
		\nonumber
		\\
		&\left. -\log\det[(1-\alpha)(I+B)^{-\alpha} + \alpha (I+B)^{1-\alpha}]]\right| 
		\\
		&\leq \alpha (1-\alpha)||(I+A)^{-1}||\;||(I+B)^{-1}||
		\nonumber
		\\
		& \quad \times [||A||_{\HS} + ||B||_{\HS} + ||A||_{\HS}||B||_{\HS}]||A-B||_{\HS}.
	\end{align}
	Consequently, let $\{A_k\}_{k\in \Nbb} \in \Sym(\Hcal) \cap \HS(\Hcal)$, $I+A_k > 0$ $\forall k \in \Nbb$, be such that
	$\lim_{k \approach \infty}||A_k - A||_{\HS} = 0$. Then
	\begin{align}
		&\lim_{k \approach \infty}\log\det[(1-\alpha)(I+A_k)^{-\alpha} + \alpha (I+A_k)^{1-\alpha}]
		\nonumber
		\\
		&=
		\log\det[(1-\alpha)(I+A)^{-\alpha} + \alpha (I+A)^{1-\alpha}].
	\end{align}
	Let $\beta \in \R$ be fixed but arbitrary. Then
	\begin{align}
		&\lim_{k \approach \infty}\det[(1-\alpha)(I+A_k)^{-\alpha} + \alpha (I+A_k)^{1-\alpha}]^{\beta}
		\nonumber
		\\
		&=
		\det[(1-\alpha)(I+A)^{-\alpha} + \alpha (I+A)^{1-\alpha}]^{\beta}.
	\end{align}
\end{proposition}

The following is a special case of Theorem 33 in \cite{Minh:2019AlphaBeta}.

\begin{proposition}
	[\cite{Minh:2019AlphaBeta}, Theorem 33]
	\label{proposition:limit-(I+A)-power-r}.
	Let $r \in \R$ be fixed but arbitrary.
	Let $A,\{A_k\}_{k\in \Nbb} \in \Sym(\Hcal) \cap \HS(\Hcal)$ be such that $I + A > 0$, $I+A_k > $ $\forall k \in \Nbb$. Assume that
	$\lim_{k \approach \infty}||A_k - A||_{\HS} = 0$. Then
	\begin{align}
		\lim_{k \approach \infty}||(I+A_k)^r - (I+A)^r||_{\HS} = 0.
	\end{align}
\end{proposition}

The following result then follows from Proposition \ref{proposition:limit-(I+A)-power-r} by repeated application of the triangle inequality.

\begin{corollary}
	\label{corollary:limit-product-I+A-I+B}
	Let $A,B,\{A_k\}_{k \in \Nbb}, \{B_k\}_{k \in \Nbb} \in \Sym(\Hcal) \cap \HS(\Hcal)$ be such that $I + A > 0$, $I+B > 0$, $I+A_k > $, $I+B_k > 0$ $\forall k \in \Nbb$. Assume that
	$\lim_{k \approach \infty}||A_k - A||_{\HS} = 0$, $\lim_{k \approach \infty}||B_k - B||_{\HS} = 0$. Then
	\begin{align}
		&\lim_{k \approach \infty}||(I+B_k)^{-1/2}(I+A_k)(I+B_k)^{-1/2} 
		\nonumber
		\\
		& \quad \quad \quad - (I+B)^{-1/2}(I+A)(I+B)^{-1/2}||_{\HS} 
		\\
		&= \lim_{k \approach \infty}||(I+B_k)^{-1/2}(A_k-B_k)(I+B_k)^{-1/2} 
		\nonumber
		\\
		&\quad \quad \quad - (I+B)^{-1/2}(A-B)(I+B)^{-1/2}||_{\HS}
		= 0.
	\end{align}
\end{corollary}

\begin{corollary}
	\label{corollary:limit-det-S0-S1}
	Let $0 \leq \alpha \leq 1$.
	Let $\beta \in \R$ be fixed but arbitrary. Let $S_0, S_1$, $\{S_0^k\}_{k\in \Nbb}$, $\{S_1^k\}_{k \in \Nbb} \in \SymHS(\Hcal)_{< I}$.
	Assume that $\lim_{k \approach \infty}||S_0^k - S_0||_{\HS} = 0$, $\lim_{k \approach \infty}||S_1^k - S_1||_{\HS} = 0$.
	Let $A= (I-S_1)^{-1/2}(S_1-S_0)(I-S_1)^{-1/2}$, $A^k = (I-S_1^k)^{-1/2}(S_1^k-S_0^k)(I-S^k_1)^{-1/2}$.
	Then
	\begin{align}
		&\lim_{k \approach \infty}\det[(1-\alpha)(I+A^k)^{-\alpha} + \alpha (I+A^k)^{1-\alpha}]^{\beta}
		\nonumber
		\\
		& = \det[(1-\alpha)(I+A)^{-\alpha} + \alpha (I+A)^{1-\alpha}]^{\beta}.
	\end{align}
	%\begin{align}
	%&\lim_{k \approach \infty}
	%\det[(1-\alpha)(I-S_0^k)^{-\alpha}(I-S_1^k)^{\alpha} + \alpha(I-S_0^k)^{1-\alpha}(I-S_1^k)^{-(1-\alpha)}]^{\beta}
	%\nonumber
	%\\
	%&=\det[(1-\alpha)(I-S_0)^{-\alpha}(I-S_1)^{\alpha} + \alpha(I-S_0)^{1-\alpha}(I-S_1)^{-(1-\alpha)}]^{\beta}.
	%\end{align}
\end{corollary}
\begin{proof}
	By Lemma \ref{lemma:A-alpha-I+A-HS-Tr}, we have $(1-\alpha)(I+A)^{-\alpha} + \alpha (I+A)^{1-\alpha} \in \PC_1(\Hcal)$ 
	and $(1-\alpha)(I+A^k)^{-\alpha} + \alpha (I+A^k)^{1-\alpha} \in \PC_1(\Hcal)$ $\forall k \in \Nbb$.
	By Lemma \ref{lemma:I+A-power-Tr}, we have $\forall \beta \in \R$,  $[(1-\alpha)(I+A)^{-\alpha} + \alpha (I+A)^{1-\alpha}]^{\beta} \in \PC_1(\Hcal)$ 
	and $[(1-\alpha)(I+A^k)^{-\alpha} + \alpha (I+A^k)^{1-\alpha}]^{\beta} \in \PC_1(\Hcal)$ $\forall k \in \Nbb$.
	\begin{comment}
		Let $A= (I-S_1)^{-1/2}(S_1-S_0)(I-S_1)^{-1/2} \in\Sym(\Hcal) \cap\HS(\Hcal)$.
		Let $A^k = (I-S_1^k)^{-1/2}(S_1^k-S_0^k)(I-S^k_1)^{-1/2}$. Then
		\begin{align*}
			&\det[(1-\alpha)(I-S_0)^{-\alpha}(I-S_1)^{\alpha} + \alpha(I-S_0)^{1-\alpha}(I-S_1)^{-(1-\alpha)}]
			\\
			&= \det[(1-\alpha) I + \alpha (I-S_0)(I-S_1)^{-1}]\det(I-S_0)^{-\alpha}\det(I-S_1)^{\alpha}
			\\
			& = \det[(1-\alpha) I + \alpha (I-S_1)^{-1/2}(I-S_0)(I-S_1)^{-1/2}]
			\\
			& \quad \times \det[[(I-S_1)^{-1/2}(I-S_0)(I-S_1)^{-1/2}]^{-\alpha}]
			\\
			& = \det[(1-\alpha)I + \alpha (I+A)]\det[(I+A)^{-\alpha}]
			\\
			& = \det[(1-\alpha)(I+A)^{-\alpha} + \alpha (I+A)^{1-\alpha}]
		\end{align*}
		Similarly,
		\begin{align*}
			&\det[(1-\alpha)(I-S_0^k)^{-\alpha}(I-S_1^k)^{\alpha} + \alpha(I-S_0^k)^{1-\alpha}(I-S_1^k)^{-(1-\alpha)}]
			\\
			&=\det[(1-\alpha)(I+A^k)^{-\alpha} + \alpha (I+A^k)^{1-\alpha}].
		\end{align*}
	\end{comment}
	By Corollary \ref{corollary:limit-product-I+A-I+B},
	\begin{align*}
		\lim_{k \approach \infty}	||A^k - A||_{\HS} = 0.  
	\end{align*}
	Thus it follows from Proposition \ref{proposition:limit-logdet-A-B-alpha} that $\forall \beta \in \R$,
	\begin{align*}
		&\lim_{k \approach \infty}\det[(1-\alpha)(I+A^k)^{-\alpha} + \alpha (I+A^k)^{1-\alpha}]^{\beta}
		\\
		& = \det[(1-\alpha)(I+A)^{-\alpha} + \alpha (I+A)^{1-\alpha}]^{\beta}.
	\end{align*}
	This gives the desired limit.
	\qed
	%\begin{align*}
	%&\lim_{k \approach \infty}\det[(1-\alpha)(I-S_0^k)^{-\alpha}(I-S_1^k)^{\alpha} + \alpha(I-S_0^k)^{1-\alpha}(I-S_1^k)^{-(1-\alpha)}]^{\beta}
	%\\
	%&=\det[(1-\alpha)(I-S_0)^{-\alpha}(I-S_1)^{\alpha} + \alpha(I-S_0)^{1-\alpha}(I-S_1)^{-(1-\alpha)}]^{\beta}.
	%\end{align*}
	%\qed
\end{proof}

\begin{proof}
	[\textbf{of Theorem \ref{theorem:geometric-interpolation-equivalent-Gaussian-Hilbert-space} - the general case
		$S_i \in \SymHS(\Hcal)_{<I}$, $i=0,1$}]

For $i=0,1$, let $\{\mu_i^k = \Ncal(m_i,C_i^k)\}_{k\in \Nbb} \in \Gauss(\Hcal, \mu_{*})$, be such that
$C_i^k = C_{*}^{1/2}(I-S_i^k)C_{*}^{1/2}$, $S_i^k \in \SymTr(\Hcal)_{< I}$, 
$\lim_{k \approach \infty}||S_i^k - S_i||_{\HS} = 0$. Let $p_i^k = \frac{d\mu_i^k}{d\mu_{*}}$. 
Let $\nu \in \Gauss(\Hcal,\mu_{*})$ be fixed but arbitrary. By Corollary \ref{corollary:log-Radon-Nikodym-L2-convergence}, with $p_i = \frac{d\mu_i}{d\mu_{*}}$,
\begin{align*}
\lim_{k \approach \infty}	||\log(p^k_i) - \log(p_i)||_{\Lcal^2(\Hcal,\nu)} = 0, \;\; i=0,1.
\end{align*}
%On the other hand, with $p_i = \frac{d\mu_i}{d\mu_{*}}$, $i=0,1$, we have
It follows that for $\alpha \in [0,1]$, with the limits in the $\Lcal^2(\Hcal,\nu)$ sense,
\begin{align}
	p_0^{1-\alpha}(x)p_1^{\alpha}(x) &= \exp\left((1-\alpha)\log(p_0(x)) + \alpha \log(p_1(x))\right)
	\nonumber
	\\
	\equivalent \log[p_0^{1-\alpha}(x)p_1^{\alpha}(x)] &= (1-\alpha)\log(p_0(x)) + \alpha \log(p_1(x))
	\nonumber
	\\
	&= \lim_{k \approach \infty}[(1-\alpha)\log(p_0^k(x)) + \alpha \log(p_1^k(x))]
	\nonumber
	\\
	& = \lim_{k \approach \infty}\log[(p_0^k(x))^{1-\alpha}(p_1^k(x))^{\alpha}].
	\label{equation:proof-general-1}
\end{align}
Let $A^k = (I-S_1^k)^{-1/2}(S_1^k - S_0^k)(I-S_1^k)^{-1/2}$.
By Proposition \ref{proposition:normalizing factor},
\begin{align*}
	&Z_{\alpha}^G(p^k_0:p^k_1)
	%\nonumber
	%\\
	%& 
	= 
	\int_{\Hcal}(p^k_0)^{1-\alpha}(x)(p^k_1)^{\alpha}(x) \; d\mu_{*}(x)
	\\
	&=
\det[(1-\alpha)(I+A^k)^{-\alpha} + \alpha (I+A^k)^{1-\alpha}]^{-1/2}	
	\nonumber
	\\
	& \quad \times \exp\left[-\frac{1-\alpha}{2}||(I-S^k_0)^{-1/2}C_{*}^{-1/2}(m_0 - m_{*})||^2\right]
	\nonumber
	\\
	& \quad \times 
	\exp\left[-\frac{\alpha}{2}||(I-S^k_1)^{-1/2}C_{*}^{-1/2}(m_1 - m_{*})||^2\right]
	\nonumber
	\\
	& \quad \times \exp\left(\frac{1}{2}||(I-S^k_{\alpha})^{-1/2}C_{*}^{-1/2}(m_{\alpha} - m_{*})||^2\right).
%	\label{equation:normalizing-factor-HS}
\end{align*}
Let $A = (I-S_1)^{-1/2}(S_1 - S_0)(I-S_1)^{-1/2}$.
Define
\begin{align*}
	&Z
	%_{\alpha}^G(p_0:p_1)
	%\nonumber
	%\\
	%& = 
	%\int_{\Hcal}p_0^{1-\alpha}(x)p_1^{\alpha}(x) d\mu_{*}(x)
	%&
	=\det[(1-\alpha)(I+A)^{-\alpha} + \alpha (I+A)^{1-\alpha}]^{-1/2}	
	\\
	& \quad \times \exp\left[-\frac{1-\alpha}{2}||(I-S_0)^{-1/2}C_{*}^{-1/2}(m_0 - m_{*})||^2\right]
	\nonumber
	\\
	& \quad \times 
	\exp\left[-\frac{\alpha}{2}||(I-S_1)^{-1/2}C_{*}^{-1/2}(m_1 - m_{*})||^2\right]
	\nonumber
	\\
	& \quad \times \exp\left(\frac{1}{2}||(I-S_{\alpha})^{-1/2}C_{*}^{-1/2}(m_{\alpha} - m_{*})||^2\right).
%	\label{equation:normalizing-factor-HS}
\end{align*}
%By the continuity of the Fredholm determinant in the trace norm, we have
%\begin{align*}
%(p_0p_1)^G_{\alpha} = \frac{p_0^{1-\alpha}(x)p_1^{\alpha}(x)}{\int_{\Hcal}p_0^{1-\alpha}(x)p_1^{\alpha}(x)d\mu_{*}(x)}
%\end{align*}
%On the other hand,
Since $\lim_{k \approach \infty}||S_i^k - S_i||_{\HS} = 0$, by Corollary \ref{corollary:limit-det-S0-S1},
\begin{align*}
	&\lim_{k \approach \infty}\det[(1-\alpha)(I+A^k)^{-\alpha} + \alpha (I+A^k)^{1-\alpha}]^{-1/2}
	\nonumber
	\\
	& = \det[(1-\alpha)(I+A)^{-\alpha} + \alpha (I+A)^{1-\alpha}]^{-1/2}.
\end{align*}
Thus it is straightforward to see that
\begin{align*}
	\lim_{k \approach \infty}Z_{\alpha}^G(p^k_0:p^k_1) = Z.
\end{align*}
It follows that, with the limit being in the $\Lcal^2(\Hcal, \nu)$ sense,
\begin{align}
	\log\left\{\frac{p_0^{1-\alpha}(x)p_1^{\alpha}(x)]}{Z}\right\} =  \lim_{k \approach \infty}\log\left\{\frac{(p_0^k(x))^{1-\alpha}(p_1^k(x))^{\alpha}}{Z_{\alpha}^G(p^k_0:p^k_1)}\right\}.
	\label{equation:proof-general-4}
\end{align}
On the other hand, 
from the first part of the theorem for the case $S_i \in \SymTr(\Hcal)_{< I}$, the geometric interpolation between $p_0^k$ and $p_1^k$ is 
\begin{align}
(p_0^kp_1^k)^G_{\alpha} = \frac{(p_0^k(x))^{1-\alpha}(p_1^k(x))^{\alpha}}{Z_{\alpha}^G(p^k_0:p^k_1)} = \frac{d\mu^k_{\alpha}}{d\mu_{*}}, \;\; \mu^k_{\alpha} = \Ncal(m_{\alpha},C^k_{\alpha}).
\label{equation:proof-general-2}
\end{align}
Here $C^k_{\alpha} = C_{*}^{1/2}(I-S_{\alpha}^k)C_{*}^{1/2}$, $S^k_{\alpha} = I - [(1-\alpha)(I-S^k_0)^{-1} + \alpha (I-S^k_1)^{-1}]^{-1} \in \SymTr(\Hcal)_{< I}$, $ k \in \Nbb$.
With $S_{\alpha} = I - [(1-\alpha)(I-S_0)^{-1} + \alpha (I-S_1)^{-1}]^{-1} \in \SymHS(\Hcal)_{< I}$ and $\lim_{k \approach \infty}||
S_i^k - S_i||_{\HS} = 0$, $i=0,1$, we have
\begin{align*}
\lim_{k \approach \infty}||S^k_{\alpha} - S_{\alpha}||_{\HS} = 0.
\end{align*}
%By the continuity of the Hilbert-Carleman determinant in the Hilbert-Schmidt norm, 
%\begin{align*}
%\imply \lim_{k \approach \infty}\KL(\mu^k_{\alpha}||\mu_{\alpha}) = 0.
%\end{align*}
It follows from Corollary \ref{corollary:log-Radon-Nikodym-L2-convergence} then that
\begin{align}
	\lim_{k \approach \infty}\left\| \log\left\{{\frac{d\mu^k_{\alpha}}{d\mu_{*}}(x)}\right\}- \log\left\{{\frac{d\mu_{\alpha}}{d\mu_{*}}(x)}\right\}\right\|_{\Lcal^2(\Hcal,\nu)} = 0.
	\label{equation:proof-general-3}
\end{align}
%From Eqs.\eqref{equation:proof-general-1}, \eqref{equation:proof-general-2}, \eqref{equation:proof-general-3},
%there must exist $Z > 0$ such that
%\begin{align*}
%	\lim_{k \approach \infty}Z_{\alpha}^G(p^k_0:p^k_1) = Z.
%\end{align*}
%It follows that, with the limit being in the $\Lcal^2(\Hcal, \nu)$ sense,
%\begin{align}
%	\log\left\{\frac{p_0^{1-\alpha}(x)p_1^{\alpha}(x)]}{Z}\right\} =  \lim_{k \approach \infty}\log\left\{\frac{(p_0^k(x))^{1-\alpha}(p_1^k(x))^{\alpha}}{Z_{\alpha}^G(p^k_0:p^k_1)}\right\}.
%	\label{equation:proof-general-4}
%\end{align}
Combining Eqs.\eqref{equation:proof-general-4}, \eqref{equation:proof-general-2}, \eqref{equation:proof-general-3}, %\eqref{equation:proof-general-4}, 
we have that, as elements in $\Lcal^2(\Hcal,\nu)$,
\begin{align*}
\log\left\{\frac{p_0^{1-\alpha}(x)p_1^{\alpha}(x)}{Z}\right\}
=\log\left\{{\frac{d\mu_{\alpha}}{d\mu_{*}}(x)}\right\}.
\end{align*}
In particular, setting $\nu = \mu_{\alpha}$ shows that they are identical as elements in $\Lcal^2(\Hcal, \mu_{\alpha})$.
Let $P = \mu_{*}\frac{p_0^{1-\alpha}(x)p_1^{\alpha}(x)}{Z}$, then $P \sim \mu_{*}$ and
$\log\left(\frac{dP}{d\mu_{*}}\right) = \log\left(\frac{d\mu_{\alpha}}{d\mu_{*}}\right)$ as elements in $\Lcal^2(\Hcal,\mu_{\alpha})$. By the chain rule,
this implies that $\left\|\log\left(\frac{d\mu_{\alpha}}{dP}\right)\right\|_{\Lcal^2(\Hcal,\mu_{\alpha})} = 0 \equivalent \mu_{\alpha} = P$ by Proposition \ref{proposition:KL-log-Radon-Nikodym-convergence}.
 Thus $\frac{dP}{d\mu_{*}} = \frac{p_0^{1-\alpha}(x)p_1^{\alpha}(x)}{Z} = 
\frac{d\mu_{\alpha}}{d\mu_{*}}$ as elements in $\Lcal^1(\Hcal, \mu_{*})$ and we must have
$Z^G_{\alpha}(p_0:p_1) = Z$.
%
%
%In particular, setting $\nu = \mu_{\alpha}$ and invoking Proposition \ref{proposition:KL-log-Radon-Nikodym-convergence}
%\begin{align*}
%\lim_{k \approach \infty}\left\| \log\left\{{\frac{d\mu^k_{\alpha}}{d\mu_{\alpha}}(x)}\right\}\right\|_{\Lcal^2(\Hcal,\mu_{\alpha})} = 0
%\imply \lim_{k \approach \infty}\KL(\mu^k_{\alpha}||\mu_{\alpha}) = 0.
%\end{align*}
%\begin{align*}
%\\
%\log[p_0^{1-\alpha}(x)p_1^{\alpha}(x)] = \lim_{k \approach \infty}[(1-\alpha)\log(p_0^k(x)) + \alpha \log(p_1^k(x))]
%\\
%\log[(p_0P_1)^G_{\alpha}] = \lim_{k \approach \infty}\log{\frac{d\mu^k_{\alpha}}{d\mu_{*}}} = \log{\frac{d\mu_{\alpha}}{d\mu_{*}}}
%\end{align*}
%where the limit is in the $\Lcal^2(\Hcal,\mu_{*})$ sense.
	\qed
\end{proof}

\subsection{Proof of the exact geometric Jensen-Shannon divergence between equivalent Gaussian measures}
\label{section:proof-exact-JS-equivalent-Gaussian}

%By Proposition \ref{proposition:logdet2-I-S-finite}, all expressions of the form
%$\log\dettwo(I-S)$ for $S \in \SymHS(\Hcal)_{< I}$ are well-defined and finite.

Theorem \ref{theorem:KL-gaussian} admits the following equivalent form on $\Gauss(\Hcal, \mu_{*})$.
	\begin{proposition}
	\label{proposition:KL-divergence-equivalent-Gaussian-measures-mu-star}
	Let $\mu_i = \Ncal(m_i, C_i) \in \Gauss(\Hcal, \mu_{*})$, $i=0,1$, with $C_i= C_{*}^{1/2}(I-S_i)C_{*}^{1/2}$, $S_i \in \SymHS(\Hcal)_{< I}$. Then
	\begin{align}
		\KL(\mu_2||\mu_1) &= \frac{1}{2}||C_1^{-1/2}(m_2-m_1)||^2 -\frac{1}{2}\log\dettwo(C_1^{-1/2}C_2C_1^{-1/2})
		\\
		&=\frac{1}{2}||C_1^{-1/2}(m_2-m_1)||^2 - \frac{1}{2}\log\dettwo([(I-S_1)^{-1}(I-S_2)]).
	\end{align}
\end{proposition}

\begin{proof}
	[\textbf{of Proposition \ref{proposition:KL-divergence-equivalent-Gaussian-measures-mu-star}}]
	%	$C_1 = C_{*}^{1/2}(I-S_1)C_{*}^{1/2}$, $C_2 = C_{*}^{1/2}(I-S_2)C_{*}^{1/2}$.
	Since $\mu_1 \sim \mu_{*}$, $\mu_2 \sim \mu_{*}$, we have $\mu_1 \sim \mu_2$. Thus $\exists S \in \SymHS(\Hcal)_{< I}$ such that $C_2 = C_1^{1/2}(I-S)C_1^{1/2}$, i.e. the following operator is well-defined
	\begin{align*}
		I-S = C_1^{-1/2}C_2C_1^{-1/2}, \;\; S \in \SymHS(\Hcal)_{<I}.
	\end{align*}
	By Theorem \ref{theorem:KL-gaussian},
	\begin{align*}
		\KL(\mu_2||\mu_1)= \frac{1}{2}||C_1^{-1/2}(m_2-m_1)||^2 - \frac{1}{2}\log\dettwo(I-S).
	\end{align*}
	%	By symmetry, $\exists T_1 \in \SymHS(\Hcal)_{< I}$ such that $C_{*} = C_1^{1/2}(I-T_1)C_1^{1/2}$. Since $\ker(C_1) = \{0\}$, the operator $(I-T_1)^{1/2}C_1^{1/2}$ admits the polar 
	%decomposition 
	%\begin{align*}
	%(I-T_1)^{1/2}C_1^{1/2} = U (C_1^{1/2}(I-T_1)C_1^{1/2})^{1/2} = UC_{*}^{1/2},
	%\end{align*}
	%where $U \in \Ubb(\Hcal)$. It follows that, since $C_{*}$ is self-adjoint
	%	\begin{align*}
		%		C_{*}^{1/2} = U^{*}(I-T_1)^{1/2}C_1^{1/2} = C_1^{1/2}(I-T_1)^{1/2}U.
		%	\end{align*}
	%	Thus, with $C_2 = C_{*}^{1/2}(I-S_2)C_{*}^{1/2}$, we have
	%	\begin{align*}
		%	C_2 &= C_{*}^{1/2}(I-S_2)C_{*}^{1/2} = C_1^{1/2}(I-T_1)^{1/2}U(I-S_2)U^{*}(I-T_1)^{1/2}C_1^{1/2}
		%	\\
		%	& = C_1^{1/2}(I-T_1)^{1/2}(I-US_2U^{*})(I-T_1)^{1/2}C_1^{1/2}
		%	\end{align*}
	%
	%\begin{align*}
	%I-S = C_1^{-1/2}C_2C_1^{-1/2}
	%\end{align*}
	By assumption, $C_1 = C_{*}^{1/2}(I-S_1)C_{*}^{1/2}$. Since $\ker(C_{*}) = 0$, the operator 
	$(I-S_1)^{1/2}C_{*}^{1/2}$ admits the polar decomposition
	\begin{align*}
		(I-S_1)^{1/2}C_{*}^{1/2} = V(C_{*}^{1/2}(I-S_1)C_{*}^{1/2})^{1/2} = VC_{1}^{1/2},
	\end{align*}
	where $V \in \Ubb(\Hcal)$. Since $C_1 \in \Sym(\Hcal)$, it follows that
	\begin{align*}
		C_{1}^{1/2} = V^{*}(I-S_1)^{1/2}C_{*}^{1/2} = C_{*}^{1/2}(I-S_1)^{1/2}V.
	\end{align*}
	With $C_2 = C_{*}^{1/2}(I-S_2)C_{*}^{1/2}$, it follows that
	\begin{align*}
		I-S = C_1^{-1/2}C_2C_1^{-1/2} = V^{*}(I-S_1)^{-1/2}(I-S_2)(I-S_1)^{-1/2}V,
	\end{align*}
	as can be verified directly that $C_1^{1/2}(I-S)C_1^{1/2} = C_2$. In particular,
	\begin{align*}
	S &= I - V^{*}(I-S_1)^{-1/2}(I-S_2)(I-S_1)^{-1/2}V 
	\\
	&= V^{*}[I - (I-S_1)^{-1/2}(I-S_2)(I-S_1)^{-1/2}]V
	\\
	&= V^{*}[(I-S_1)^{-1/2}(S_2-S_1)(I-S_1)^{-1/2}]V \in \Sym(\Hcal) \cap \HS(\Hcal).
	\end{align*}
	Since $\dettwo(I-S)$, $S \in \HS(\Hcal) \cap \Sym(\Hcal)$, is completely determined by the eigenvalues of $S$, 
	it follows that, with $VV^{*} = V^{*}V = I$, 
	\begin{align*}
		\dettwo(I-S) &= \dettwo[(I-S_1)^{-1/2}(I-S_2)(I-S_1)^{-1/2}] 
		\\
		&= \dettwo([(I-S_1)^{-1}(I-S_2)]).
	\end{align*}
	\qed
\end{proof}

\begin{proof}[\textbf{of Theorem \ref{theorem:geometric-JS-equivalent-Gaussian-measures-Hilbert-space}}]
	%	\begin{align*}
		%		\JS^{G_{\alpha}}(\mu_0||\mu_1) = (1-\alpha)\KL(\mu_0||\mu_{\alpha}) + \alpha \KL(\mu_1||\mu_{\alpha})
		%	\end{align*}
	For $C_i = C_{*}^{1/2}(I-S_i)C_{*}^{1/2}$, $i=0,1$, $C_{\alpha} = C_{*}^{1/2}(I-S_{\alpha})C_{*}^{1/2}$,
	by Proposition \ref{proposition:KL-divergence-equivalent-Gaussian-measures-mu-star},
	\begin{align*}
		\KL(\mu_i||\mu_{\alpha}) = \frac{1}{2}||C_{\alpha}^{-1/2}(m_i- m_{\alpha})||^2 - \frac{1}{2}\log\dettwo[(I -S_{\alpha})^{-1}
		(I-S_i)].
	\end{align*}
	It follows then that
	\begin{align*}
		&\JS^{G_{\alpha}}(\mu_0||\mu_1) = (1-\alpha)\KL(\mu_0||\mu_{\alpha}) + \alpha \KL(\mu_1||\mu_{\alpha})
		\\
		& = \frac{(1-\alpha)}{2}||C_{\alpha}^{-1/2}(m_0- m_{\alpha})||^2 - \frac{1-\alpha}{2}\log\dettwo[(I -S_{\alpha})^{-1}
		(I-S_0)]
		\\
		&\quad +\frac{\alpha}{2}||C_{\alpha}^{-1/2}(m_1- m_{\alpha})||^2 -\frac{\alpha}{2}\log\dettwo[(I -S_{\alpha})^{-1}
		(I-S_1)].
	\end{align*}	
	%By Proposition \ref{proposition:geometric-interpolation-equivalent-Gaussian-Hilbert-space},
	%\begin{align*}
	%(I-S_{\alpha})^{-1} = (1-\alpha)(I-S_0)^{-1} + \alpha (I-S_1)^{-1}.
	%\end{align*}
	%It follows that
	%\begin{align*}
	%(I-S_{\alpha})^{-1}(I-S_0) &
	%= (1-\alpha) I + \alpha (I-S_1)^{-1}(I-S_0) 
	%\\
	%&= I + \alpha (I-S_1)^{-1}(S_1-S_0),
	%\\
	%(I-S_{\alpha})^{-1}(I-S_1) &= (1-\alpha)(I-S_0)^{-1}(I-S_1) + \alpha I
	%\\
	%& = I + (1-\alpha)(I-S_0)^{-1}(S_0-S_1).
	%\end{align*}
	Write $(I-S_{\alpha})^{-1}(I-S_i) = I + (I-S_{\alpha})^{-1}(S_{\alpha} - S_i)$, $i=0,1$.
	For $A \in \Tr(\Hcal)$, we have $\dettwo(I+A) = \det[(I+A)\exp(-A)]$, so that $\log\dettwo(I+A) = \log\det(I+A) - \trace(A)$.
	Thus, for $S_0, S_1 \in \Tr(\Hcal)$,
	\begin{align*}
		&\log\dettwo[(I-S_{\alpha})^{-1}(I-S_i)] 
		\\
		&= \log\det[(I-S_{\alpha})^{-1}(I-S_i)] - \trace[(I-S_{\alpha})^{-1}(S_{\alpha}-S_i)], \;\; i=0,1.
		%\\
		%&\log\dettwo[(I-S_{\alpha})^{-1}(I-S_1)]
		%\\
		%& = \log\det[(I-S_{\alpha})^{-1}(I-S_1)] - (1-\alpha)\trace[(I-S_0)^{-1}(S_0-S_1)]
	\end{align*}
	It follows that
	\begin{align*}
		&(1-\alpha)\log\dettwo[(I-S_{\alpha})^{-1}(I-S_0)] + \alpha \log\dettwo[(I-S_{\alpha})^{-1}(I-S_1)]
		\\
		& = \log\frac{\det(I-S_0)^{1-\alpha}\det(I-S_1)^{\alpha}}{\det(I-S_{\alpha})}- \trace[(I-S_{\alpha})^{-1}(S_{\alpha}-(1-\alpha)S_0- \alpha S_1)].
	\end{align*}
	\qed
\end{proof}

The following proof of Corollary \ref{corollary:geometric-JS-Gaussian-finite} is provided to verify
the correctness of Theorem \ref{theorem:geometric-JS-equivalent-Gaussian-measures-Hilbert-space} in the finite-dimensional setting.
\begin{proof}[\textbf{of Corollary \ref{corollary:geometric-JS-Gaussian-finite}}]
	Instead of computing $\JS^{G_{\alpha}}(\mu_1 ||\mu_0)$ directly using the KL divergence between Gaussian densities in $\R^n$, we derive it from the general formula in Eq.\eqref{equation:geometric-JS-Gaussian-Hilbert-trace-class-S} in Theorem
	\ref{theorem:geometric-JS-equivalent-Gaussian-measures-Hilbert-space}.
	From the formulas $C_i = C_{*}^{1/2}(I-S_i)C_{*}^{1/2}$, $i=0,1,\alpha$, we have $I-S_i = C_{*}^{-1/2}C_iC_{*}^{-1/2}$, $(I-S_i)^{-1} = C_{*}^{1/2}C_i^{-1}C_{*}^{1/2}$, $S_i = I - C_{*}^{-1/2}C_iC_{*}^{-1/2}$. 
	It follows that for the $\log\det$ term
	\begin{align*}
		\log\frac{\det(I-S_0)^{1-\alpha}\det(I-S_1)^{\alpha}}{\det(I-S_{\alpha})} = \log\frac{\det(C_0)^{1-\alpha}\det(C_1)^{\alpha}}{\det(C_{\alpha})}.
	\end{align*}
	Similarly, for the trace term,
	\begin{align*}
		&S_{\alpha} - (1-\alpha)S_0 - \alpha S_1 = C_{*}^{-1/2}[(1-\alpha)C_0 + \alpha C_1 - C_{\alpha}]C_{*}^{-1/2},
		\\
		&\trace[(I-S_{\alpha})^{-1}(S_{\alpha}-(1-\alpha)S_0- \alpha S_1)] = \trace[C_{\alpha}^{-1}((1-\alpha)C_0 + \alpha C_1) - I].
	\end{align*}
	Combining these two expressions with Eq.\eqref{equation:geometric-JS-Gaussian-Hilbert-trace-class-S} gives
	%\begin{align*}
	%&(1-\alpha)\log\dettwo[(I-S_{\alpha})^{-1}(I-S_0)] + \alpha \log\dettwo[(I-S_{\alpha})^{-1}(I-S_1)]
	%\\
	%& = \log\frac{\det(C_0)^{1-\alpha}\det(C_1)^{\alpha}}{\det(C_{\alpha})} 
	% -
	%\trace[C_{\alpha}^{-1}((1-\alpha)C_0 + \alpha C_1) - I].
	%\end{align*}
	%Thus it follows that
	\begin{align*}
		\JS_{G_{\alpha}}(\mu_0 ||\mu_1) &= \frac{(1-\alpha)}{2}||C_{\alpha}^{-1/2}(m_0- m_{\alpha})||^2 +\frac{\alpha}{2}||C_{\alpha}^{-1/2}(m_1- m_{\alpha})||^2
		\\
		&-\frac{1}{2}\log\frac{\det(C_0)^{1-\alpha}\det(C_1)^{\alpha}}{\det(C_{\alpha})} 
		+
		\frac{1}{2}\trace[C_{\alpha}^{-1}((1-\alpha)C_0 + \alpha C_1) - I].
	\end{align*}
	\qed
\end{proof}

\subsection{Proof of of the limiting behavior of the regularized geometric Jensen-Shannon divergence}
\label{section:proof-regularized}

\begin{proof}
	[\textbf{of Lemma \ref{lemma:C-alpha-gamma-form}}]
	It is clear by definition that $C_{\alpha,\gamma} \in \bP(\Hcal)$.
	We have the identity $(I+A)^{-1} = I- A(I+A)^{-1}$ for any $A \in \Lcal(\Hcal)$ 
	such that $I+A$ is invertible. Thus for $A+\gamma I$ invertible, $\gamma \in \R, \gamma \neq 0$,
	we have $(A+\gamma I)^{-1} = \frac{1}{\gamma}[I - A(A + \gamma I)^{-1}]$. Hence
	\begin{align*}
		&(1-\alpha)(C_0 + \gamma I)^{-1} + \alpha (C_1 + \gamma I)^{-1} 
		%\\
		%&= 
		%\frac{1-\alpha}{\gamma}\left(\frac{C_0}{\gamma} + I\right)^{-1} + \frac{\alpha}{\gamma}\left(\frac{C_1}{\gamma} + I\right)^{-1} 
		%\\
		%&= \frac{1-\alpha}{\gamma}
		%\left[I - \frac{C_0}{\gamma}\left(\frac{C_0}{\gamma} + I\right)^{-1}\right]
		%+\frac{\alpha}{\gamma}\left[I - \frac{C_1}{\gamma}\left(\frac{C_1}{\gamma} + I\right)^{-1}\right]
		\\
		& = \frac{1}{\gamma}\left[I - (1-\alpha)C_0(C_0+\gamma I)^{-1} - \alpha C_1(C_1+\gamma I)^{-1}\right] = \frac{1}{\gamma}(I-B),
	\end{align*}
	where $B = (1-\alpha)C_0(C_0+\gamma I)^{-1} +\alpha C_1(C_1+\gamma I)^{-1} \in \SymTr(\Hcal)_{<I}$.
	Thus
	\begin{align*}
		C_{\alpha, \gamma} &= \gamma (I-B)^{-1} =\gamma [I + (I-B)^{-1/2}B(I-B)^{-1/2}] 
		\\
		&= \gamma I + \gamma (I-B)^{-1/2}B(I-B)^{-1/2} 
		% \gamma [I + B(I-B)^{-1}] = \gamma I + \gamma B(I-B)^{-1} 
		= \gamma I + A > 0,
	\end{align*}
	with $A = \gamma(I-B)^{-1/2}B(I-B)^{-1/2} \in \Sym(\Hcal) \cap \Tr(\Hcal)$.
	It follows that $C_{\alpha,\gamma} \in \PC_1(\Hcal)$.
	\qed
\end{proof}

%\begin{align}
%d^1_{\logdet}[(C_0 + \gamma I), C_{\alpha,\gamma}].
%\end{align}
\begin{proposition}
	[\cite{Minh:2020regularizedDiv}]
	\label{proposition:limit-product-A-plus-gamma-B}
	Let $A$ be a compact, self-adjoint, strictly positive operator on $\Hcal$. Let $B \in \HS(\Hcal)$. Then
	\begin{align}
		\lim_{\gamma \approach 0^{+}}||(A+\gamma I)^{-1/2}A^{1/2}BA^{1/2}(A+\gamma I)^{-1/2} - B||_{\HS} = 0.
	\end{align}
\end{proposition}

\begin{lemma}
	\label{lemma:Hilbert-Carleman-det2-modified}
	Let $S \in \SymHS(\Hcal)_{< I}$ be fixed. Then $(1-\alpha)S \in \SymHS(\Hcal)_{< I}$ and consequently $\log\dettwo[I-(1-\alpha)S]$ is finite and $\dettwo[I-(1-\alpha)S] > 0$ $\forall \alpha \in [0,1]$.
\end{lemma}
\begin{proof}
	Let $\{\lambda_k\}_{k=1}^{\infty}$ be the eigenvalues of $S$, then $1-\lambda_k > 0$ $\forall k \in \Nbb$ and $\lim_{k \approach \infty}\lambda_k = 0$.
	Let $\lambda_{\max}$ be the largest eigenvalue of $S$.
	If $\lambda_{\max} \geq 0$, then for $0 \leq \alpha \leq 1$, we have $1-(1-\alpha)\lambda_k \geq 1-(1-\alpha)\lambda_{\max} \geq 1-\lambda_{\max} > 0$.
	If $\lambda_{\max} < 0$, then obviously $1-(1-\alpha)\lambda_k \geq 1$ $\forall k \in \Nbb$. Thus $I-(1-\alpha)S > 0$, i.e. $(1-\alpha)S \in \SymHS(\Hcal)_{< I}$. %$1-(1-\alpha)\lambda_k > 0$ $\forall k \in \Nbb$.
%	If $\lambda_k \geq 0$, then for $0 \leq \alpha \leq 1$, we have $1-(1-\alpha)\lambda_k \geq 1-\lambda_k > 0$.
%	If $\lambda_k < 0$, then obviously $1-(1-\alpha)\lambda_k > 0$. Thus $1-(1-\alpha)\lambda_k > 0$ $\forall k \in \Nbb$.
	Thus, by Proposition \ref{proposition:logdet2-I-S-finite}, $\log\dettwo[I-(1-\alpha)S]$ is finite and $\dettwo[I-(1-\alpha)S] > 0$ $\forall \alpha \in [0,1]$.
	%By definition
	%$\dettwo[I-(1-\alpha)S] = \det[(I-(1-\alpha)S)\exp((1-\alpha)S)] = \prod_{k=1}^{\infty}(1-(1-\alpha)\lambda_k) \exp((1-\alpha)\lambda_k) > 0$.
	\qed
\end{proof}

\begin{proposition}
	\label{proposition:logdet-C0-C-alpha-gammma-0}
	Assume the hypothesis of Theorem \ref{theorem:geometric-interpolation-equivalent-Gaussian-Hilbert-space}.
	For $i=0,1$
	\begin{align}
		\label{equation:logdet-1-limit-gamma-0}
		\lim_{\gamma \approach 0^{+}}d^1_{\logdet}[(C_i + \gamma I), C_{\alpha,\gamma}]= -\log\dettwo[(I-S_{\alpha})^{-1}(I-S_i)],
	\end{align}
	where $S_{\alpha}$ is as given in Theorem \ref{theorem:geometric-interpolation-equivalent-Gaussian-Hilbert-space}.
\end{proposition}
\begin{proof}
	[\textbf{of Proposition \ref{proposition:logdet-C0-C-alpha-gammma-0}}]
	Let us prove Eq.\eqref{equation:logdet-1-limit-gamma-0} for $i=1$.
	Write $C_{\alpha,\gamma } = A+ \gamma I$,
	with $A \in \Sym(\Hcal) \cap \Tr(\Hcal)$ as in Lemma \ref{lemma:C-alpha-gamma-form}. 
	Then $(A+\gamma I)^{-1} = [(1-\alpha)(C_0 + \gamma I)^{-1} + \alpha (C_1 + \gamma I)^{-1}]$ and
	\begin{align*}
		&(A+\gamma I)^{-1}(C_1 + \gamma I) = [(1-\alpha)(C_0 + \gamma I)^{-1} + \alpha (C_1 + \gamma I)^{-1}](C_1+\gamma I)
		\\
		& = \alpha I + (1-\alpha) (C_0 +\gamma I)^{-1}(C_1+\gamma I)
		\\
		& = I + (1-\alpha)[(C_0 +\gamma I)^{-1}(C_1+\gamma I)-I].
	\end{align*}
	By definition 
	of $d^1_{\logdet}$,
	\begin{align*}
		&d^1_{\logdet}[(C_0 + \gamma I), C_{\alpha,\gamma}] = d^1_{\logdet}[(C_0 + \gamma), (A+\gamma I)]
		\\
		& = \traceX[(A+\gamma I)^{-1}(C_0 + \gamma I) - I] - \log\detX[(A+\gamma I)^{-1}(C_0 + \gamma I)]
		\\
		& = \trace[(1-\alpha)[(C_0 +\gamma I)^{-1}(C_1+\gamma I)-I]] 
		\\
		& \quad - \log\det[I + (1-\alpha)[(C_0 +\gamma I)^{-1}(C_1+\gamma I)-I]]
		\\
		&  = -\log\dettwo(I + (1-\alpha)[(C_0+\gamma I)^{-1}(C_1 + \gamma I) - I]).
	\end{align*}
	Since $\mu_0 \sim \mu_1$, $\exists S \in \SymHS(\Hcal)_{< I}$ such that $C_1 = C_0^{1/2}(I-S)C_0^{1/2}$.
	With $C_1 = C_0^{1/2}(I-S)C_0^{1/2} = C_0 - C_0^{1/2}SC_0^{1/2}$, for any $\gamma \in \R, \gamma > 0$, we have $C_1 +\gamma I = C_0+\gamma I - C_0^{1/2}SC_0^{1/2}$ and consequently,
	\begin{align*}
		&(C_0 + \gamma I)^{-1/2}(C_1 + \gamma I)(C_0 + \gamma I)^{-1/2} 
		\\
		&= I - (C_0 + \gamma I)^{-1/2}C_0^{1/2}SC_0^{1/2}(C_0 + \gamma I)^{-1/2}.
	\end{align*}
	By Proposition \ref{proposition:limit-product-A-plus-gamma-B},
	\begin{align*}
		\lim_{\gamma \approach 0^{+}}||(C_0 + \gamma I)^{-1/2}C_0^{1/2}SC_0^{1/2}(C_0 + \gamma I)^{-1/2} - S||_{\HS} = 0.
	\end{align*}
	By the continuity of the Hilbert-Carleman determinant $\dettwo$ in the $\HS$ norm
	\begin{align*}
		&\lim_{\gamma \approach 0^{+}}\log\dettwo(I + (1-\alpha)[(C_0+\gamma I)^{-1/2}(C_1 + \gamma I)(C_0+\gamma I)^{-1/2} - I])
		\\
		&=\log\dettwo[I-(1-\alpha)S],
	\end{align*}
	which is well-defined by Lemma \ref{lemma:Hilbert-Carleman-det2-modified}. Let us express this in terms of the operators $S_{\alpha}, S_0, S_1 \in \SymHS(\Hcal)_{< I}$ as defined in Theorem
	\ref{theorem:geometric-interpolation-equivalent-Gaussian-Hilbert-space}. We have
	\begin{align*}
		(I-S_{\alpha})^{-1} &= (1-\alpha)(I-S_0)^{-1} + \alpha (I-S_1)^{-1},
		\\
		%\end{align*}
		%\begin{align*}
		(I-S_{\alpha})^{-1}(I-S_1) &= (1-\alpha)(I-S_0)^{-1}(I-S_1) + \alpha I.
	\end{align*}
	As in the proof of Proposition \ref{proposition:KL-divergence-equivalent-Gaussian-measures-mu-star}, 
	consider the polar decomposition $(I-S_0)^{1/2}C_{*}^{1/2} = V(C_{*}^{1/2}(I-S_0)C_{*}^{1/2})^{1/2} = VC_0^{1/2}$
	with $V \in \Ubb(\Hcal)$. Then
	\begin{align*}
		I-S &= V^{*}(I-S_0)^{-1/2}(I-S_1)(I-S_0)^{-1/2}V
		\\
		\equivalent S &= I- V^{*}(I-S_0)^{-1/2}(I-S_1)(I-S_0)^{-1/2}V 
		\\
		&= V^{*}[I - (I-S_0)^{-1/2}(I-S_1)(I-S_0)^{-1/2}]V.
	\end{align*}
	With $A \in \Sym(\Hcal) \cap \HS(\Hcal)$, $\dettwo(I+A)$ is completely determined by the eigenvalues of $A$,
	which are the same as those of $V^{*}AV$ for $V \in \Ubb(\Hcal)$. Thus
	\begin{align*}
		&\log\dettwo[I - (1-\alpha)S] 
		\\
		& = \log\dettwo[I - (1-\alpha)V^{*}[I-(I-S_0)^{-1/2}(I-S_1)(I-S_0)^{-1/2}]V]
		\\
		&= \log\dettwo[I - (1-\alpha)[I-(I-S_0)^{-1/2}(I-S_1)(I-S_0)^{-1/2}]]
		\\
		& = \log\dettwo[I - (1-\alpha)[I-(I-S_0)^{-1}(I-S_1)]]
		\\
		& = \log\dettwo[\alpha I + (1-\alpha)(I-S_0)^{-1}(I-S_1)]
		\\
		& = \log\dettwo[(I-S_{\alpha})^{-1}(I-S_1)]
	\end{align*}
	It follows that $\lim_{\gamma \approach 0^{+}}d^1_{\logdet}[(C_1 + \gamma I), C_{\alpha,\gamma}]= -\log\dettwo[(I-S_{\alpha})^{-1}(I-S_1)]$.
	Similarly, $\lim_{\gamma \approach 0^{+}}d^1_{\logdet}[(C_0 + \gamma I), C_{\alpha,\gamma}]= -\log\dettwo[(I-S_{\alpha})^{-1}(I-S_0)]$.
	\qed
\end{proof}

\begin{lemma}
	\label{lemma:equivalent-Gaussian-covariance-range}
	Let $\Ncal(m,C)$, $\Ncal(m_0, C_0)$ be two equivalent Gaussian measures on $\Hcal$. Then
	$\range(C^{1/2}) = \range(C_0^{1/2})$.
\end{lemma}
\begin{proof}
	By Theorem 2.1 in \cite{Fillmore:1971Operator}, for any bounded operator $A \in \Lcal(\Hcal)$, we have $\range(A) = \range(AA^{*})^{1/2}$.
	With $\Ncal(m,C) \sim \Ncal(m_0,C_0)$, $\exists S \in \SymHS(\Hcal)_{< I}$ such that
	$C= C_0^{1/2}(I-S)C_0^{1/2}$. Let $A = C_0^{1/2}(I-S)^{1/2}$, then 
	\begin{align*}
		\range(C^{1/2}) &= \range(AA^{*})^{1/2} = \range(A) = \range[C_0^{1/2}(I-S)^{1/2}] 
		\\
		&= \range(C_0^{1/2}).
	\end{align*} 
	\qed
\end{proof}

%\begin{proposition}
%	\label{proposition:limit-quadratic-form}
%Let $0 \leq r \leq 1$ be fixed. For $m_0, m \in \Hcal$ and two compact, self-adjoint, strictly positive operators $C_0, C$ on $\Hcal$,
%\begin{align}
%&\lim_{\gamma \approach 0^{+}}\la m- m_0, [(1-r)(C+\gamma I) + r(C_0+\gamma I)]^{-1}(m-m_0)\ra
%\\
%&=
% \left\{
%\begin{matrix}
%||[(1-r)C + rC_0]^{-1/2}(m-m_0)||^2 & \text{ if } m-m_0 \in \range[(1-r)C + rC_0]^{1/2},
%\\
%\infty & \text{otherwise}.
%\end{matrix}
%\right.
%\end{align}
%\end{proposition}

\begin{lemma}
	[\cite{Minh:2020regularizedDiv}]
	\label{lemma:limit-quadratic-form}
	Let $A$ be a self-adjoint, compact, strictly positive operator on $\Hcal$. Then
	\begin{align}
		\lim_{\gamma \approach 0^{+}}\la x, (A+\gamma I)^{-1}x\ra 
		= \left\{
		\begin{matrix}
			||A^{-1/2}x||^2 & \text{ if } x \in \range(A^{1/2}),
			\\
			\infty & \text{otherwise}.
		\end{matrix}
		\right.
	\end{align}
\end{lemma}

\begin{lemma}
	\label{lemma:equivalent-norm-C-m}
	Let $\Ncal(m_0,C_0)$, $\Ncal(m,C)$, $\ker(C) = \ker(C_0) = \{0\}$, be two equivalent Gaussian measures on $\Hcal$, with
	$C = C_0^{1/2}(I-S)C_0^{1/2}$, where $S \in \SymHS(\Hcal)_{< I}$. Then
	\begin{align}
		||C^{-1/2}x|| = ||(I-S)^{-1/2}C_0^{-1/2}x||, \;\; \forall x \in \range(C^{1/2}) = \range(C_0^{1/2}).
	\end{align}
\end{lemma}
\begin{proof}
	%	For $C = C_0^{1/2}(I-S)C_0^{1/2}$, for $m \in \range(C_0^{1/2}) = \range(C^{1/2})$
	%	\begin{align*}
		%		||C^{-1/2}m||^2 = ||[C_0^{1/2}(I-S)C_0^{1/2}]^{-1/2}m||^2 = 
		%	\end{align*}
	We have $\range(C^{1/2}) = \range(C_0^{1/2})$ by Lemma \ref{lemma:equivalent-Gaussian-covariance-range}.
	Since $\ker(C_0) = 0$, the operator $(I-S)^{1/2}C_0^{1/2}$ admits the following polar decomposition
	\begin{align*}
		(I-S)^{1/2}C_0^{1/2} = V[C_0^{1/2}(I-S)C_0^{1/2}]^{1/2},
	\end{align*}
	with $V \in \Ubb(\Hcal)$, from which it follows that
	\begin{align*}
		C^{1/2} = [C_0^{1/2}(I-S)C_0^{1/2}]^{1/2} = V^{*}(I-S)^{1/2}C_0^{1/2} = C_0^{1/2}(I-S)^{1/2}V.
	\end{align*}
	It follows that for  $x \in \range(C_0^{1/2}) = \range(C^{1/2})$,
	\begin{align*}
	C^{-1/2}x = y =  V^{*}(I-S)^{-1/2}C_0^{-1/2}x,
	\end{align*}
	as can be verified directly that $C^{1/2}y = x$. Thus
	\begin{align*}
		||C^{-1/2}x|| = [C_0^{1/2}(I-S)C_0^{1/2}]^{-1/2}x|| &= ||V^{*}(I-S)^{-1/2}C_0^{-1/2}x||
		\\
		& = ||(I-S)^{-1/2}C_0^{-1/2}x||.
	\end{align*}
	\qed
\end{proof}

\begin{proposition}
	\label{proposition:limit-regularized-JS-mean-term}
	Let $\Ncal(m_i,C_i) \in \Gauss(\Hcal, \mu_{*})$. Let
	$m_{\alpha}\in \Hcal, C_{\alpha}\in \Sym(\Hcal)\cap\Tr(\Hcal)$ be defined as in Theorem \ref{theorem:geometric-interpolation-equivalent-Gaussian-Hilbert-space}.
	Let $C_{\alpha,\gamma}$ be defined as in Eq.\eqref{equation:C-alpha-gamma}. Then
	%for $i=0,1$,
	\begin{align}
		\lim_{\gamma \approach 0^{+}}||C_{\alpha,\gamma}^{-1/2}(m_i - m_{\alpha})|| = ||C_{\alpha}^{-1/2}(m_i - m_{\alpha})||, \;\; i=0,1.
	\end{align}
\end{proposition}
\begin{proof}
	[\textbf{of Proposition \ref{proposition:limit-regularized-JS-mean-term}}]
	It suffices to prove for $i=0$.
	With $C_{\alpha,\gamma}=	[(1-\alpha)(C_0 + \gamma I)^{-1} + \alpha (C_1 + \gamma I)^{-1}]^{-1}$, by 
	Lemma \ref{lemma:limit-quadratic-form} we have
	%	With $C_i = C_{*}^{1/2}(I-S_i)C_{*}^{1/2}$, we have
	%\begin{align*}
	%(1-\alpha)C_0 + \alpha C_1 = C_{*}^{1/2}[(1-\alpha)S_0 + \alpha S_1]C_{*}^{1/2}
	%\end{align*}
	\begin{align*}
		&\lim_{\gamma \approach 0^{+}}||C_{\alpha,\gamma}^{-1/2}(m_0 - m_{\alpha})||^2 
		= \lim_{\gamma \approach 0^{+}}\la m_0 - m_{\alpha}, C_{\alpha,\gamma}^{-1}(m_0 - m_{\alpha})\ra
		\\
		& = \lim_{\gamma \approach 0^{+}}\la m_0 - m_{\alpha}, [(1-\alpha)(C_0 + \gamma I)^{-1} + \alpha (C_1 + \gamma I)^{-1}](m_0 - m_{\alpha})\ra
		\\
		& = (1-\alpha)||C_0^{-1/2}(m_0 - m_{\alpha})||^2 + \alpha ||C_1^{-1/2}(m_0-m_{\alpha})||^2.
	\end{align*}
	With $C_i = C_{*}^{1/2}(I-S_i)C_{*}^{1/2}$, $i=0,1$, by Lemma \ref{lemma:equivalent-norm-C-m},
	\begin{align*}
		||C_i^{-1/2}(m_0 - m_{\alpha})||^2 &= ||(I-S_i)^{-1/2}C_{*}^{-1/2}(m_0 - m_{\alpha})||^2 
		\\
		& = \la C_{*}^{-1/2}(m_0 - m_{\alpha}), (I-S_i)^{-1}C_{*}^{-1/2}(m_0 - m_{\alpha})\ra.
	\end{align*}
	Combining this with the previous expression gives
	\begin{align*}
		&\lim_{\gamma \approach 0^{+}}||C_{\alpha,\gamma}^{-1/2}(m_0 - m_{\alpha})||^2 
		\\
		&= 
		\la C_{*}^{-1/2}(m_0 - m_{\alpha}), [(1-\alpha)(I-S_0)^{-1} + \alpha (I-S_1)^{-1}]C_{*}^{-1/2}(m_0-m_{\alpha})\ra
		\\
		& = \la C_{*}^{-1/2}(m_0 - m_{\alpha}), (I-S_{\alpha})^{-1}C_{*}^{-1/2}(m_0 - m_{\alpha})\ra
		\\
		& = ||(I-S_{\alpha})^{-1/2}C_{*}^{-1/2}(m_0 - m_{\alpha})||^2 = ||C_{\alpha}^{-1/2}(m_0 - m_{\alpha})||^2,
	\end{align*}
	where the last equality follows from Lemma \ref{lemma:equivalent-norm-C-m}.
	\qed
\end{proof}

\begin{proof}
	[\textbf{of Theorem \ref{theorem:limit-regularized-JS-Gaussian}}]
	This follows by combining Propositions \ref{proposition:logdet-C0-C-alpha-gammma-0} and \ref{proposition:limit-regularized-JS-mean-term} with Theorem \ref{theorem:geometric-JS-equivalent-Gaussian-measures-Hilbert-space}.
	\qed
\end{proof}

%\begin{lemma}
%\end{lemma}
%\begin{proof}
%	\begin{align*}
%		||Am_1||^2 - ||Am_2||^2
%	\end{align*}
%\end{proof}

{\bf Competing interests}. The author declares that there are no competing interests, financial or otherwise.

\bibliographystyle{plain}
%\bibliography{/Users/Minhs/Dropbox/cite_RKHS}
%\bibliography{/Users/Minh/Dropbox/cite_RKHS}
\bibliography{../cite_RKHS}
\end{document}